\documentclass{article}

\usepackage{amsmath,amsfonts,amssymb,amsthm}
\usepackage{hyperref}
\usepackage{fullpage}
\usepackage{color}
\usepackage{tikz}

\newcommand{\q}[1]{( {#1} )_q}
\newcommand{\Q}[1]{ [ {#1} ]_q }
\newcommand{\id}[1]{\mathrm{id}\vert_{ {#1} } }
\newcommand{\binomq}[2]{ \binom{ {#1}}{{#2}}_q}
\newcommand{\binomqq}[2]{ \binom{ {#1}}{{#2}}_{q^{2}}}

\newcommand{\binomQ}[2]{ \left[ \begin{array}{c} {#1} \\ {#2} \end{array}\right]_q}

\newcommand{\Gr}{\mathrm{Gr}}
\newcommand{\Ga}{\mathrm{Ga}}

\newcommand{\ot}{\otimes}
\newcommand{\DOI}[1]{\href{http://dx.doi.org/#1}{DOI:#1}}
\newcommand{\arXiv}[1]{\href{http://arxiv.org/abs/#1}{arXiv:#1}}

\newtheorem{theorem}{Theorem}[section]
\newtheorem{proposition}[theorem]{Proposition}
\newtheorem{lemma}[theorem]{Lemma}
\newtheorem{corollary}[theorem]{Corollary}

\theoremstyle{definition}
\newtheorem{definition}[theorem]{Definition}
\newtheorem{remark}[theorem]{Remark}

\begin{document}

\title{An algebraic construction of duality functions for the stochastic $\mathcal{U}_q( A_n^{(1)})$ vertex model and its degenerations} \author{Jeffrey Kuan}

\date{}

\maketitle

\abstract{A recent paper \cite{KMMO} introduced the stochastic $\mathcal{U}_q(A_n^{(1)})$ vertex model. The stochastic $S$--matrix is related to the $R$--matrix of the quantum group $\mathcal{U}_q(A_n^{(1)})$ by a gauge transformation. We will show that a certain function $D^+_{\mu}$ intertwines with the transfer matrix and its space reversal. When interpreting the transfer matrix as the transition matrix of a discrete--time totally asymmetric particle system on the one--dimensional lattice $\mathbb{Z}$, the function $D^+_{\mu}$ becomes a Markov duality function $D_{\mu}$ which only depends on $q$ and the vertical spin parameters $\{\mu_x\}$. By considering degenerations in the spectral parameter, the duality results also hold on a finite lattice with closed boundary conditions, and for a continuous--time degeneration. This duality function had previously appeared in a multi--species $\textrm{ASEP}(q,j)$ process \cite{K}.  The proof here uses that the $R$--matrix intertwines with the co--product, but does not explicitly use the Yang--Baxter equation. 

It will also be shown that the stochastic $\mathcal{U}_q(A_n^{(1)})$ is a multi--species version of a stochastic vertex model studied in \cite{BP,CP}. This will be done by generalizing the fusion process of \cite{CP} and showing that it matches the fusion of \cite{KRL} up to the gauge transformation. 

We also show, by direct computation, that the multi--species $q$--Hahn Boson process (which arises at a special value of the spectral parameter) also satisfies duality with respect to $D_0$, generalizing the single--species result of \cite{C}.

}

\tableofcontents

\section{Introduction}
Over the last 25 years, much work has been done investigating interacting particle systems with a property called \textit{stochastic duality} (e.g. \cite{BC,BS,BS2,BS3,BorCor13,BCS,CGRS,CGRS2,C,GKRV,GRV,K2,K,O,S,S2,SS}). Duality has been shown to be useful for asymptotics \cite{BCS}, weak convergence \cite{CST}, and shock dynamics \cite{BS3}. The first use of duality in interacting particle systems actually goes back even farther, to 1970 \cite{Sp}. See also Chapter III of \cite{Ligg} for an exposition and references. 

A more recent direction of research has been to find dualities in multi--species versions of some of these systems. This was done in \cite{BS,BS2,BS3,K2,K}. In these cases, the interacting particle system satisfied a Lie algebra symmetry, and the rank of the Lie algebra corresponded to the number of species of particles. 

Duality has also been discovered in stochastic vertex models (see e.g. \cite{BCG,BP,CP}). These vertex models enjoy the property of degenerating to a large class of other probabilistic models, including some of the ones above. However, the proofs of these dualities seem to be ad--hoc, in the sense that they required knowing the duality function beforehand, and did not involve constructing the duality function using the algebraic symmetry of the model. Furthermore, there were no known examples of multi--species vertex models satisfying duality.

Thus, it is natural to look for a stochastic vertex model such that

 (1) self--duality holds with respect to a duality function which can be defined with the algebraic symmetry of the vertex model, and

 (2) in certain degenerations of the vertex model, previous duality results can be recovered, and
 
 (3) is a multi--species generalization of an existing single--species model.

The purpose of this paper is to prove that the stochastic $\mathcal{U}_q( A_n^{(1)})$ vertex model of \cite{KMMO} satisfies all three of these properties. 

The stochastic $\mathcal{U}_q( A_n^{(1)})$ vertex model is defined from a stochastic matrix $S(z)$ depending on a spectral parameter $z$; see section \ref{KMMOBACK} for definitions. The matrix $S(z)$ is obtained from the (non--stochastic) $R$--matrix of $\mathcal{U}_q( A_n^{(1)})$ by a gauge transform. The action of $S(z)$ on a certain representation $V_l^{z_1} \otimes V_m^{z_2}$ then defines a local Markov operator. Here, $l$ is the horizontal spin parameter, $m$ is the vertical spin parameter, and the spectral parameters $z_1$ and $z_2$ satisfy $z=z_1/z_2$. The corresponding transfer matrix then defines a Markov operator for an interacting particle system on a one--dimensional lattice. 

The original paper \cite{KMMO} specifically considers $z=q^{l-m}$ and finds a $n$--species version of the $q$--Hahn Boson process introduced in \cite{Po}. The $q$--Hahn Boson process has both a discrete--time and a continuous--time definition, with the continuous--time process satisfying a Hecke algebra symmetry, as shown in \cite{T0}. When the vertical spin parameter $\mu=q^{-m}$ converges to $0$, this further degenerates to a single species $q$--Boson process introduced by \cite{SW}. The paper here will consider general values of $z$. 

The main results will be summarized as follows.

(1) We will show that the stochastic $\mathcal{U}_q( A_n^{(1)})$ vertex model (for generic values of $z$) satisfies self--duality with respect to a certain explicit duality function $D_{\mu}$. The function $D_{\mu}$ had previously appeared as the duality function of a multi--species $\textrm{ASEP}(q,j)$, and is defined from the action of a certain element $u_0\in \mathcal{U}_q\left(A_n\right).$ It only depends on the vertical spin parameters $\mu_x$ and not on the horizontal spin parameter. The proof uses that the $R$--matrix intertwines with the action of $\mathcal{U}_q( A_n^{(1)})$, but does not explicitly use the Yang--Baxter equation. The proof also involves showing that the gauge is the same as the ``ground state transformation'' of \cite{K} up to a diagonal change of basis. 

(2) For a range of values of the spectral parameter $z$, the matrix $S(z)$ is stochastic (see Proposition \ref{Stoch}). In section \ref{BBB}, some degenerations of $S(z)$ will be considered. In particular, for degenerations of $z$, the matrix $S(z)$ becomes trivial (see Theorem \ref{Stuff}). This actually allows for construction of the particle system in continuous time as well as on a finite lattice with closed boundary (see section \ref{Iaap}). A noteworthy result is that the $n$--species $q$--Boson process introduced by \cite{T} can be shown to satisfy self--duality, which had previously been shown with different methods in \cite{K}. Another interesting case is when $z=q^{l-m}$, which shows that the $n$--species discrete--time $q$--Hahn Boson process is self--dual with respect to $D_{\mu}$. See Figures \ref{Deg} and \ref{Deg2} for the degenerations discussed here.

It also turns out that the $n$--species discrete--time $q$--Hahn Boson process is also self--dual with respect to $D_0$, which was shown for $n=1$ in \cite{C}. However, it is not clear how to prove this algebraically. A direct proof will be given.

(3) It will be shown that the stochastic $\mathcal{U}_q(A_n^{(1)})$ vertex model is a $n$--species generalization of the models of \cite{CP,BP}. This will be done by showing that the ``stochastic fusion'' procedure of \cite{CP} can be generalized for multiple species (see Theorem \ref{StochasticFusion}).  Additionally, the Markov projection to the first $k$ species is a $\mathcal{U}_q(A_k^{(1)})$ vertex model (see Proposition \ref{JJJ}). The latter property is typical for multi--species models (see e.g. \cite{K}).

It is also worth explicitly mentioning the role of the boundary conditions in the duality results. The transfer matrix of the stochastic $\mathcal{U}_q(A_n^{(1)})$ vertex model intertwines with its space reversal under a certain function $D^+_{\mu}$ which also acts on the auxiliary space. In order to reduce $D^+_{\mu}$ to a duality functional $D_{\mu}$ which does not act on the auxiliary space, a certain cancellation is needed (Lemma \ref{Reduction}), but this reduction does not seem to hold for open or periodic boundary conditions.

The remainder of the paper is outlined as such: Section \ref{PR} states the necessary definitions, notations and results from previous papers. {In \color{black} section \ref{Further}, further properties of $S(z)$ will be proved, including ranges for which the matrix is stochastic (Proposition \ref{Stochastic}), degenerations (section \ref{BBB}), Markov projections (Theorem \ref{JJJ}), and stochastic fusion (Theorem \ref{StochasticFusion}).} 

Section \ref{ALGDUA} defines the transfer matrix and proves duality results for the resulting particle systems. The main theorem is stated in Theorem \ref{Major}, which shows that $D_{\mu}^+$ intertwines between the transfer matrix and its space reversal. This results in duality results for a discrete--time particle system on the infinite line (Theorem \ref{DualThm}), on a finite lattice with closed boundary conditions (Proposition \ref{412}), and for a continuous--time degeneration (Proposition \ref{AAA2}). Section \ref{Dop} describes the processes that can be obtained from the various degenerations: subsection \ref{ZZZ} considers the multi--species $q$--Hahn Boson process, and section \ref{PPP} considers the case when $l=1$.  Section \ref{DIRECTDUA} shows, using direct computation,  the self--duality of multi--species $q$--Hahn Boson with respect to $D_0$, as well as the Markov projection property. 

\textbf{Acknowledgments}. The author would like to thank Alexei Borodin and Ivan Corwin for helpful conversations. Financial support was provided by the Minerva Foundation and NSF grant DMS--1502665.

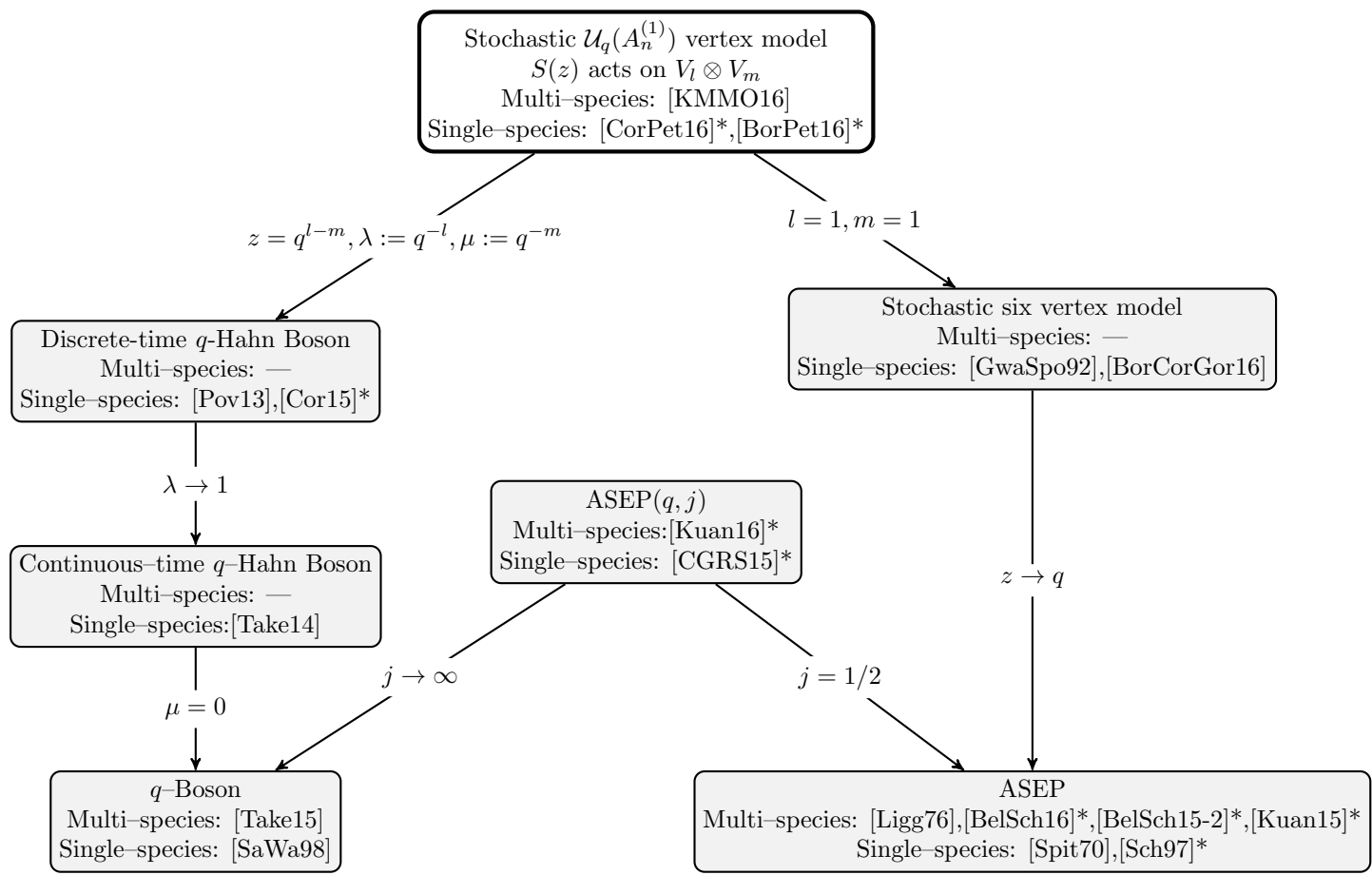
\begin{figure}

\begin{tikzpicture}[scale=0.9, every text node part/.style={align=center}]
\usetikzlibrary{arrows}
\usetikzlibrary{shapes}
\usetikzlibrary{shapes.multipart}
\tikzstyle{process}=[rectangle ,draw,thick ,rounded corners=4pt,fill=black!5 ]
\tikzstyle{AEP}=[rectangle,draw,ultra thick, rounded corners=6pt ]
\tikzstyle{qHahn}=[rectangle,draw, ultra thick, rounded corners=4pt, gray, text=black]
\tikzstyle{limit}=[->,>=stealth',thick,rounded corners=4pt]
\tikzstyle{relation}=[ ultra thick, dashed, rounded corners=4pt, gray]

\node[process]  (cqHahn) at (-4,-3.5) {Discrete-time $q$-Hahn Boson \\ Multi--species: --- \\ Single--species: \cite{Po},\cite{C}*};
\node[process]  (qTASEP) at (-4,-7) {Continuous--time $q$--Hahn Boson \\ Multi--species: ---  \\ Single--species:\cite{T0}};
\node[process]  (TASEP) at (-4,-10.5) {$q$--Boson \\ Multi--species: \cite{T} \\ Single--species: \cite{SW}};
\node[process]  (MADM) at (3,-6) {ASEP$(q,j)$ \\ Multi--species:\cite{K}* \\Single--species: \cite{CGRS}*};
\node[AEP]  (qHahn) at (3,1) {Stochastic $\mathcal{U}_q(A_n^{(1)})$ vertex model \\ $S(z)$ acts on $V_l \otimes V_m$ \\ Multi--species: \cite{KMMO} \\ Single--species: \cite{CP}*,\cite{BP}*};
\node[process] (SixVer) at (9,-3) {Stochastic six vertex model \\ Multi--species: --- \\ Single--species: \cite{GS92},\cite{BCG} };
\node[process]  (ASEP) at (9,-10.5) {ASEP \\ Multi--species: \cite{L},\cite{BS3}*,\cite{BS2}*,\cite{K2}* \\ Single--species: \cite{Sp},\cite{S}*} ;

\draw[limit] (qHahn) -- (cqHahn) node[midway,fill=white]{$z=q^{l-m}, \lambda := q^{-l}, \mu:=q^{-m}$};
\draw[limit] (cqHahn) -- (qTASEP) node[midway,fill=white]{$\lambda\rightarrow 1$};
\draw[limit] (qTASEP) -- (TASEP) node[midway,fill=white]{$\mu=0$};
\draw[limit] (MADM) -- (ASEP) node[midway,fill=white]{$j=1/2$};
\draw[limit] (SixVer) -- (ASEP) node[midway,fill=white]{$z \rightarrow q$};
\draw[limit] (qHahn) -- (SixVer) node[midway,fill=white]{$l=1,m=1$};
\draw[limit] (MADM) -- (TASEP) node[midway,fill=white]{$j\rightarrow \infty$};
\end{tikzpicture}
\caption{The various degenerations and limits explicitly mentioned in this paper. A * indicates a paper with a duality result. See Figure 1 of \cite{CP} for a more complete diagram. Additional degenerations will be shown in Figure \ref{Deg2}.}
\label{Deg}
\end{figure}

{\color{black}
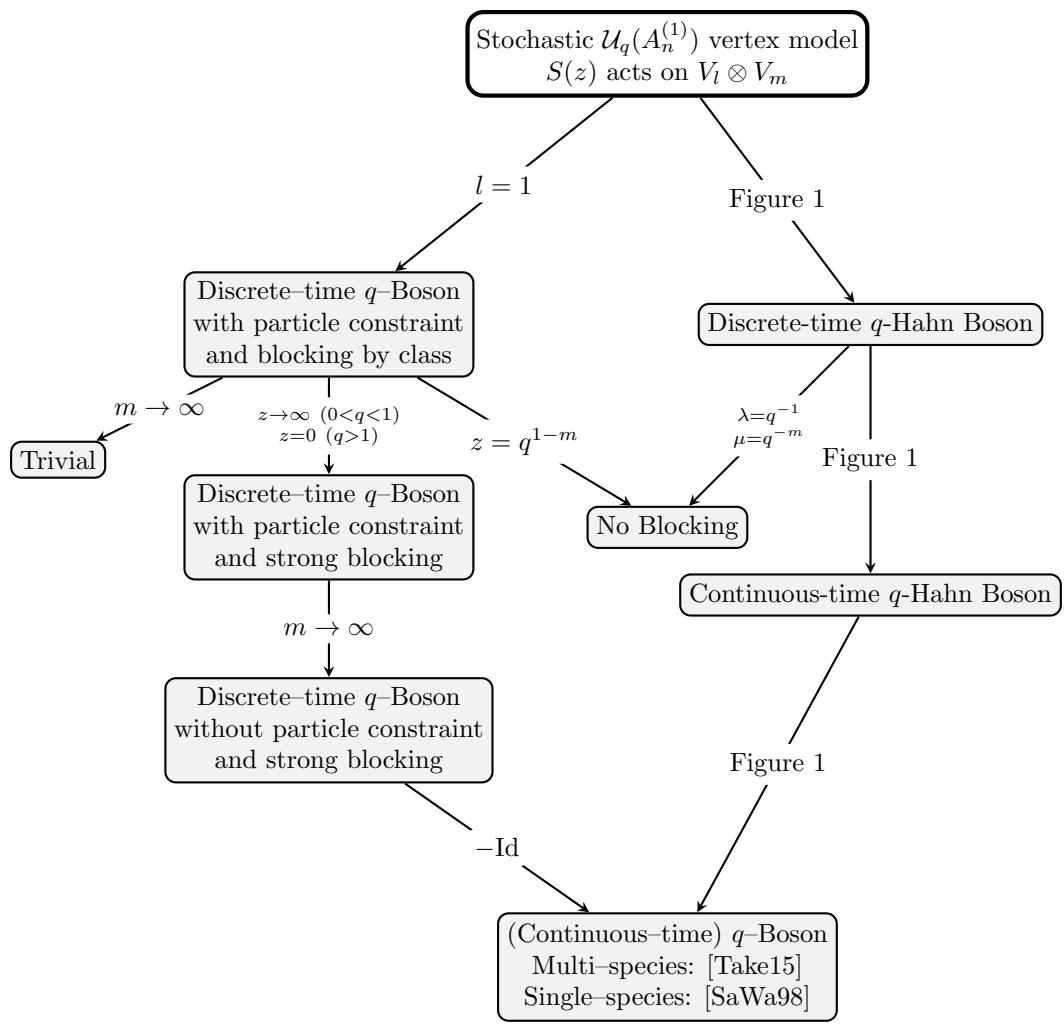
\begin{figure}

\begin{tikzpicture}[scale=0.9, every text node part/.style={align=center}]
\usetikzlibrary{arrows}
\usetikzlibrary{shapes}
\usetikzlibrary{shapes.multipart}
\tikzstyle{process}=[rectangle ,draw,thick ,rounded corners=4pt,fill=black!5 ]
\tikzstyle{AEP}=[rectangle,draw,ultra thick, rounded corners=6pt ]
\tikzstyle{qHahn}=[rectangle,draw, ultra thick, rounded corners=4pt, gray, text=black]
\tikzstyle{limit}=[->,>=stealth,thick,rounded corners=4pt]
\tikzstyle{relation}=[ ultra thick, dashed, rounded corners=4pt, gray]

\node[AEP] (KMMO) at (12,3) {Stochastic $\mathcal{U}_q(A_n^{(1)})$ vertex model \\ $S(z)$ acts on $V_l \otimes V_m$};
\node[process]  (l=1) at (7,-1) {Discrete--time $q$--Boson \\ with particle constraint \\ and blocking by class};
\node[process] (wit) at (7,-4) { Discrete--time $q$--Boson \\ with particle constraint \\ and strong blocking };
\node[process] (trivial) at (3,-3) {Trivial};
\node[process]  (discreteqBoson) at (7,-7) {Discrete--time $q$--Boson \\ without particle constraint \\ and strong blocking};
\node[process]  (qBoson) at (12,-10.5) {(Continuous--time) $q$--Boson \\ Multi--species: \cite{T} \\ Single--species: \cite{SW}};
\node[process]  (cqHahn) at (15,-1) {Discrete-time $q$-Hahn Boson};
\node[process]  (contHahn) at (15,-5) {Continuous-time $q$-Hahn Boson};
\node[process]  (nob) at (12,-4) {No Blocking};

\draw[limit] (KMMO) -- (l=1) node[midway,fill=white]{$l=1$};
\draw[limit] (l=1) -- (wit) node[midway,fill=white]{$\substack{z \rightarrow \infty \ (0<q<1)  \\ z=0 \ (q>1)}$};
\draw[limit] (l=1) -- (trivial) node[midway,fill=white]{$m\rightarrow\infty$};
\draw[limit] (wit) -- (discreteqBoson) node[midway,fill=white]{$m \rightarrow \infty$};
\draw[limit] (discreteqBoson) -- (qBoson) node[midway,fill=white]{$ - \mathrm{Id}$};
\draw[limit] (KMMO) -- (cqHahn) node[midway,fill=white]{Figure \ref{Deg}};
\draw[limit] (l=1) -- (nob) node[midway,fill=white]{$z=q^{1-m}$};
\draw[limit] (cqHahn) -- (nob) node[midway,fill=white]{$\substack{\lambda=q^{-1} \\ \mu=q^{-m}}$};
\draw[limit] (cqHahn) -- (contHahn) node[midway,fill=white]{Figure \ref{Deg}};
\draw[limit] (contHahn) -- (qBoson) node[midway,fill=white]{Figure \ref{Deg}};
\end{tikzpicture}
\caption{Some new degenerations developed in this paper. The limit $z\rightarrow\infty$ is expressed as $w_x\rightarrow 0$, but the set of $x$ such that $w_x\neq 0$ must be infinite. The degeneration labeled $-\mathrm{Id}$ is described in section \ref{CTP}.}
\label{Deg2}
\end{figure}
}

\section{Preliminaries}\label{PR}

This section will review results from previous papers and state necessary definitions.

\subsection{Definitions}
\subsubsection{$q$--notation}

For $z\in \mathbb{C}, q\in (0,1)$ and $n\in \mathbb{N} \cup \{\infty\}$, let the $q$--Pochhammer symbol be defined by
$$
(z;q)_n = (1-z)(1-zq) \cdots (1-zq^{n-1}).
$$
Furthermore define
$$
(q)_n = (q;q)_n, \quad  \binomq{n}{k} = \frac{ (q)_n  }{ (q)_k(q)_{n-k} }.
$$
Notice that
$$
\lim_{q \rightarrow 1} \frac{(q)_n}{(1-q)^n} = n!,
$$
so these can be viewed as $q$--deformations of the usual integers, factorials, and binomials. Another useful $q$--deformation is 
$$
\Q{n} = \frac{q^n - q^{-n}}{q-q^{-1}}, \quad \Q{n}^! = \Q{1} \cdots \Q{n}, \quad \binomQ{n}{k} = \frac{ \Q{n}^!}{ \Q{k}^! \Q{n-k}^!}.
$$
Note that $[n]_q=[n]_{q^{-1}}$, and similarly for the $q$--factorials and binomials.

There is a $q$--analog of the Binomial theorem, which states that if $A$ and $B$ are variables such that $BA=qAB$, then 
$$
(A+B)^l = \sum_{d=0}^l \binomq{l}{d} A^d B^{l-d}.
$$
This can be stated equivalently as a sum over subsets. Any subset $L \subseteq \{1,\ldots,l\}$ can be identified with a monomial in $A$ and $B$ by setting $L$ index the locations of the variable $A$: for example, if $L = \{2,3,5\} \subseteq \{1,\ldots,6\}$, then the corresponding monomial is $BAABAB.$ For any subset  $L  = \{r_1,\ldots,r_d\} \subseteq \{1,\ldots,l\}$ with $d$ elements, let 
$$
c_d(L) = | \{(r,s): 1 \leq r < s \leq l \text{ and } r \in L, s\notin L\} |.
$$ 
Then 
\begin{equation}\label{qBin}
\sum_{\substack{L \subseteq \{1,\ldots,l\} \\ |L| =d}} q^{c_d(L)} = \binomq{l}{d}.
\end{equation}
For example, for $l=4$ and $d=2$, then 
$$
c_d(\{1,2\}) = 4, \quad c_d(\{1,3\}) = 3, \quad c_d(\{1,4\}) = 2, \quad c_d(\{2,4\}) = 1, \quad c_d(\{3,4\}) = 0, \quad c_d(\{2,3\}) = 2 
$$
and
$$
1 + q + 2q^2 + q^3 + q^4 = \frac{(1-q^3)(1-q^4)}{(1-q)(1-q^2)} = \binomq{4}{2}.
$$

Note that by the identity
$$
(q)_{n} = \frac{(q;q)_{\infty}}{(q^{n+1};q)_{\infty}},
$$
it is possible to extend these $q$--deformations to complex numbers. Define the $q$--Gamma function $\Gamma_q(z)$ as 
$$
\Gamma_q(z) = (1-q)^{1-z} \frac{ (q;q)_{\infty} }{ (q^z;q)_{\infty} }
$$
for $\vert q \vert<1$. When $q\rightarrow 1$, $\Gamma_q(z)$ converges to the usual Gamma function $\Gamma(z)$. This definition is related to $\Q{n}^!$ via
$$
\Q{n}^! = q^{n(n-1)/2} \Gamma_{q^{-2}}(n+1).
$$
The right--hand--side is well--defined even if $n \notin \mathbb{N}$, so the $q$--factorials and binomials are still well--defined.

\subsubsection{Representation Theory}\label{RT}
The Drinfeld--Jimbo quantum affine algebra (without derivation) $\mathcal{U}_q(A_n^{(1)}) = \mathcal{U}_q(\widehat{sl}_{n+1})$ is generated by $e_i,f_i,k_i, 0 \leq i \leq n$ with the Weyl relations
\begin{align*}
k_i^{-1} k_i = k_i k_i^{-1} = 1, & \quad [k_i,k_j]=0, \\
k_i e_j = q^{2\delta_{ij} - \delta_{i,j-1} - \delta_{i,j+1}} e_j k_i, & \quad k_i f_j = q^{-2\delta_{ij} + \delta_{i,j-1} + \delta_{i,j+1}} f_j k_i, \\
[e_i, f_j] &= \delta_{ij} \frac{k_i  - k_i^{-1}}{q-q^{-1}} ,
\end{align*}
and the Serre relations (for $n \geq 2$)
\begin{align*}
e_i^2 e_{i+1} - (q+q^{-1}) e_{i+1} e_i e_{i+1} + e_{i+1} e_i^2 &= 0 \\
f_i^2 f_{i+1} - (q+q^{-1}) f_{i+1} f_i f_{i+1} + f_{i+1} f_i^2 &= 0 \\
[e_i,e_j] = [f_i,f_j]&= 0, \quad { j \neq i\pm 1}
\end{align*}
where the indices are taken cyclically (i.e. as elements of $\mathbb{Z} / (n+1)\mathbb{Z})$. The co--product is an algebra homomorphism $\Delta: \mathcal{U}_q(A_n^{(1)})  \rightarrow \mathcal{U}_q(A_n^{(1)}) ^{\otimes 2}$ defined by
$$
\Delta k_i^{\pm 1} = k_i^{\pm 1} \otimes k_i^{\pm 1}, \quad \Delta e_i = 1 \otimes e_i + e_i \otimes k_i, \quad \Delta f_i = f_i \otimes 1 + k_i^{-1} \otimes f_i.
$$
The formula for the co--product is actually not canonical. Another choice of co--product is
$$
\Delta_0 k_i^{\pm 1} = k_i^{\pm 1} \otimes k_i^{\pm 1}, \quad \Delta_0 e_i =  e_i \otimes 1 + k_i^{-1} \otimes e_i, \quad \Delta_0 f_i =  1 \otimes f_i +  f_i \otimes k_i,
$$
which leads to essentially the same algebraic structure. The co--product $\Delta$ satisfies the co--associativity property, which says that as maps from $\mathcal{U}_q(A_n^{(1)}) $ to $\mathcal{U}_q(A_n^{(1)}) ^{\otimes 3}$,
$$
\left( \mathrm{id} \otimes \Delta\right) \circ \Delta = \left( \Delta \otimes \mathrm{id} \right) \circ \Delta.
$$
Because of co-associativity, higher powers of $\Delta$ can be defined inductively and unambiguously as algebrahomomorphisms $\Delta^{(L)}: \mathcal{U}_q(A_n^{(1)})  \rightarrow \mathcal{U}_q(A_n^{(1)})^{\otimes L}$ by
$$
\Delta^{(L)} = \left( \mathrm{id} \otimes \Delta^{(L-1)}\right) \circ \Delta
$$
We will also use Sweedler's notation:
$$
\Delta^{(L)}(u) = \sum_{(u)} u_{(1)} \otimes \cdots \otimes u_{(L)},
$$
where each $u_{(x)}$ is some element of $ \mathcal{U}_q(A_n^{(1)})$.

There is an involution $\omega$ of $\mathcal{U}_q(A_n^{(1)})$ defined on generators by
$$
\omega(k_i) = k_{n+1-i}^{-1}, \quad \omega(e_i) = f_{n+1-i}, \quad \omega(f_i) = e_{n+1-i}.
$$
It is straightforward to check from the Weyl and Serre relations that $\omega$ is indeed an automorphism, and it is immediate from the definition that $\omega^2=\mathrm{id}$.

For $l \in \mathbb{Z}_{\geq 0}$, let $V_l$ be the vector space with basis indexed by the set
$$
{\color{black} \mathcal{B}_l^{(n+1)}:=}\{ \alpha = (\alpha_1,\ldots,\alpha_{n+1}) \in \mathbb{Z}_{\geq 0 }^{n+1} : \alpha_1 + \ldots + \alpha_{n+1} = l\}.
$$
The superscript $(n+1)$ will be dropped if it is clear from the context. For $i \in \{1,2,\ldots,n+1\}=\mathbb{Z}/(n+1)\mathbb{Z}$, define $\hat{i} = (0,\ldots,0,1,-1,\ldots,0) \in \mathbb{Z}^{n+1}$ with the $1$ at the $i$th position and the $-1$ at the $(i+1)$th position. For $z\in \mathbb{C}$, define the representation $\pi_l^z$ of $ \mathcal{U}_q(A_n^{(1)})$ on $V_l$ by
\begin{equation}\label{ACT}
\pi_l^z(e_i) \vert \alpha \rangle = z^{\delta_{i,0}} [\alpha_i]_q \vert \alpha - \hat{i} \rangle, \quad \pi_l^z(f_i) \vert \alpha \rangle = z^{-\delta_{i,0}} [\alpha_{i+1}]_q \vert \alpha + \hat{i} \rangle, \quad 
\pi_l^z(k_i) = q^{\alpha_{i+1}-\alpha_i} \vert \alpha \rangle.
\end{equation}
The parameter $z$ is called the spectral parameter of the representation. Let $V_l^z$ denote the vector space of the representation $\pi_l^z$. The subalgebra generated by $e_i,f_i,k_i, 1\leq i \leq n$ is denoted $\mathcal{U}_q(A_n)$, and for this subalgebra the spectral parameter does not play a role. {\color{black} The vector space $V_l$ is the $l$--th symmetric tensor representation, which will be used in section \ref{FusionSection} below.}

For any $l\geq 0$, let $\vert \Omega \rangle \in V_l$ be the vacuum vector. In other words, $\Omega$ is the basis vector indexed by $(0,\ldots,0,l)$. Analogously, let capital alpha $\vert \mathrm{A} \rangle \in V_l$ denote the basis vector indexed by $(l,0,\ldots,0)$. In general, $\vert \alpha \rangle$ will be interpreted as a particle configuration with an $\alpha_i$ number of $i$th species particles for $1\leq i \leq n$. The $(n+1)$th species of particles will be considered holes. {\color{black} From the viewpoint of probability theory, it is somewhat unnatural to consider holes as being present in the state space. Because of this, it will also be useful to define $\bar{\alpha}=(\alpha_1,\ldots,\alpha_n)$ and $\vert \alpha \vert = \alpha_1 + \ldots + \alpha_n$. Note that if $\alpha \in \mathcal{B}^{(n+1)}_m$, then $\alpha_{n+1}$ equals $m-\vert \alpha \vert$. Thus, any expression $E(\alpha)$ depending on $\alpha \in \mathcal{B}_m$ can be written as an expression $E(\bar{\alpha},m)$ depending on $\bar{\alpha}$ and $m$. In particular, define the limit
$$
\lim_{m\rightarrow \infty} E(\alpha) := \lim_{m\rightarrow\infty} E(\bar{\alpha},m),
$$
where $\bar{\alpha}$ does not depend on $m$.

This definition can be extended to vector spaces and operators. For any $m \in \{0,1,\ldots, \}\cup \{\infty\}$, define $\bar{V}_m$ to be the vector space indexed by the set
$$
\bar{\mathcal{B}}_m = \{ (\alpha_1,\ldots,\alpha_n) \in \mathbb{Z}_{\geq 0}^{n}:  \alpha_1 + \ldots + \alpha_n \leq m\}.
$$
If $\alpha \in \mathcal{B}_m$, then $\bar{\alpha} \in \bar{\mathcal{B}}_m$, and thus any map $M$ on $V_m$ can also be defined as a map on $\bar{V}_m$. Extend the map $M$ on $\bar{V}_m$ to a map on $\bar{V}_{\infty}$ by defining $M$ to be zero outside of $\bar{V}_m$. With these definitions, given any sequence of maps $M_m$ on $V_m$, define the limit $\lim_{m\rightarrow \infty} M_m$ to be a map on $\bar{V}_{\infty}$, if the limit exists.
} 

Given $1 \leq i \leq j\leq n+1$, also let $\alpha_{[i,j]} = \alpha_i + \ldots + \alpha_j$. If $i>j$, then set $\alpha_{[i,j]}=0$.

\subsubsection{Duality}
Recall the definition of stochastic duality:
\begin{definition} Two Markov processes (either discrete or continuous time) $X(t)$ and $Y(t)$ on state spaces $\mathfrak{X}$ and $\mathfrak{Y}$ are dual with respect to a function $D$ on $ \mathfrak{X}\times \mathfrak{Y}$ if
$$
\mathbb{E}_x\left[  D(X(t),y) \right] = \mathbb{E}_y \left[ D(x,Y(t))\right] \text{ for all } (x,y) \in \mathfrak{X}\times \mathfrak{Y} \text{ and all } t \geq 0.
$$
On the left--hand--side, the process $X(t)$ starts at $X(0)=x$, and on the right--hand--side the process $Y(t)$ starts at $Y(0)=y$.

An equivalent definition (for continuous--time processes and discrete state spaces) is that if the generator\footnote{Note that in probabilistic literature, a stochastic matrix has rows which sum to $1$, whereas in mathematical physics literature, the columns sum to 1. This paper uses the latter definition. If the former definition were used, then the definition of duality would be $\mathcal{L}_XD = D\mathcal{L}_Y^*$.} $\mathcal{L}_X$ of $X(t)$ is viewed as a $\mathfrak{X} \times \mathfrak{X}$ matrix, the generator $\mathcal{L}_Y$ of $Y(t)$ is viewed as a $\mathfrak{Y}\times \mathfrak{Y}$ matrix, and $D$ is viewed as a $\mathfrak{X} \times \mathfrak{Y}$ matrix, then $\mathcal{L}_X^*D = D\mathcal{L}_Y$. For discrete--time chains with transition matrices $P_X$ and $P_Y$ also viewed as $\mathfrak{X}\times \mathfrak{X}$ and $\mathfrak{Y}\times \mathfrak{Y}$ matrices, an equivalent definition is
$$
P_X^* D = D P_Y.
$$
If $X(t)$ and $Y(t)$ are the same process, in the sense that $\mathfrak{X}=\mathfrak{Y}$ and $\mathcal{L}_{{X}} = \mathcal{L}_{{Y}}$ (for continuous time) or $P_X=P_Y$ (for discrete--time), then we say that $X(t)$ is self--dual with respect to the function $D$. 

Suppose that $\mathfrak{X} = \mathfrak{Y} = S^I$, where $I \subseteq \mathbb{Z}$ is an interval and $S$ is a countable set. If $\sigma$ is an involution of $I$ such that $\sigma(x+1)=\sigma(x)-1$ for all $x$, then $\sigma$ induces an involution $\sigma^*$ of $S^I$ by
$$
(\sigma^*\eta)(x) = \eta(\sigma(x)) \text{ for } \eta: I\rightarrow S.
$$
 If $\mathcal{L}_X = \sigma^* \circ \mathcal{L}_Y \circ \sigma^*$ and $\mathcal{L}_X^*D = D\mathcal{L}_Y$, then we say that $X(t)$ satisfies {\color{black} space--reversed self--duality} with respect to $D$. 
\end{definition}

\begin{remark}
In the literature, some authors do not draw a distinction between self--duality and space--reversed self--duality. However, for the duality functions of interest here, a \textit{totally} asymmetric process cannot satisfy self--duality, but it does satisfy space--reversed self--duality (see the remarks before Proposition 2.6 of \cite{K}). The terminology here is chosen to emphasize this distinction.
\end{remark}

\begin{remark}\label{Constant}
Note that if $c(x,y)$ is a function on $\mathfrak{X} \times \mathfrak{Y}$ which is constant under the dynamics of $X(t)$ and $Y(t)$, then $c(x,y)D(x,y)$ is also a duality function. This will be used to simplify the expression for $D(x,y)$. For this paper, $c(x,y)$ will be a function which only depends on the number of particles of each species, which is a constant assuming particle number conservation. See \cite{O} for an example of duality functions on a lattice with open boundary conditions, in which this type of simplification is not applicable.
\end{remark}

\subsubsection{Lumpability}\label{LumpSection}
Let $T_0$ be a $\mathfrak{X}_0 \times \mathfrak{X}_1$ matrix and let $\mathcal{P}_0 = \{ p_0^{(i)}\}_i$ be a partition of $\mathfrak{X}_0$, and $\mathcal{P}_1 = \{p_1^{(i)} \}_i$ a partition of $\mathfrak{X}_1$.  Recall the convention that a matrix is stochastic if the columns (rather than the rows) sum to $1$. Say that $T_0$ is \textit{lumpable} (with respect to $\mathcal{P}_1$ and $\mathcal{P}_2$) if for all $p_1^{(i)}\in \mathcal{P}_1$ 
$$
\sum_{x_1 \in p_1^{(i)}} T_0(x_1,x_0) = \sum_{x_1 \in p_1^{(i)}} T_0(x_1,x'_0) 
$$
 whenever $x_0,x_0'$ are in the same block $p_0^{(j)} \in \mathcal{P}_0$. Define the $\mathcal{P}_0 \times \mathcal{P}_1$ matrix $\bar{T_0}$ by setting $\bar{T_0}(p_0^{(j)}, p_1^{(i)})$ to be the quantity above. 
 
The composition of lumpable matrices is again lumpable. To see this, If $T_1$ is a $\mathfrak{X}_1 \times \mathfrak{X}_2$ matrix which is lumpable with respect to $\mathcal{P}_1$ and $\mathcal{P}_2$, then for $x_0 \in p_0^{(k)}$,
\begin{align*}
\sum_{x_2 \in p_2^{(i)}} (T_0T_1)(x_2,x_0) &= \sum_{x_2 \in p_2^{(i)}} \sum_{x_1 \in \mathfrak{X}_1} T_0(x_2,x_1)T_1(x_1,x_0) \\
&= \sum_{p_1^{(j)} \in \mathcal{P}_1} \sum_{x_1 \in p_1^{(j)}} \sum_{x_2 \in p_2^{(i)}} T_0(x_2,x_1)T_1(x_1,x_0)\\
&=\sum_{p_1^{(j)} \in \mathcal{P}_1} \sum_{x_1 \in p_1^{(j)}} \bar{T}_0(p_2^{(i)},p_1^{(j)}) {T}_1(x_1,x_0) \\
&=\sum_{p_1^{(j)} \in \mathcal{P}_1} \bar{T}_0(p_2^{(i)},p_1^{(j)})  \bar{T}_1(p_1^{(j)},p_0^{(k)})\\
&=\bar{T}_0\bar{T}_1(p_2^{(i)},p_0^{(k)}).
\end{align*}
This does not depend on the choice of $x_0$ in $p_0^{(k)}$, so $T_0T_1$ is lumpable with $\overline{T_0T_1} = \bar{T_0}\bar{T_1}$.

This is a generalization of a \textit{lumped Markov process} introduced in \cite{KS}. The condition that a Markov process is lumpable is simply the condition that a projection of a Markov process is still Markov. There are more general conditions of interest: for example, \cite{PR} gives an intertwining condition in which the projection is random. In particular, if $T_0$ is a stochastic $\mathfrak{X}_0 \times \mathfrak{X}_0$ matrix, and $\Xi$ is a stochastic $\mathfrak{X}_0 \times \mathfrak{X}_1$ matrix, $\Lambda$ is a stochastic $\mathfrak{X}_1 \times \mathfrak{X}_0$ matrix, then define the $\mathfrak{X}_1 \times \mathfrak{X}_1$ matrix
$$
T_1 = \Lambda T_0 \Xi.
$$
If $\Lambda\Xi$ is the identity matrix on $\mathfrak{X}_1$, and $T_0,T_1$ satisfy the intertwining relation
$$
\Lambda T_0 = T_1 \Lambda,
$$
then $\Lambda$ maps the Markov chain defined by $T_0$ to a well--defined Markov chain defined by $T_1$. 

It is not hard to see that the Pitman--Rogers relation is a generalization of lumpability. if $\mathfrak{X}_1$ is a partition $\mathcal{P}_0$ of $\mathfrak{X}_0$, pick an arbitrary element $x(p)\in p$ for each $p\in \mathcal{P}_0$. Then define 
$\Xi$ and $\Lambda$ by 
$$
\Xi(x,p) = 1_{x = x(p)}, \quad \Lambda(p,x) = 1_{x\in p}.
$$
It is immediate that $\Lambda\Xi$ is the identity matrix on $\mathcal{P}_0$. If $T_0$ is lumpable with respect to $\mathcal{P}_0$, then
$$
T_1(p_1,p_2)=(\Lambda T_0 \Xi ) (p_1,p_2) = \sum_{x \in p_1} T_0(x,x(p_2))
$$
does not depend on the choice of $x(p)$, and is the transition matrix of the lumped Markov chain. Furthermore, the lumpability implies that for $x \in q$,  
$$
\Lambda T_0(p,x) = \sum_{y \in p} T_0(y,x) = \sum_{y\in p} T_0(y,x(q))= \Lambda T_0 \Xi\Lambda(p,x)=T_1\Lambda(p,x).
$$

\begin{remark}
One example of a lumpable Markov process is $n$--species ASEP. In this process, there are $n$ species of particles, and at most one particle may occupy a lattice site. One can think of each species as having a different mass. If a particle attempts to jump to a site occupied with a heavier particle, then the jump is blocked. If a particle attempts to jump to a site occupied with a lighter particle, then the two particles switch places. All left jumps occur with the same rate (independent of the species), and likewise all the right jumps occur with the same rate (independent of the species). It is not hard to see that the projection onto the first $k$ species results in $k$--species ASEP, since each particle treats lighter particles the same as holes. This model was first introduced in \cite{L}.

A more general model would allow the jump rates to depend on the species of the particles. In this case, the projection onto the first $k$ species is no longer a $k$--species ASEP. See \cite{Ka,K2} for examples of multi--species ASEP which have jump rates depending on the species. With open boundary conditions, several models (see. e.g. \cite{C2,CFRV,CGGW,CMW,M,MV,Uch}) have jump rates at the boundaries which depend on the species, with jump rates in the bulk that are independent of the species. See also \cite{CGW,PEM} for multi--species ASEP on a ring, with jump rates independent of the species.

\end{remark}

\subsubsection{Operator Notation}
We introduce some notation for operators. Given two linear spaces $V$ and $W$, a symbol of the form $M_{VW}$ will denote an linear map with domain $V\otimes W$. In particular, let $P_{VW}$ be the permutation operator $P_{VW}: V \otimes W \rightarrow W \otimes V$ defined by $P_{VW}(v\otimes w) = w\otimes v$. Given an operator $M$ from $V\otimes W$ to itself, let $\widetilde{M}$ denote the reversed operator on $W\otimes V$:
$$
\widetilde{M} = P_{VW} \circ M \circ P_{WV}.
$$
Given $M$ on $V\otimes W$, let $\check{M}: V\otimes W \rightarrow W \otimes V$ be the map $P_{VW} \circ M$. 

Suppose $\{V_m: m \geq 0\}$ is a family of vector spaces and for each $m\geq 0$, $M_m$ is an operator on $V_m$. By abuse of notation, the subscript $m$ in $M_m$ will often be dropped. Given $m_1,\ldots,m_L$, the tensor power $V_{m_1} \otimes \cdots \otimes V_{m_L}$ will be denoted $V^{(L)}$. For $1 \leq a \leq b \leq L$, let $V^{[a,b]}$ denote $V_{m_a} \otimes \cdots \otimes V_{m_b}$. Let $M^{\otimes L}$ denote the operator $M_{m_1} \otimes \cdots \otimes M_{m_L}$ on $V^{ (L)}$. Given $\sigma \in S_L$, let $P^{\sigma}$ be the operator from $V_{m_1} \otimes \cdots \otimes V_{m_L}$ to $V_{m_{\sigma(1)}} \otimes \cdots \otimes V_{m_{\sigma(L)}}$ defined by
$$
P^{\sigma}(v_1 \otimes \cdots \otimes v_L) =  v_{\sigma(1)} \otimes \cdots \otimes v_{\sigma(L)}
$$
and note that
$$
P^{\sigma} M^{\otimes L} = M^{\otimes L} P^{\sigma}.
$$

If $G$ is an operator on $V^{( L)}$ and $\sigma \in S_L$ is the reversal permutation $\sigma(j) = L+1-j$, let 
$$
\widetilde{G} = P^{\sigma} \circ G \circ P^{\sigma}.
$$
If $R$ acts on $V_l\otimes V_m$ for $l,m\geq 0$, then $R_{ij}$ is the action on the $i,j$ component of the tensor product of $V^{\otimes L}$ for $1\leq i,j\leq L$.

The --ket $\vert \alpha,\beta \rangle$ means $\vert \alpha \rangle \otimes \vert \beta \rangle$, and similarly for the bra-- $\langle \alpha,\beta \vert$. The Greek letters $\eta$ and $\xi$ will denote multiple tensor products, e.g. $\vert \eta \rangle = \vert \eta^1,\ldots,\eta^L \rangle = \vert \eta^1 \rangle \otimes \cdots \otimes \vert \eta^L \rangle$.

As usual, $M^*$ denotes the transpose of $M$. 

\subsection{Results from \cite{KMMO}}\label{KMMOBACK}
For this section, $\alpha,\gamma\in V_l$ and $\beta,\delta \in V_m$, where $ 0 \leq l \leq m$.

For every $z \in \mathbb{C}$ and $l,m \geq 0$, there is an $R$--matrix 
$$
R(z): V_l^{z_1} \otimes V_m^{z_2} \rightarrow V_l^{z_1} \otimes V_m^{z_2},
$$
where $z=z_1/z_2$ depends on the spectral parameters $z_1,z_2$ of $V_l^{z_1},V_m^{z_2}$. This $R$--matrix is characterized, up to a constant, by the \textit{intertwining property} (see (4) of \cite{KMMO})
\begin{equation}\label{Inter}
\check{R}(z) \Delta(u) = {\Delta}(u) \check{R}(z) ,
\end{equation}
viewed as maps $V_l^{z_1} \otimes V_m^{z_2} \rightarrow V_m^{z_2} \otimes V_l^{z_1}$.  
The constant is normalized by
\begin{equation}\label{Unit}
R(z) \vert \Omega \rangle = \vert \Omega \rangle.
\end{equation}
The $R$--matrix also satisfies the Yang--Baxter equation.

It also satisfies {(\color{black} see (12) of \cite{KMMO}, which is a corollary of (2.4) and (2.24) of  \cite{KOS}})
$$
\langle \gamma,\delta \vert R(z) \vert \alpha,\beta \rangle  =  \langle {\alpha',\beta'} \vert R(z) \vert {\gamma',\delta'} \rangle\prod_{i=1}^{n+1} \frac{(q^2)_{\alpha_i} (q^2)_{\beta_i} }{ (q^2)_{\gamma_i} (q^2)_{\delta_i}},
$$
where $\alpha'$ denotes the charge--reversed array of $\alpha$, i.e. $\alpha'=(\alpha_{n+1},\ldots,\alpha_1)$. We can write this as
\begin{equation}\label{Trans}
\Pi^{\otimes 2} B^{\otimes 2} R(z) =  R^*(z) \Pi^{\otimes 2} B^{\otimes 2},
\end{equation}
where $B$ is the diagonal change of basis matrix on $V_l, l \geq 0$ 
$$
\langle \gamma \vert B \vert {\alpha} \rangle = 1_{\{\alpha=\gamma\}} \prod_{i=1}^{n+1} (q^2)_{\gamma_i} 
$$
and $\Pi$ is the charge--reversal matrix on $V_l, l \geq 0$,
$$
\langle  {\gamma} \vert \Pi \vert {\alpha} \rangle = 1_{\{\alpha = \gamma ' \}}.
$$
Note that for $\sigma \in S_L$,
\begin{equation}\label{Diag}
B\Pi = \Pi B, \quad P^{\sigma}\Pi^{\otimes L}=\Pi^{\otimes L}P^{\sigma}, \quad B^{\otimes L}P^{\sigma}=P^{\sigma}B^{\otimes L}.
\end{equation}
Additionally, $R$ satisfies the particle conservation property, which is that
\begin{equation}\label{PCP}
\langle {\gamma,\delta} \vert R(z) \vert {\alpha,\beta} \rangle  = 0 \text{ if } \alpha+\beta \neq \gamma + \delta.
\end{equation}

The $S$--matrix $S(z)$ is related to the $R(z)$ by a gauge transform, with the explicit definition (see (15) of \cite{KMMO})
$$
S(z) = \widetilde{\Ga}^{-1} R(z) {\Ga}.
$$
Here, let $\Ga$ be the gauge transform defined on $V^{[1,L]}$ by the diagonal matrix 
\begin{align*}
\langle \eta \vert \Ga \vert {\xi}\rangle &= 1_{\{\xi_i^x = \eta_i^ x \ \forall i,x\}} q^{ - \sum_{1 \leq i < j \leq n+1} \sum_{1 \leq y <x \leq L} \xi_i^y \xi_j^x} \\
&= 1_{\{\eta = \xi\} } q^{-\sum_{1 \leq y < x \leq L} \sum_{i=1}^{n} \xi_{[1,i]}^y \xi_{i+1}^x}.
\end{align*}
The operator $\Pi$ intertwines between $\Ga$ and $\widetilde{\Ga}$ via
\begin{equation}\label{VVV}
 \Pi^{\otimes L}  \circ \Ga \circ \Pi^{\otimes L}  = \widetilde{\Ga}, 
\end{equation}
because the left--hand--side switches the indices $i$ and $j$, while the right--hand--side switches the indices $x$ and $y$. The original paper \cite{KMMO} only defined $\Ga$ for $L=2$, but it will be seen below that this is the natural extension to $L$ sites. The superscript $\Ga^{(L)}$ will sometimes be included to emphasize the number of lattice sites.

Just like the $R$--matrix $R(z)$, the $S$--matrix $S(z)$ satisfies the Yang--Baxter equation. 

\begin{remark} In a comment after Theorem 6 of \cite{KMMO}, it is noted that the gauge transformation comes from the $U_q(A_n)$--orbit of the unit normalization condition \eqref{Unit}.
This is a similar idea to the ``ground state transformation'' of \cite{K}, using the framework of \cite{CGRS}, in which the creation operators $e_i \in U_q(A_n)$ are applied to the ground state $\vert \Omega \rangle$. Because of this similarity, one might expect a simple relationship between these two transformations. Indeed, this will be stated and shown explicitly in Proposition \ref{GaGr}.
\end{remark}

\begin{remark}
Theorem 6 of \cite{KMMO} explicitly states that the sum of the output of $S(z)$ is equal to $1$, for any value of the spectral parameter $z$. This result is proved again in \cite{BM}, using a factorized expression for $S(z)$, and furthermore gives a range of values for which $S(z)$ has non--negative entries, and is therefore stochastic. Section \ref{Stoch} below will also give a range of values of $z$ for which $S(z)$ is stochastic, using different methods. 
\end{remark}

The transfer matrices (with periodic boundary conditions) are defined as follows:
$$
\mathrm{Tr}\vert_{V_l} \left( S_{0,L}(z/w_L) \cdots S_{0,1}(z/w_1)\right) \in \mathrm{End}\left( V_{m_1}^{} \otimes \cdots \otimes V_{m_L}^{} \right),
$$
where $S_{0,L}\cdots S_{0,1}$ is viewed as an operator on $V_l \otimes V_{m_1} \otimes \cdots \otimes V_{m_L}$, and the trace is taken over the auxiliary space $V_l$. As stated in \cite{KMMO}, the stochastic $S$--matrices satisfying the Yang--Baxter equation implies that the transfer matrices form a commuting family, but this will not be needed here. Because the transfer matrices are operators on $V^{[1,L]}$, they can be viewed as transition matrices for a particle system on the lattice $\{1,\ldots,L\}$.

So far, we have not used any explicit formulas for the stochastic matrix $S(z)$. In \cite{KMMO}, there were explicit formulas for $S(z)$ acting on $V_l \otimes V_m$ when $z=q^{l-m}$ and $l\leq m$. In this case, define
$$
\Phi_q(\gamma \vert \beta;\lambda,\mu) = q^{\chi_{\beta,\gamma}} \left( \frac{\mu}{\lambda} \right)^{\vert \gamma \vert} \frac{(\lambda;q)_{\vert \gamma \vert} (\tfrac{\mu}{\lambda};q )_{\vert \beta \vert - \vert \gamma \vert} }{ (\mu;q)_{\vert \beta \vert}}  \prod_{i=1}^{n} \binomq{\beta_i}{\gamma_i},
$$
where
$$
\chi_{\beta,\gamma} = \sum_{1\leq i<j \leq n} (\beta_i-\gamma_i)\gamma_j,
$$
and the stochastic operator, written as $\mathcal{S}$, acts as
$$
\langle {\gamma,\delta} \vert \mathcal{S} \vert {\alpha,\beta} \rangle  = 1_{\{\alpha+\beta=\delta+\gamma\}}\Phi_q({\gamma} \vert {\beta};q^{-l},q^{-m}).
$$
Also note that taking the derivative in $\lambda$, one obtains (see (43) of \cite{KMMO})
$$
\frac{\partial \Phi_q(\gamma \vert \beta;\lambda,\mu)}{\partial \lambda} \Big|_{\lambda=1} = \Phi'_q(\gamma \vert \beta;\mu) := q^{\chi_{\beta,\gamma}} \mu^{\vert\gamma\vert-1} \frac{ (q)_{\vert \gamma\vert -1} }{ (\mu q^{\vert \beta\vert-\vert\gamma\vert};q)_{\vert \gamma\vert}}   \prod_{i=1}^n \binomq{\beta_i}{\gamma_i}.
$$
The functions $\Phi_q$ and $\Phi'_q$ can be used to define a multi--species version of the $q$--Hahn Boson process of \cite{Po} in discrete and continuous--time (respectively), as described below in Section \ref{POVBACK}. Observe that $\mathcal{S}$ only depends on $\alpha$ through particle conservation, which will not be true of $S(z)$ for generic values of $z$. This means that parallel update will generally not be possible.

It is worth noting that several subsequent papers (\cite{KO,KO2}) prove results for the stochastic $\mathcal{U}_q(A_n^{(1)})$ vertex model in the $q$--Hahn Boson degeneration. It is possible that those results hold for more general values of the spectral parameter, but this is not pursued here.

The appendix of \cite{KMMO} also includes explicit formulas for $R(z)$ acting on $V_l \otimes V_m$ in the cases that $l=1$ and $m=1$. The vector space $V_1$ is spanned by the basis elements $\{\epsilon_k: 1\leq k \leq n+1\}$ where $\epsilon_k$ has a $1$ at the $k$th location and zeroes elsewhere. When $l=1,$
\begin{equation}\label{l=1}
R(z)_{\epsilon_j,\beta}^{\epsilon_k,\delta} 
= 1_{ \{\epsilon_j  + \beta = \epsilon_k + \delta \}} \times 
\begin{cases}
q^{\beta_k+1}\dfrac{1-q^{-2\delta_k+m-1}z}{q^{m+1}-z}, & \text{ if } k=j \\
-q^{\beta_{j+1} + \cdots + \beta_{k-1}} \dfrac{1-q^{2\beta_k}}{q^{m+1}-z}, & \text{ if } k>j, \\
-q^{m-(\beta_k+\ldots+\beta_j)} \dfrac{z(1-q^{2\beta_k})}{q^{m+1}-z}, & \text{ if } k <j.
\end{cases}
\end{equation}
When $m=1$,
\begin{equation}\label{m=1}
R(z)_{\alpha,\epsilon_j}^{\gamma,\epsilon_k} 
= 1_{ \{ \alpha + \epsilon_j   = \gamma + \epsilon_k  \}} \times 
\begin{cases}
q^{\alpha_k+1}\dfrac{1-q^{-2\alpha_k+l-1}z}{q^{l+1}-z}, & \text{ if } k=j \\
-q^{l-(\alpha_j+\ldots+\alpha_k)} \dfrac{z(1-q^{2\alpha_k})}{q^{l+1}-z}, & \text{ if } k >j,\\
-q^{\alpha_{k+1} + \cdots + \alpha_{j-1}} \dfrac{1-q^{2\alpha_k}}{q^{l+1}-z}, & \text{ if } k<j.
\end{cases}
\end{equation}

\subsection{The $q$--Hahn Boson process}\label{POVBACK}
For $n=1$, the function $\Phi_q$ from the previous section was introduced by \cite{Po} in the form $\varphi(m\vert m') = \Phi_q(m \vert m', \nu/\mu,\nu)$, and was used to define the (single--species) $q$--Hahn Boson process in discrete and continuous time. 

The state space consists of particle configurations $\eta=(\eta^x)$ on a lattice. There is no restriction on $\eta^x$, the number of particles at lattice site $x$. The discrete--time update can be described in the following way. Given a particle configuration $(\eta^x)$ with $\eta^x$ particles at lattice site $x$, the probability measure after the (discrete--time) update is described by 
$$
\mathbb{P}(\text{exactly } k \text{ particles at site } x) = \sum_{ \substack{ \gamma^{x-1} \geq 0 \\ \gamma^x \geq 0}} 1_{\{\eta^x - \gamma^x + \gamma^{x-1} = k\}} \Phi_q( \gamma^x \vert \eta^x) \Phi_q( \gamma^{x-1} \vert \eta^{x-1}).
$$
The physical description is that with probability $\Phi_q( \gamma^{x-1} \vert \eta^{x-1})$, $\gamma^{x-1}$ particles leave lattice site $x-1$ and jump to lattice site $x$. Simultaneously (i.e. in parallel), $\gamma^x$ particles leave lattice site $x$ and jump to lattice site $x+1$ with probability $\Phi_q( \gamma^x \vert \eta^x)$. In this case, we will say that the process evolves with total asymmetry to the right. If the probability measure after the update is given by 
$$
\mathbb{P}(\text{exactly } k \text{ particles at site } x) =  \sum_{\substack{ \gamma^{x} \geq 0 \\ \gamma^{x+1} \geq 0}} 1_{\{\eta^x - \gamma^x + \gamma^{x+1} = k \}} \Phi_q( \gamma^x \vert \eta^x) \Phi_q( \gamma^{x+1} \vert \eta^{x+1}),
$$
then we say that the process evolves with total asymmetry to the left.

The continuous--time update can be described as follows. For evolution with total asymmetry to the left, the generator ${\mathcal{L}}$ can be written as a sum of local generators $\sum_x {\mathcal{L}}_x$, where the off-diagonal entries of ${\mathcal{L}}_x$  are
$$
\langle \xi \vert {\mathcal{L}}_x \vert \eta \rangle = 
\begin{cases}
0, & \text{ if } \eta^y \neq \xi^y \text{ for some } y \neq x,x+1  \\
0, & \text{ if } \eta^x + \eta^{x+1} \neq \xi^x + \xi^{x+1} , \\
\Phi'_q( \eta^x - \xi^x \vert \eta^x), & \text{ else},
\end{cases}
$$ 
and the diagonal entries are given by the condition that $\sum_{\xi} \langle \xi \vert {\mathcal{L}}_x \vert \eta \rangle=0$. The first line indicates that $\mathcal{L}_x$ only causes particles to jump from $x$ to $x+1$, and the second line expresses particle conservation. 

For evolution with total asymmetry to the right, the generator $\tilde{\mathcal{L}}$ can be written as a sum of local generators $\sum_x \tilde{\mathcal{L}}_x$, where the off-diagonal entries of $\tilde{\mathcal{L}}_x$  are
$$
\langle \xi \vert \tilde{\mathcal{L}}_x \vert \eta \rangle = 
\begin{cases}
0, & \text{ if } \eta^y \neq \xi^y \text{ for some } y \neq x,x-1  \\
0, & \text{ if } \eta^x + \eta^{x-1} \neq \xi^x + \xi^{x-1} , \\
\Phi'_q( \eta^x - \xi^x \vert \eta^x), & \text{ else}.
\end{cases}
$$

At $n=1$, the single--species continuous--time $q$--Hahn Boson process can also be constructed through a deformation of an affine Hecke algebra \cite{T0}. Additionally, for general $n$ and $\mu= 0$, the process had been previously constructed in \cite{T} using a higher rank affine Hecke algebra, and there it is called a multi--species $q$--Boson process. The single--species $q$--Boson process goes back to \cite{SW}.

{\color{black} See Figure \ref{Deg} for a diagram showing the various processes.}

\begin{remark}
Even though the entries of $S(q^{l-m})$ acting on $V_l \otimes V_m$ can be written in terms of the function $\Phi_q$, it is not technically accurate to describe the resulting particle system as the $q$--Hahn Boson process. This is because the state space of the $q$--Hahn Boson process does not place a constraint on the number of particles at each site, whereas the vector space $V_m$ only constrains for $m$ particles at a site. This distinction is important here because Proposition \ref{DualityResult} is false if only finitely many particles are allowed at each site; see Remark \ref{Closed} for an explanation. However, after analytic continuation in the variables $\lambda,\mu$, there is no longer such a particle constraint, and the statement is true.
\end{remark}

\subsection{Results from \cite{K}}\label{KUANBACK}
Let $\exp_{r}$ be the deformed exponential
$$
\exp_r(x) = \sum_{k=0}^{\infty} \frac{x^k}{\{k\}^!_r}, \text{ where } \{k\}_r = \frac{1-r^k}{1-r}, \quad \{k\}_r^! = \{1\}_r \cdots \{k\}_r.
$$
Note that as $r\rightarrow 1$, $\exp_r$ becomes the usual exponential. Let $u_0$ be the element\footnote{Note that because $u_0$ is an infinite sum in the generators, it actually belongs to a completion of $\mathcal{U}_q(A_n)$.} of $\mathcal{U}_q\left(A_n\right)$  
$$
u_0 :=   (\exp_{q^{2}} f_1) \cdots ( \exp_{q^{2}} f_n).
$$ 
The deformed exponential $\exp_r$ satisfies a pseudo--factorization property (see \cite{CGRS}), which implies
\begin{equation}\label{PseudoFac}
\exp_{q^2}(\Delta^{(L)}f_i) = \exp_{q^2}(f_i \otimes 1^{\otimes L-1})  \exp_{q^2}(k_i^{-1} \otimes f_i \otimes 1^{\otimes L-2})  \cdots \exp_{q^2}((k_i^{-1})^{\otimes L-1} \otimes f_i)
\end{equation}
This will result in a simpler expression for the duality function, as will be seen below.

Let $\Gr$ be the ground state transformation, which is the diagonal matrix with entries
$$
\langle {\xi} \vert  \Gr \vert{\xi} \rangle = \langle \xi \vert \Delta^{(L)}(u_0) \vert \Omega^{ {\color{black} \otimes L}} \rangle,
$$
where $\vert \Omega^{\otimes L}\rangle$ denotes the vacuum vector $\vert \Omega \rangle$ on $L$ sites. This transformation was previously used in \cite{K}, using the framework of \cite{CGRS}. By Proposition 4.2 of \cite{K},
\begin{equation}\label{Gr}
\langle {\eta} \vert \Gr \vert {\eta} \rangle  = \mathrm{const} \cdot \prod_{i=1}^{n+1} \prod_{x=1}^L \frac{1}{ [\eta_i^x]_q^!}
 \times \prod_{i=1}^{n}  \prod_{1 \leq y <x \leq L} q^{  -\eta_{i+1}^y \eta_{[1,i]}^x} ,
\end{equation}
where for $\eta=(\eta_1,\ldots,\eta_{n+1}) \in \mathbb{Z}^{n+1}$ and $1 \leq i \leq j \leq n+1$, set
$$
\eta_{[a,b]} = \eta_a + \ldots + \eta_b.
$$
Here, and below, we say that a function $c(\eta,\xi)$ is constant under particle conservation if it only depends on the values of $ \eta_i^1 + \ldots + \eta_i^{L+1}$ and $\xi_i^1 + \ldots + \xi_i^{L+1}$ for $1 \leq i \leq n+1$. The notation $\mathrm{const}$ will denote a constant under particle conservation. By Remark \ref{Constant}, if $D(\eta,\xi)$ is a duality function then so is $c(\eta,\xi)D(\eta,\xi)$, as long as particle conservation is satisfied. 

For $u\in \mathcal{U}_q(A_n^{(1)})$, define the operator on $V_{m_1} \otimes \cdots \otimes V_{m_L}$: 
$$
D(u) = \Pi^{\otimes L} \circ \Gr^{-1} \circ \Delta^{(L)}(u) \circ \Gr^{-1} \circ \left(B^{-1}\right)^{\otimes L}.
$$
Note that the $B$ in this paper is denoted $B^{-2}$ in \cite{K}. Then by Proposition 5.1 of \cite{K}, the operator $D(u_0)$ has an explicit formula, which is that $\langle \xi \vert D(u_0) \vert \eta \rangle$ is equal to 
\begin{equation}\label{ExpExp0}
\mathrm{const}\cdot \prod_{x=1}^L \Q{m_x-\eta_{[1,n]}^x}^!  \prod_{i=1}^{n}   1_{\{m_x -\eta_{[1,n+1-i]}^x \geq \xi_{[1,i]}^x\}}   \frac{  \Q{  m_x - \eta_{[1,n-i]}^x - \xi_{[1,i]}^x   }^!  }{\Q{m_x - \eta_{[1,n+1-i]}^x   - \xi_{[1,i]}^x }^! }   q^{ -\xi_i^x (\sum_{y>x} 2\eta_{[1,n+1-i]}^y + \eta_{[1,n+1-i]}^x)} .
\end{equation}
It is not hard to see that an equivalent expression is
\begin{equation}\label{ExpExp}
\langle \xi \vert D(u_0) \vert \eta \rangle=
\mathrm{const}\cdot \prod_{x=1}^L \Q{\eta_1^x}^! \cdots\Q{\eta_{n+1}^x}^!  \prod_{i=1}^{n}     \binomQ{  m_x - \eta_{[1,n-i]}^x - \xi_{[1,i]}^x   }{ \eta_{n+1-i}^x }   q^{ -\xi_i^x (\sum_{y>x} 2\eta_{[1,n+1-i]}^y + \eta_{[1,n+1-i]}^x)}.
\end{equation}
To see this, note that the indicator term $1_{\{m_x -\eta_{[1,n+1-i]}^x \geq \xi_{[1,i]}^x\}}$ can be removed, because if its condition does not hold, then the $q$--binomial term is zero anyway. The other necessary identity is
$$
\Q{\eta_1^x}^!\cdots \Q{\eta_{n+1}^x}^!\prod_{i=1}^{n}\binomQ{  m_x - \eta_{[1,n-i]}^x - \xi_{[1,i]}^x   }{ \eta_{n+1-i}^x } = \Q{\eta_1^x}^!\cdots \Q{\eta_{n+1}^x}^!\prod_{i=1}^{n} \frac{\Q{m_x - \eta_{[1,n-i]}^x  }^!}{\Q{m_x - \eta_{[1,n+1-i]}^x  }^! \Q{\eta_{n+1-i}^x}^!}.
$$
Note that the expression for $D(u_0)$ is still well--defined even when $m_x \notin \mathbb{N}$.
Letting $\mu_x=q^{-m_x}$, the operator will sometimes be denoted $D_{\mu}(u_0)$ to emphasize the dependence on $\mu$. Additionally (see the proof of Proposition 5.2(b) of \cite{K}),
\begin{equation}\label{mLimit}
\lim_{m_x \rightarrow \infty} \Q{\eta_1^x}^! \cdots\Q{\eta_{n+1}^x}^!  \prod_{i=1}^{n}     \binomQ{  m_x - \eta_{[1,n-i]}^x - \xi_{[1,i]}^x   }{ \eta_{n+1-i}^x }   = \prod_{i=1}^n q^{\xi_i^x\eta^x_{[1,n+!-i]}}
\end{equation}
so that if one takes all $m_x\rightarrow\infty$ and assumes $q>1$, then the limit is
\begin{equation}\label{LimitDuality}
\langle \xi \vert D_0 \vert \eta \rangle =\prod_{x=1}^L \prod_{i=1}^{n}      q^{ \xi_i^x \left(\sum_{y \leq x} 2\eta_{[1,n+1-i]}^y \right) } .
\end{equation}
Since $u_0 \in \mathcal{U}_q(A_n)$ and does not involve $f_0$, the operator $D(u_0)$ does not involve the spectral parameter.  

The paper \cite{K} constructs a multi--species version of a process called $\textrm{ASEP}(q,j)$. The single--species case was introduced in \cite{CGRS}, and is itself a generalization of the usual $\textrm{ASEP}$, in which up to $2j$ particles can occupy a lattice site. In the homogeneous case when all $m_x=2j$, \cite{K} shows that this multi--species $\textrm{ASEP}(q,j)$ has a duality property with respect to the function $D_{\mu}(u_0)$. 

Furthermore, when $j\rightarrow\infty$, the multi--species $\textrm{ASEP}(q,j)$ converges to the multi--species $q$--TAZRP of \cite{T}. Taking $m_x\rightarrow\infty$, this shows that the multi--species $q$--TAZRP satisfies the duality with respect to the duality function $D_0$ of \eqref{LimitDuality}. This was explicitly stated in Theorem 2.5(b) of \cite{K}, and will be proved again below as Corollary \ref{ReProve}.

\subsection{Fusion}\label{FusionSection}
The $R$--matrix $R(z_1/z_2)$ acting on $V_l^{z_1} \otimes V_m^{z_2}$ can also be defined through a process called fusion, developed in \cite{KRL}. See also the exposition in section 3.5 of \cite{Resh10}\footnote{\color{black} Note that the notation here is different than in \cite{Resh10}, due to slightly different conventions in the definition of the quantized affine Lie algebras, which result in substitutions $q\rightarrow q^{-2},z\rightarrow z^{1/2}$. Section \ref{ReshConfirm} will give examples demonstrating that this is the correct expression.}. The representation $V_m$ is the $m$th symmetric tensor representation, meaning that it is the symmetric projection of $V_1^{\otimes m}$, the $m$th tensor power of the canonical representation $V_1 = \mathbb{C}^{n+1}$. There is an isomorphism (of representations) from $V_m$ to the image of the symmetric projection. This isomorphism is unique up to a constant, because it must map a lowest weight vector of $V_m$ to a lowest weight vector of $V_1^{\otimes m}$. Let $\mathcal{I}_m$ denote the isomorphism satisfying
\begin{equation}\label{CCCC}
\mathcal{I}_m \vert 0,\ldots,0,m\rangle = \vert 0,\ldots,0,1\rangle^{\otimes m}.
\end{equation}

For generic values of $q$, the image of the projection can be written using the expression for the ground state transformation from \eqref{Gr}. This can be seen for the following reasons: The representation $V_m$ is the irreducible sub--representation of $V_1^{\otimes m}$ generated by the vector $\vert \mathrm{A}^{\otimes l} \rangle$. Therefore, the element $ \Delta^{(l)}(u_0) \vert \mathrm{A}^{\otimes m}\rangle$ is also in $V_m$. Because $V_m$ has a weight space decomposition $V_m = \oplus_{\alpha} V_m[\alpha]$, the element $ \Delta^{(m)}(u_0) \vert \mathrm{A}^{\otimes m}\rangle$ decomposes as a sum over $\alpha$, with each term in the summand also in $V_m$. Each of these terms can be computed from the ground state transformation, which is given by the coefficients of the action of $ \Delta^{(l)}(u_0)$ on $ \vert \mathrm{A}^{\otimes m}\rangle$.  More explicitly, given any $\alpha \in \mathcal{B}_m$,
\begin{equation}\label{ProjGround}
\sum_{\vec{\alpha} \in \mathcal{B}_1^{\times m}} \langle \vec{\alpha} \vert \Gr \vert \vec{\alpha} \rangle \cdot 1_{\{ \vec{\alpha}_1 + \ldots + \vec{\alpha}_m  = \alpha \}} \cdot \vert \vec{\alpha} \rangle \in V_m[\alpha] \subseteq V_m \subseteq V_1^{\otimes m}
\end{equation}

There are two expressions for the fusion that will be used here. The $R$--matrix acting on $V_l \otimes V_m$ can be written as an operator on $V_1^{\otimes l } \otimes V_1^{\otimes m}$. Then, the $R$--matrix can be determined from $R(z)$ acting on $V_1 \otimes V_1$ by

\begin{align}\label{Fusion}
R(z) \Big|_{V_l \otimes V_m} = (\mathcal{I}_l \otimes \mathcal{I}_m)^{-1} &R_{1,l+m}\left( zq^{ {m+l-2}} \right) \cdots R_{1,l+1}\left( z q^{{l-m}}\right)\notag \\ 
&R_{2,l+m}\left( zq^{ {m+l-4}} \right) \cdots R_{2,l+1}\left( z q^{{l-m-2}}\right) \notag \\
& \cdots \cdots \cdots \notag \\ 
&R_{l,l+m}\left( zq^{ {m-l}} \right) \cdots R_{l,l+1}\left( z q^{{-m-l+2}}\right) (\mathcal{I}_l \otimes \mathcal{I}_m)
\end{align}
Note that the power of $q$ decreases by $2$ in both the horizontal and vertical directions. If the $R$--matrix acting on $V_1 \otimes V_m$ has already been defined, then $R(z)$ acting on $V_l \otimes V_m$ can be written as an operator on $V_1^{\otimes l} \otimes V_m$. In this case, {\color{black}
\begin{equation}\label{Fusion2}
R(z) \Big|_{V_l \otimes V_m} = (\mathcal{I}_l \otimes \mathrm{id})^{-1} R_{1,l+1}(zq^{l-1}) \cdots R_{l,l+1}(zq^{1-l}) (\mathcal{I}_l \otimes \mathrm{id}).
\end{equation}}
The two above equations are meaningful because the fusion of $R$--matrices preserves the image of $\mathcal{I}$, in the sense that
\begin{align}\label{INT}
R_{1,l+m}(zq^{m+l-2}) \cdots R_{l,l+1}(zq^{-m-l+2}) \left( \text{Im}(\mathcal{I}_l \otimes \mathcal{I}_m)\right) &\subseteq \text{Im}(\mathcal{I}_l \otimes \mathcal{I}_m) \notag \\
R_{1,l+1}(zq^{l-1}) \cdots R_{l,l+1}(zq^{1-l}) \left( \text{Im}(\mathcal{I}_l \otimes \mathrm{id} )\right) &\subseteq \text{Im}(\mathcal{I}_l \otimes \mathrm{id})
\end{align}
This is non--trivial, and uses the fact that $R(z)$ satisfies the Yang--Baxter equation. A stronger statement holds as well: if $P^+$ denotes the projection from $V_1^{\otimes l}$ onto the sub--representation $V_l$, then 
\begin{align}\label{Stronger}
(P^+ \otimes \mathrm{id})R_{1,l+1}(zq^{1-l})\cdots R_{l,l+1}(zq^{l-1}) (P^+ \otimes \mathrm{id})&= (P^+ \otimes \mathrm{id})R_{1,l+1}(zq^{1-l})\cdots R_{l,l+1}(zq^{l-1})  \\
=(P^+ \otimes \mathrm{id})R_{1,l+1}(zq^{l-1})\cdots R_{l,l+1}(zq^{1-l}) (P^+ \otimes \mathrm{id})&=  R_{1,l+1}(zq^{l-1})\cdots R_{l,l+1}(zq^{1-l}) (P^+ \otimes \mathrm{id})\notag
\end{align}
as operators from $V_1^{\otimes l}$ to $V_l$. See equations (16)--(18) of \cite{KRL}.

\subsection{Relationship to previous results}

\subsubsection{The \cite{CGRS} framework}
In \cite{CGRS}, the authors lay out a framework for constructing interacting particle systems with duality functions from a quantum group $\mathcal{U}_q(\mathfrak{g})$ and a central element $C \in \mathcal{U}_q(\mathfrak{g})$. There is some overlap between the argument here: for example, the construction of the duality function is identical. 

Despite these similarities, there are still two differences worth noting. In \cite{CGRS}, the relevant information about the central element $C$ is that its co--product commutes with $\Delta(u)$ for any $u\in \mathcal{U}_q(\mathfrak{g})$:
$$
\Delta(C) \Delta(u) = \Delta(u) \Delta(C).
$$
By comparing with \eqref{Inter}, one can think of $\check{R}(z)$ as taking the role of $\Delta(C)$. However, in \eqref{Inter}, the maps permute the order of $V_l \otimes V_m$, which was not the case before in \cite{CGRS}.

Another difference occurs through \eqref{Trans}. In \cite{CGRS}, it is assumed that $\Delta(C)$ is self--adjoint. In \cite{K}, this assumption is weakened so that $B^{-1} \Delta(C) B$ is self--adjoint for some diagonal matrix $B$. In the situation here, $R(z)$ needs to be conjugated by a non--diagonal matrix, the charge reversal matrix, to obtain a self--adjoint operator. Note that a formula similar to \eqref{Trans} appears as (34) in \cite{PoV}. Indeed, \eqref{Trans} can be interpreted as charge--time symmetry, and is used as such in \cite{BCPS}.

\subsubsection{{\color{black} Single--species stochastic vertex model from \cite{BP,CP}}}\label{Osvm}

The stochastic matrices $\mathring{S}$ from \cite{CP} have the expression (after substituting $q$ with $q^2$)
\begin{align}\label{n=1}
\mathring{S}_{\alpha}(g,0;g,0) &= \frac{1 + \alpha q^{2g}}{1+\alpha}, \quad \quad & \mathring{S}_{\alpha}(g,0;g-1,1) &= \frac{\alpha(1-q^{2g})}{1+\alpha} \\
\mathring{S}_{\alpha}(g,1;g+1,0) &= \frac{1 - \mu^2 q^{2g}}{1+\alpha}, \quad \quad & \mathring{S}_{\alpha}(g,1;g,1) &= \frac{\alpha + \mu^2 q^{2g}}{1+\alpha}\notag
\end{align}
Here, $g$ is the number of particles at a lattice site, with either $0$ or $1$ particles entering in the auxiliary space. These are also the expressions from \cite{BP} with $\alpha=-su$ and $\mu=s$. For $\mu = q^{-m}$, $\mathring{S}_{\alpha}$ can be viewed as a stochastic operator from $V_1 \otimes V_m$ to itself. In general, it can be viewed as a stochastic operator from $V_1 \otimes \bar{V}_{\infty}$ to itself.

The fusion procedure from \cite{CP} is written in the following way. Define the matrix $\Xi$, with rows indexed by $ \{0,1\}^l $ and columns indexed by  $\{0, \ldots, l\} $, which has entries
$$
\Xi((h_1,\ldots,h_l),h) = 
\begin{cases}
1, \quad h = h_1+ \ldots + h_l, \\
0, \quad h \neq h_1+ \ldots + h_l.
\end{cases}
$$
It is immediate that $\Xi$ is a stochastic matrix. 
Define the matrix $\Lambda$, with rows indexed by $\{0, \ldots, l\}$ and columns indexed by $ \{0,1\}^l $, which has entries\footnote{In \cite{CP}, the auxiliary space is written on the right, in the sense that the operators act on $V_m \otimes V_l$ instead of $V_l \otimes V_m$. Reversing the order of the tensor products results in $\prod q^{-2y}$ in the definition of $\Lambda$, instead of $\prod q^{2y}$. It also results in the reversal of the operators $\mathring{S}$ in \eqref{FusionCP}.}
$$ 
\Lambda(h,(h_1,\ldots,h_l)) =
\begin{cases}
Z_h^{-1}  \displaystyle\prod_{y: h_y=1} q^{-2y}, \quad & h= h_1 + \ldots + h_l,\\
0, \quad  & h \neq h_1+ \ldots + h_l. 
\end{cases}
$$
The normalizing constant $Z_h$ is chosen so that $\Lambda$ is stochastic. Now identify $\{0,1\}^l$ with $\mathcal{B}_1^{(2)} \times \cdots \times \mathcal{B}_1^{(2)}$ (with $l$ products) by sending $1$ to $(1,0)$ and $0$ to $(0,1)$. Also identify  $\{0,\ldots, l\}$ with $\mathcal{B}_l^{(2)}$ by sending $ h$ to $(h,l-h).$ Since $\mathcal{B}_1^{(2)} \times \cdots \times \mathcal{B}_1^{(2)}$ is a basis of $V_1^{\otimes l}$ and $\mathcal{B}_l^{(2)}$ is a basis of $V_l$, with these identifications $\Xi$ is a stochastic operator from $V_l$ to $V_1^{\otimes l}$ and likewise $\Lambda$ is a stochastic operator from $V_1^{\otimes l}$ to $V_l$. 

When viewed as operators, the composition
\begin{equation}\label{FusionCP}
(\Lambda \otimes \mathrm{id}) \circ [\mathring{S}_{q^{2(l-1)}\alpha}]_{1,l+1} \cdots [\mathring{S}_{\alpha}]_{l,l+1} \circ (\Xi \otimes \mathrm{id})
\end{equation}
is a stochastic operator from $V_l \otimes \bar{V}_{\infty}$ to itself, and it was shown that it satisfies the Pitman--Rogers intertwining condition \cite{PR} for a map of a Markov chain to be Markov. We we see below in Theorem \ref{StochasticFusion} that this fusion process matches the one from section \ref{FusionSection}, up to the application of the gauge transformation. This will result in the statement that the $\mathcal{U}_q(A_n^{(1)})$ stochastic vertex model is a multi--species generalization of this model; see Proposition \ref{Match}.

\begin{remark}
It is not immediately obvious that the $n=1$ case reduces to the stochastic vertex model of \cite{BP,CP}. For example, it is remarked (see Remark 6.9 of \cite{BP2}) that the transformation from the non--stochastic matrix $R(z)$ to the stochastic matrix $S(z)$ uses the eigenfunctions of the transfer matrices, whereas the gauge/ground state transformation here does not require the definition of the transfer matrices. 
\end{remark}

\section{Further results about $S(z)$}\label{Further}
Before continuing on to the results concerning dualities and the transfer matrices, a few more results about $S(z)$ will be collected in the section.

\subsection{Additional Degenerations}\label{BBB}

\subsubsection{At $l=1$}\label{atl=1}
When $l=1$, the explicit expressions for $R(z)$ are given in \eqref{l=1}. After applying the gauge transformation, the resulting matrix $S(z)$ is
\begin{equation}\label{Stol=1}
(q^{m+1}-z)S(z)_{\epsilon_j,\beta}^{\epsilon_k,\delta}  = 1_{ \{\epsilon_j  + \beta = \epsilon_k + \delta \}} \times 
\begin{cases}
q^{2\beta_{[1,k]}  -  m + 1} \left( 1 - q^{-2\beta_k+m-1}z\right), & \text{ if } k = j , \\
-q^{2\beta_{[1,k-1]} - m + 1} \left( 1 - q^{2\beta_k} \right), & \text{ if } k > j , \\
-q^{2 \beta_{[1,k-1]}  } z \left( 1 - q^{2\beta_k} \right), & \text{ if } k < j .
\end{cases}
\end{equation}
Notice that in order for parallel update to be possible, the output $(\epsilon_k,\beta)$ of $S(z)$ cannot depend on the input. In other words, the expressions cannot depend on $j$, and one can quickly see that this only happens at $z=q^{1-m}$. This corresponds to the case considered in \cite{KMMO}, when the expressions are given by $\Phi_q$.

\begin{remark}\label{SimpleCase}
There are a few interesting degenerations to consider:
\begin{itemize}
\item
Taking the limit $z\rightarrow\infty$ yields 
$$
\lim_{z\rightarrow \infty} S(z)_{\epsilon_j,\beta}^{\epsilon_k,\delta} = 
\begin{cases}
q^{2\beta_{[1,k-1]}  }, & \text{ if } k = j , \\
0, & \text{ if } k > j , \\
q^{2 \beta_{[1,k-1]}  }  \left( 1 - q^{2\beta_k} \right), & \text{ if } k < j .
\end{cases}
$$
Observe that in the case when $k>j$, the corresponding entry is $0$. Furthermore, none of the entries depend on the parameter $m$. Note that the third line also describes the jump rates of the continuous--time multi--species $q$--Boson process introduced in \cite{T}.

\item Another point of interest is at $z=0$, in which case
$$
S(0)_{\epsilon_j,\beta}^{\epsilon_k,\delta} = 
\begin{cases}
q^{2\beta_{[1,k]}  -2m } , & \text{ if } k = j , \\
-q^{2\beta_{[1,k-1]}  - 2m } \left( 1 - q^{2\beta_k} \right), & \text{ if } k > j , \\
0, & \text{ if } k < j .
\end{cases}
$$
Furthermore, assuming $q>1$, then 
$$
\lim_{m\rightarrow \infty} S(0)_{\epsilon_j,\beta}^{\epsilon_k,\delta} = 
\begin{cases}
1 , & \text{ if } k =n+1  , \\
0, & \text{ if } k < n+1.
\end{cases}
$$
Theorem \ref{Stuff} will show that these two items are true for general values of $l$. 

\item Also, under the inversions $z\rightarrow z^{-1},q\rightarrow q^{-1}$ and charge reversal,
$$
(  q^{m+1} -z)S(z^{-1})_{\epsilon_j',\beta'}^{\epsilon_k',\delta'} \Big|_{q\rightarrow q^{-1}}  = 
\begin{cases}
q^{2\beta_{[1,n+2-k]}   -m+1} \left(  1 - q^{-2\beta_{n+2-k}+m-1}z\right), & \text{ if } k = j , \\
-q^{2\beta_{[1,n+1-k]} }z \left(  1 - q^{2\beta_{n+2-k}} \right), & \text{ if } k < j , \\
-q^{2 \beta_{[1,n+1-k]} -m+1 } \left(  1 - q^{2\beta_{n+2-k}}  \right), & \text{ if } k > j ,
\end{cases}
$$
which are similar expressions to the entries of $S(z)$. See Theorem \ref{Invert} for a precise statement which holds for general values of $l$. 
\item 
If the $m\rightarrow\infty$ limit is taken first, then (assuming $q>1$ and $z$ is finite)
$$
S(z)_{\epsilon_j,\beta}^{\epsilon_k,\delta} = 
\begin{cases}
1 , & \text{ if } k =n+1  , \\
0, & \text{ if } k < n+1.
\end{cases}
$$
If $q<1$, then in the $m\rightarrow\infty$ limit, the result is no longer stochastic.
\end{itemize}
\end{remark}

\subsubsection{At $m=1$}\label{SSm=1}
Now fix $m=1$ and let $l\geq 0$. By \eqref{m=1} and the gauge transformation,
\begin{equation*}
S(z)_{\alpha,\epsilon_j}^{\gamma,\epsilon_k} 
= 
\begin{cases}
q^{l-2\alpha_{[1,k-1]} +1}\dfrac{1-q^{-2\alpha_k+l-1}z}{q^{l+1}-z}, & \text{ if } k=j \\
-q^{2l-2\alpha_{[1,k]}   }   \dfrac{z(1-q^{2\alpha_k})}{q^{l+1}-z}, & \text{ if } k >j,\\
-q^{l-2\alpha_{[1,k]} +1  } \dfrac{1-q^{2\alpha_k}}{q^{l+1}-z}, & \text{ if } k<j.
\end{cases}
\end{equation*}
In order for parallel update to be possible, these expressions cannot depend on $j$, which only happens when $z=q^{-l+1}.$ This can be rewritten as
\begin{equation*}
S(z)_{\alpha,\epsilon_j}^{\gamma,\epsilon_k} 
=  \frac{q^{l-2\alpha_{[1,k-1]} +1}}{q^{l+1}-z} \times 
\begin{cases}
{1-q^{-2\alpha_k+l-1}z}, & \text{ if } k=j \\
q^{l-1  }   {z(1-q^{-2\alpha_k})}, & \text{ if } k >j,\\
 ({1-q^{-2\alpha_k}}), & \text{ if } k<j.
\end{cases}
\end{equation*}
As in the previous section, there are a few interesting degenerations. If $0<q<1$, then in the limit $z\rightarrow \infty$, the resulting stochastic matrix has entries
\begin{equation*}
\lim_{z \rightarrow \infty}S(z)_{\alpha,\epsilon_j}^{\gamma,\epsilon_k} 
= 
\begin{cases}
q^{2l-2\alpha_{[1,k]} }, & \text{ if } k=j \\
q^{2l-2\alpha_{[1,k]}   }   (1-q^{2\alpha_k}), & \text{ if } k >j,\\
0, & \text{ if } k<j.
\end{cases}
\end{equation*}
If $l$ is then taken to infinity, the limit is
\begin{equation*}
\lim_{l\rightarrow \infty} \lim_{z \rightarrow \infty}S(z)_{\alpha,\epsilon_j}^{\gamma,\epsilon_k} 
= 
\begin{cases}
1 & \text{ if } k=n+1 \\
0, & \text{ if } k<n+1.
\end{cases}
\end{equation*}

If $z=0$, then 
\begin{equation*}
S(0)_{\alpha,\epsilon_j}^{\gamma,\epsilon_k} 
=   
\begin{cases}
{q^{-2\alpha_{[1,k-1]} } }, & \text{ if } k=j \\
0, & \text{ if } k >j,\\
 q^{-2\alpha_{[1,k-1]} } ({1-q^{-2\alpha_k}}), & \text{ if } k<j.
\end{cases}
\end{equation*}
which does not depend on $l$. In the limit $l\rightarrow \infty$ and assuming $0<q<1,0<z$,
\begin{equation*}
S(z)_{\alpha,\epsilon_j}^{\gamma,\epsilon_k} 
= 
\begin{cases}
1, & \text{ if } k=j \\
0, & \text{ if } k<n+1.
\end{cases}
\end{equation*}

\subsubsection{At $z=0,z\rightarrow\infty$}
Here, we show that the examples in the previous two sections are true in general. Note that the theorem does not assume that $S(z)$ is stochastic, but it does use that the columns sum to $1$. 

\begin{theorem}\label{Stuff}
(a) When $z=z_1/z_2=0$, the matrix $S(z)$ acting on $V_l^{z_1} \otimes V_m^{z_2}$ satisfies the property that 
$
S(0)_{\alpha,\beta}^{\gamma,\delta} 
$
can only be nonzero if $\gamma_{[1,k]} \leq \alpha_{[1,k]}$ for all $1\leq k \leq n+1$. By particle conservation, an equivalent conclusion is $\beta_{[1,k]} \leq \delta_{[1,k]}$.

When $z=z_1/z_2 \rightarrow \infty$, the matrix $S(z)$ acting on $V_l^{z_1} \otimes V_m^{z_2}$ satisfies the property that 
$
S(\infty)_{\alpha,\beta}^{\gamma,\delta} 
$
can only be nonzero if $\delta_{[1,k]} \leq \beta_{[1,k]}$ for all $1\leq k \leq n+1$. By particle conservation, an equivalent conclusion is $\alpha_{[1,k]} \leq \gamma_{[1,k]}$.

(b) When $z\rightarrow \infty$, the matrix $S(z)$ acting on $V_l \otimes V_m$ does not depend on $m$, in the following sense: If $\beta,\delta \in \mathcal{B}_m$ and $\beta^+,\delta^+ \in \mathcal{B}_{m^+}$ satisfy $\overline{\beta}=\overline{\beta^+},\overline{\delta} = \overline{\delta^+}$, then 
$$
S(z)_{\alpha,\beta}^{\gamma,\delta} = S(z)_{\alpha,\overline{\beta}}^{\gamma,\overline{\delta}}.
$$
Similarly, when $z\rightarrow 0$, the matrix $S(z)$ acting on $V_l \otimes V_m$ does not depend on $l$.

(c) If $q > 1 $, then in the limit $m\rightarrow\infty$,
$$
\lim_{m \rightarrow\infty} S(0)_{\alpha,\beta}^{\gamma,\delta} = 
\begin{cases}
1, &\gamma = \Omega, \\
0, &\gamma \neq \Omega.
\end{cases}
$$
If $q  < 1 $, then in the limit $l\rightarrow\infty$,
$$
\lim_{l \rightarrow\infty} S(\infty)_{\alpha,\beta}^{\gamma,\delta} = 
\begin{cases}
1, &\delta = \Omega, \\
0, &\delta \neq \Omega.
\end{cases}
$$
These limits are taken in the sense described in Section \ref{RT}.
\end{theorem}
\begin{proof}
(a) By fusion \eqref{Fusion}, the $R$--matrix $R(z)$ acting on $V_l^{z_1} \otimes V_m^{z_2}$ can be written as a product of $R$--matrices acting on $V_1\otimes V_1$. From the explicit formula for the $l=m=1$ case \eqref{m=1}, the relevant off--diagonal entries become $0$ when $z\rightarrow 0$ or $z\rightarrow \infty$. This implies (a). 

(b) We prove the first statement, with the proof of the second statement being similar.

Let $\beta,\delta \in \mathcal{B}_m$ and $\beta^+,\delta^+ \in \mathcal{B}_{m^+}$ satisfy $\overline{\beta}=\overline{\beta^+},\overline{\delta} = \overline{\delta^+}$, as in the statement of the theorem. By the explicit formula for the gauge transformation,
\begin{align*}
\Ga_{\alpha\beta^+}^{\alpha\beta^+} &= q^{-(m^+-m)\vert \alpha\vert}\Ga_{\alpha\beta}^{\alpha\beta} ,\\
\widetilde{\Ga}_{\gamma\delta^+}^{\gamma\delta^+} &= \widetilde{\Ga}_{\gamma\delta}^{\gamma\delta}.
\end{align*}
Therefore it is equivalent to show that
\begin{equation}\label{Equiv}
R(\infty)_{\alpha, \beta}^{\gamma,\delta} = q^{-(m^+-m)\vert \alpha\vert}R(\infty)_{\alpha, \beta^+}^{\gamma,\delta^+} .
\end{equation}
The proof of \eqref{Equiv} will be a strong induction argument on $(\vert \alpha\vert, n)$. Define the total ordering $\leq$ by
$$
(l,1) \leq (l-1,1) \leq \ldots \leq (0,1) \leq (l,2) \leq \ldots \leq (0,2) \leq (l,3) \leq \ldots . 
$$

In the base case when $\vert \alpha\vert =l$ and $n=1$, the input $\alpha$ must equal $(l,0)$. By part (a), $S(\infty)_{\alpha\beta}^{\gamma\delta}$ is only nonzero for $\gamma = (l,0)$. Since the columns of $S(\infty)$ sum to one, 
$$
S(\infty)_{\alpha\beta}^{\gamma\delta} 
= 
\begin{cases}
1, \quad \gamma=\alpha, \\
0, \quad \gamma \neq \alpha.
\end{cases}
$$
This does not depend on $m$.

For the inductive step, there are the two cases when $\vert \alpha \vert =l$ and $\vert \alpha \vert \neq l$. Start with $\vert \alpha \vert =l$. First we show a preliminary identity: if  $\beta_{\text{red}}$ denotes the ``reduced'' $(\beta_1,\ldots,\beta_{n-1},\beta_{n}+\beta_{n+1},0)$, then
\begin{align}\label{Reduced}
R(\infty)_{\alpha,\beta}^{\gamma,\delta} &= q^{\alpha_n  \beta_{n+1}  }R(\infty)_{\alpha,\beta_{\text{red}}}^{\gamma,\delta_{\text{red}}}  \text{ assuming } \alpha_{n+1}=0   \\ 
R(\infty)_{\alpha, \beta^+}^{\gamma,\delta^+} &= q^{\alpha_n(\beta_{n+1}+m^+-m)}R(\infty)_{\alpha,\beta^+_{\text{red}}}^{\gamma,\delta^+_{\text{red}}}  \text{ assuming } \alpha_{n+1}=0 \notag
\end{align}
The proof of this identity uses \eqref{Inter}. By \eqref{ACT},
$$
\check{R}(z) \left( f_n \otimes 1 + k_n^{-1} \otimes f_n \right) \vert \alpha,\beta \rangle = \left( f_n \otimes 1 + k_n^{-1} \otimes f_n \right) \check{R}(z) \vert \alpha,\beta \rangle .
$$
This simplifies to 
$$
\check{R}(z) \left( [\alpha_{n+1}]_q\vert \alpha + \hat{n}, \beta\rangle + q^{\alpha_n - \alpha_{n+1}} [ \beta_{n+1}]_q\vert \alpha, \beta + \hat{n} \right) = \sum_{\delta, \gamma}  \left( f_n \otimes 1 + k_n^{-1} \otimes f_n \right) R(z)_{\alpha\beta}^{\gamma\delta} \vert \delta,\gamma \rangle,
$$
which furthermore simplifies to 
\begin{multline*}
\sum_{\delta, \gamma} \left( [\alpha_{n+1}]_q  R(z) _{\alpha+\hat{n},\beta}^{\gamma\delta} + q^{\alpha_n - \alpha_{n+1}} [ \beta_{n+1}]_q  R(z)_{\alpha,\beta+\hat{n}}^{\gamma\delta} \right) \vert \delta, \gamma \rangle \\
= \sum_{\delta,\gamma} R(z)_{\alpha\beta}^{\gamma\delta} \left( [\delta_{n+1}]_q \vert \delta + \hat{n},\gamma \rangle + q^{\delta_n-\delta_{n+1}}[\gamma_{n+1}]_q \vert \delta, \gamma + \hat{n} \rangle \right)
\end{multline*}
Because $\alpha_{n+1}=0$ by assumption, part (a) implies that $R(z)_{\alpha\beta}^{\gamma,\delta- \hat{n}}$ can only be nonzero if $\gamma_{n+1}= 0$. Therefore,
$$
q^{\alpha_n - \alpha_{n+1}} [ \beta_{n+1}]_q  R(z)_{\alpha,\beta+\hat{n}}^{\gamma\delta} = R(z)_{\alpha\beta}^{\gamma,\delta - \hat{n}}  [\delta_{n+1}+1]_q 
$$
But then by particle conservation, $\delta_{n+1} = \beta_{n+1}-1$, implying that
$$
q^{\alpha_n }   R(\infty)_{\alpha,\beta+\hat{n}}^{\gamma\delta} = R(\infty)_{\alpha\beta}^{\gamma,\delta - \hat{n}}  , 
$$
or equivalently,
$$
R(\infty)_{\alpha,\beta }^{\gamma,\delta}   =q^{\alpha_n }   R(\infty)_{\alpha,\beta + \hat{n}}^{\gamma,\delta + \hat{n}} ,
$$
By repeatedly applying this last identity,
$$
R(\infty)_{\alpha, \beta}^{\gamma,\delta} = q^{\alpha_n\beta_{n+1}}R(\infty)_{\alpha,\beta_{\text{red}}}^{\gamma,\delta_{\text{red}}}.
$$
But this also holds for $\beta^+$ and $\delta^+$, and $\beta_{n+1}^+ = \beta_{n+1} + m^+ - m$, so \eqref{Reduced} holds.

Observe now that since we just showed that $\gamma_{n+1}=0$ for nonzero entries, therefore the right--hand--side of \eqref{Reduced} only contains basis elements $\alpha$ for which $\alpha_{n+1}=0$. By the strong induction hypothesis, 
$$
R(\infty)_{\alpha, \beta_{\text{red}}}^{\gamma,\delta_{\text{red}}} = q^{-(m^+-m) (\alpha_1 + \ldots + \alpha_{n-1})}R(\infty)_{\alpha, \beta_{\text{red}}^+}^{\gamma,\delta_{\text{red}}^+} .
$$
Therefore
\begin{align*}
R(\infty)_{\alpha,\beta}^{\gamma,\delta} &= q^{\alpha_n  \beta_{n+1} }R(\infty)_{\alpha,\beta_{\text{red}}}^{\gamma,\delta_{\text{red}}}   \\
&  = q^{\alpha_n  \beta_{n+1} }q^{-(m^+-m) (\alpha_1 + \ldots + \alpha_{n-1})} R(\infty)_{\alpha, \beta_{\text{red}}^+}^{\gamma,\delta_{\text{red}}^+} \\
&= q^{\alpha_n  \beta_{n+1} }q^{-(m^+-m) (\alpha_1 + \ldots + \alpha_{n-1})}  q^{-\alpha_n(\beta_{n+1}+m^+-m)}R(\infty)_{\alpha,\beta^+_{\text{}}}^{\gamma,\delta^+_{\text{}}} \\
&= q^{- (m^+-m)\vert \alpha \vert}R(\infty)_{\alpha,\beta^+_{\text{}}}^{\gamma,\delta^+_{\text{}}}. 
\end{align*}
This completes the inductive step when $\vert \alpha \vert =l$.

Now turn to the inductive step when $\vert \alpha \vert <l$. By \eqref{Inter} and \eqref{ACT},
$$
\check{R}(z) \left( 1 \otimes e_n   +  e_n \otimes k_n\right) \vert \alpha,\beta \rangle = \left( 1 \otimes e_n  + e_n \otimes  k_n \right) \check{R}(z) \vert \alpha,\beta \rangle .
$$
This simplifies to 
$$
\check{R}(z) \left( [\beta_{n}]_q\vert \alpha , \beta - \hat{n}\rangle + q^{\beta_{n+1} - \beta_{n}} [ \alpha_{n}]_q\vert \alpha - \hat{n}, \beta  \right) = \sum_{\delta, \gamma} \left( 1 \otimes e_n  + e_n \otimes  k_n \right)  R(z)_{\alpha\beta}^{\gamma\delta} \vert \delta,\gamma \rangle,
$$
which furthermore simplifies to 
\begin{multline*}
\sum_{\delta, \gamma} \left( [\beta_{n}]_q  R(z) _{\alpha,\beta - \hat{n}}^{\gamma\delta} + q^{\beta_{n+1} - \beta_{n}} [ \alpha_{n}]_q  R(z)_{\alpha - \hat{n}, \beta}^{\gamma\delta} \right) \vert \delta, \gamma \rangle \\
= \sum_{\delta,\gamma} R(z)_{\alpha\beta}^{\gamma\delta} \left( [\gamma_{n}]_q \vert \delta,\gamma - \hat{n} \rangle + q^{\gamma_{n+1}-\gamma_{n}}[\delta_{n}]_q \vert \delta - \hat{n}, \gamma  \rangle \right)
\end{multline*}
Therefore
\begin{equation*}
[\beta_n]_q R(z) _{\alpha,\beta - \hat{n}}^{\gamma\delta} + q^{\beta_{n+1} - \beta_{n}} [ \alpha_{n}]_q  R(z)_{\alpha - \hat{n}, \beta}^{\gamma\delta} = R(z)_{\alpha\beta}^{\gamma + \hat{n}, \delta}  [ \gamma_{n}+1]_q + R(z)_{\alpha\beta}^{\gamma,\delta+\check{n}}q^{\gamma_{n+1}-\gamma_{n}}[\delta_{n}+1]_q  .
\end{equation*}
This can be rewritten as
\begin{equation}\label{Key}
q^{\beta_{n+1} }   R(z)_{\alpha, \beta}^{\gamma\delta} = \frac{-[\beta_n]_q R(z) _{\alpha+\hat{n},\beta - \hat{n}}^{\gamma\delta} + R(z)_{\alpha+\hat{n},\beta}^{\gamma + \hat{n}, \delta}  [ \gamma_{n}+1]_q + R(z)_{\alpha+\hat{n},\beta}^{\gamma,\delta+\check{n}}q^{\gamma_{n+1}-\gamma_{n}}[\delta_{n}+1]_q}{q^{- \beta_{n}}[ \alpha_{n}+1]_q}  .
\end{equation}
The $\alpha+\hat{n}$ on the right--hand--side are well--defined because by assumption $\alpha_{n+1}>0$. By the strong induction hypothesis, 
$$
q^{\beta_{n+1} }   R(z)_{\alpha, \beta}^{\gamma\delta} = q^{-(m^+-m)\vert \alpha + \hat{n}\vert} \frac{-[\beta_n]_q R(z) _{\alpha+\hat{n},\beta^+ - \hat{n}}^{\gamma\delta^+} + R(z)_{\alpha+\hat{n},\beta^+}^{\gamma + \hat{n}, \delta^+}  [ \gamma_{n}+1]_q + R(z)_{\alpha+\hat{n},\beta^+}^{\gamma,\delta^++\check{n}}q^{\gamma_{n+1}-\gamma_{n}}[\delta_{n}+1]_q}{q^{- \beta_{n}}[ \alpha_{n}+1]_q}  .
$$
At the same time, applying \eqref{Key} to $\beta^+,\delta^+$ shows that
$$
q^{\beta_{n+1} + m^+ - m }   R(z)_{\alpha, \beta^+}^{\gamma\delta^+} = \frac{-[\beta_n]_q R(z) _{\alpha+\hat{n},\beta^+ - \hat{n}}^{\gamma\delta^+} + R(z)_{\alpha+\hat{n},\beta^+}^{\gamma + \hat{n}, \delta^+}  [ \gamma_{n}+1]_q + R(z)_{\alpha+\hat{n},\beta^+}^{\gamma,\delta^++\check{n}}q^{\gamma_{n+1}-\gamma_{n}}[\delta_{n}+1]_q}{q^{- \beta_{n}}[ \alpha_{n}+1]_q}  .
$$
Comparing the last two equalities shows
$$
q^{\beta_{n+1} }   R(z)_{\alpha, \beta}^{\gamma\delta} = q^{-(m^+-m)\vert \alpha + \hat{n}\vert} q^{\beta_{n+1} + m^+ - m } R(z)_{\alpha, \beta^+}^{\gamma\delta^+} ,
$$
which simplifies to 
$$
 R(z)_{\alpha, \beta}^{\gamma\delta} = q^{-(m^+-m)\vert \alpha  \vert }R(z)_{\alpha, \beta^+}^{\gamma\delta^+} .
$$
This completes the inductive step and the proof of (b).

(c) Consider the case when $z=0,q >1$ and $m\rightarrow\infty$, since the other case is similar. 

Use the expression \eqref{Fusion2} for fusion. At $z=0$ in \eqref{l=1}, the largest asymptotic contribution occurs when $k=n+1>j$, when the contribution is $q^m$. This contribution will happen $\vert \alpha\vert - \vert \gamma\vert $ times, for  a total of $q^{m( \vert \alpha \vert - \vert\gamma\vert)}$ asymptotically. The gauge transform multiplies the $R$--matrix by 
$$
q^{-\sum_{i=1}^n \alpha_{[1,i]}\beta_{i+1}} \cdot q^{\sum_{i=1}^n \delta_{[1,i]}\gamma_{i+1}},
$$
which is asymptotically $q^{-m\vert\alpha\vert}$. This gives a total of $q^{-m \vert \gamma\vert}$, which is zero for any $\gamma \neq \Omega$. Because the columns sum to $1$, the entry must be $1$ for $\gamma = \Omega$.

\end{proof}

\begin{remark}
If the limits $l,m\rightarrow \infty$ are taken for $0<z<\infty$, there is no guarantee that the result would be stochastic for every value of $q$. In particular, the limit of the entries need not be bounded.

\end{remark}

\subsection{Stochastic fusion}
In this section, we generalize the fusion procedure of \cite{CP} described in Section \ref{Osvm} to the multi--species case, and show that it matches the fusion of \cite{KRL} described in Section \ref{FusionSection}, after applying the gauge transformation.

Define an order on $\mathcal{B}_1$ by $\epsilon_1 < \epsilon_2 < \ldots < \epsilon_{n+1}$. Given $\vec{\alpha} = (\vec{\alpha}_{1}, \ldots, \vec{\alpha}_{l}) \in \mathcal{B}_1^{ \times l}$, let 
$$
E_{\vec{\alpha}} = \vert \{  (r,s)   : 1 \leq r < s \leq l \text{ and } \vec{\alpha}_r > \vec{\alpha}_s    \} \vert.
$$
Note that $E_{\vec{\alpha}}$ is related to the ground state transformation by 
\begin{equation}\label{Gre}
\langle \vec{\alpha} \vert \Gr \vert \vec{\alpha} \rangle =  \textrm{const} \cdot q^{-E_{\vec{\alpha}}}
\end{equation}
Define for any $n\geq 1$, the operator $\Xi^{(n)}$ from $V_l$ to $V_1^{\otimes l}$ by 
$$
\langle \vec{\alpha}_1, \ldots, \vec{\alpha}_l  \vert \Xi^{(n)} \vert \alpha \rangle = 
\begin{cases}
Z_{\alpha}^{-1} q^{-2E_{\vec{\alpha}}} , &\text{ if } \alpha = \vec{\alpha}_1 + \ldots + \vec{\alpha}_l, \\
0,& \text{ if } \alpha \neq \vec{\alpha}_1 + \ldots + \vec{\alpha}_l .
\end{cases}
$$
Here, $\vec{\alpha} = (\vec{\alpha}_1, \ldots, \vec{\alpha}_l)$ and $Z_{\alpha}^{-1}$ is a normalization constant chosen so that $\Xi^{(n)}$ is stochastic. Also define the operator $\Lambda^{(n)}$ from $V_1^{\otimes l}$ to $V_l$ by 
$$
\Lambda^{(n)} \vert \vec{\alpha}_1, \ldots, \vec{\alpha}_l \rangle  
= 
\begin{cases}
\vert \alpha \rangle, &\text{ if } \alpha = \vec{\alpha}_1 + \ldots + \vec{\alpha}_{l}, \\
0,& \text{ if } \alpha \neq \vec{\alpha}_1 + \ldots + \vec{\alpha}_{l} .
\end{cases}
$$
It is immediate that $\Lambda^{(n)}$ is stochastic. It is also straightforward that when $n=1$, the definitions of $\Xi^{(1)}$ and $\Lambda^{(1)}$ match that of \ref{Osvm}.

Part (a) of the next theorem shows that the $\mathcal{U}_q(A_n^{(1)})$ stochastic vertex model is a $n$--species generalization of the stochastic vertex models of \cite{BP,CP}. Part (b) generalizes Proposition 3.6 of \cite{CP}, which showed that the fused $S(z)$ matrix satisfies the Rogers--Pitman intertwining described in section \ref{LumpSection}.
\begin{theorem}\label{StochasticFusion}
(a) The $S(z)$ matrix acting on $V_l \otimes V_m$ can be written as
\begin{equation}\label{TWOSIDES}
S(z) = (\Lambda^{(n)} \otimes \mathrm{id} ) \circ S_{1,l+1}(zq^{l-1}) \cdots S_{l,l+1}(zq^{1-l}) \circ (\Xi^{(n)} \otimes \mathrm{id}).
\end{equation}

(b) The composition $ \Lambda^{(n)} \circ \Xi^{(n)} $ is the identity on $V_l$. As operators from $V_1^{\otimes l}\otimes V_m$ to $V_l \otimes V_m$,
\begin{multline}\label{RPCP}
(\Lambda^{(n)} \otimes \mathrm{id} ) \circ S_{1,l+1}(zq^{l-1}) \cdots S_{l,l+1}(zq^{1-l}) \circ (\Xi^{(n)}\Lambda^{(n)}  \otimes \mathrm{id}) \\
= (\Lambda^{(n)} \otimes \mathrm{id} ) \circ S_{1,l+1}(zq^{l-1}) \cdots S_{l,l+1}(zq^{1-l})
\end{multline}

\end{theorem}

\begin{proof} (a) Begin by analyzing the right--hand--side of \eqref{TWOSIDES}. Given $\vec{\alpha}$ and $\vec{\gamma}$ in $\mathcal{B}_1^{\times l}$, set 
$$
\vec{\alpha}^{(r)}_s 
= 
\begin{cases}
\vec{\alpha}_s,  \quad s \leq r\\
\vec{\gamma}_s, \quad s>r.
\end{cases}
$$
Further define $\beta^{(r)} \in \mathcal{B}_m$ by 
$$
\beta^{(r)} + \vec{\alpha}_1^{(r)} + \ldots + \vec{\alpha}_r^{(r)} = \beta + \vec{\alpha}_1 + \ldots + \vec{\alpha}_r .
$$
In words, this says that $\vec{\alpha}^{(r)},\beta^{(r)}$ together have the same number of each species of particles as $\vec{\alpha},\beta$. Then by particle conservation, 
\begin{align*}
&\langle \vec{\gamma}, \delta \vert S_{1,l+1}(zq^{l-1}) \cdots S_{l,l+1}(zq^{1-l}) \vert \vec{\alpha}, \beta \rangle \\
&=\langle \vec{\gamma},\delta \vert S_{1,l+1}(z q^{l-1}) \vert \vec{\alpha}^{(1)},\beta^{(1)} \rangle \cdots  \langle \vec{\alpha}^{(l-1)}, \vec{\beta}^{(l-1)} \vert S_{l,l+1}(zq^{1-l}) \vert \vec{\alpha},\beta\rangle \\
&=\langle \vec{\gamma},\delta \vert \widetilde{\Ga}^{-1}_{1,l+1} R_{1,l+1}(z q^{l-1}) \Ga_{1,l+1} \vert \vec{\alpha}^{(1)},\beta^{(1)} \rangle \cdots  \langle \vec{\alpha}^{(l-1)}, \vec{\beta}^{(l-1)} \vert \widetilde{\Ga}^{-1}_{l,l+1} R_{l,l+1}(zq^{1-l})\Ga_{l,l+1} \vert \vec{\alpha},\beta\rangle \\
&= \langle \vec{\gamma},\delta \vert \widetilde{\Ga}_{1,l+1}^{-1} \vert \vec{\gamma},\delta \rangle   \langle \vec{\alpha}^{(1)},\beta^{(1)} \vert \Ga_{1,l+1} \vert \vec{\alpha}^{(1)},\beta^{(1)} \rangle \cdots  \langle \vec{\alpha}^{(l-1)}, \vec{\beta}^{(l-1)} \vert \widetilde{\Ga}^{-1}_{l,l+1} \vert \vec{\alpha}^{(l-1)}, \vec{\beta}^{(l-1)} \rangle \langle \vec{\alpha},\beta \vert  \Ga_{l,l+1} \vert \vec{\alpha},\beta\rangle \\
& \quad \quad \times   \langle \vec{\gamma},\delta \vert  R_{1,l+1}(z q^{l-1})  \vert \vec{\alpha}^{(1)},\beta^{(1)} \rangle \cdots  \langle \vec{\alpha}^{(l-1)}, \vec{\beta}^{(l-1)} \vert R_{l,l+1}(zq^{1-l}) \vert \vec{\alpha},\beta\rangle
\end{align*}
Now we show that 
\begin{lemma}\label{ABCDE}
With the definitions above, 
\begin{align*}
\langle \vec{\gamma},\delta \vert \widetilde{\Ga}^{-1}_{1,l+1} \vert \vec{\gamma},\delta \rangle  \cdots  \langle \vec{\alpha}^{(l-1)}, \vec{\beta}^{(l-1)} \vert \widetilde{\Ga}^{-1}_{l,l+1} \vert \vec{\alpha}^{(l-1)}, \vec{\beta}^{(l-1)} \rangle &= q^{\mathcal{N}_1(\vec{\alpha},\vec{\gamma}) }  \langle {\gamma},\delta \vert \widetilde{\Ga}^{-1} \vert {\gamma},\delta \rangle,\\
\langle \vec{\alpha}^{(1)},\beta^{(1)} \vert \Ga_{1,l+1} \vert \vec{\alpha}^{(1)},\beta^{(1)} \rangle \cdots   \langle \vec{\alpha},\beta \vert  \Ga_{l,l+1} \vert \vec{\alpha},\beta\rangle 
&= q^{\mathcal{N}_2(\vec{\alpha},\vec{\gamma}) }   \langle {\alpha},\beta \vert  \Ga \vert {\alpha},\beta\rangle.
\end{align*}
where 
$$
\mathcal{N}_1(\vec{\alpha},\vec{\gamma}) :=  | \{ (r,s): 1\leq r<s \leq l \text{ and } \vec{\gamma}_r < \vec{\gamma}_s \}| - | \{  (r,s): 1\leq r<s \leq l \text{ and }  \vec{\alpha}_r < \vec{\gamma}_s\} |,
$$
and 
$$
\mathcal{N}_2(\vec{\alpha},\vec{\gamma}) := - | \{ (r,s): 1\leq r<s \leq l \text{ and } \vec{\alpha}_s  > \vec{\alpha}_r  \} | +  | \{  (r,s): 1\leq r<s \leq l \text{ and } \vec{\gamma}_s > \vec{\alpha}_r\ \} |  .
$$
\end{lemma}
{\color{black}
\begin{proof}
Using that
$$
\langle {\alpha},\beta \vert  \Ga \vert {\alpha},\beta\rangle = \langle \vec{\alpha},\beta \vert  \Ga_{1,l+1} \vert \vec{\alpha},\beta\rangle \cdots \langle \vec{\alpha},\beta \vert  \Ga_{l,l+1} \vert \vec{\alpha},\beta\rangle,
$$
it suffices to show that 
$$
\frac{    \langle \vec{\alpha},\beta \vert  \Ga_{1,l+1} \vert \vec{\alpha},\beta\rangle \cdots \langle \vec{\alpha},\beta \vert  \Ga_{l,l+1} \vert \vec{\alpha},\beta\rangle     }{ \langle \vec{\alpha}^{(1)},\beta^{(1)} \vert \Ga_{1,l+1} \vert \vec{\alpha}^{(1)},\beta^{(1)} \rangle  \cdots \langle \vec{\alpha},\beta \vert  \Ga_{l,l+1} \vert \vec{\alpha},\beta\rangle }  = q^{-\mathcal{N}_2(\vec{\alpha},\vec{\gamma})}.
$$
If $i_r$ is defined by $\vec{\alpha}_r = \epsilon_{i_r}$, then 
$$
\frac{ \langle \vec{\alpha}, \beta \vert \Ga_{r,l+1} \vert \vec{\alpha},\beta \rangle }{ \langle \vec{\alpha}^{(r)},\beta^{(r)} \vert  \Ga_{r,l+1} \vert \vec{\alpha}^{(r)},\beta^{(r)}\rangle} = q^{-   \beta_{[i_r+1,n+1]}  +    \beta^{(r)}_{[i_r+1,n+1]}   }.
$$
By particle conservation, and the fact that $  \vec{\alpha}^{(r)}_{s} = \vec{\alpha}_{s} $ for $s\leq r$, 
\begin{multline*}
\beta^{(r)}_{[i_r+1,n+1]}  + | \{ (r,s): 1 \leq r <s \leq l \text{ and } \vec{\alpha}^{(r)}_s  \in \{\epsilon_{i_r+1} , \ldots , \epsilon_{n+1}\} \  \} | \\
=  \beta_{[i_r+1,n+1]} + | \{ (r,s): 1 \leq r<s \leq l \text{ and } \vec{\alpha}_s  \in \{\epsilon_{i_r+1} , \ldots , \epsilon_{n+1}\} \  \} |.
\end{multline*}
Substituting that $\vec{\gamma}_s = \vec{\alpha}_s^{(r)}$ for $s>r,$
\begin{align*}
 \beta^{(r)}_{[i_r+1,n+1]} - \beta_{[i_r+1,n+1]} &=  | \{ (r,s): 1 \leq r<s \leq l \text{ and }  \vec{\alpha}_s  > \vec{\alpha}_r  \} | -  | \{ (r,s):1 \leq r<s \leq l \text{ and } \vec{\gamma}_s > \vec{\alpha}_r\  \} | 
\end{align*}
Therefore, 
$$
\frac{    \langle \vec{\alpha},\beta \vert  \Ga_{1,l+1} \vert \vec{\alpha},\beta\rangle \cdots \langle \vec{\alpha},\beta \vert  \Ga_{l,l+1} \vert \vec{\alpha},\beta\rangle     }{ \langle \vec{\alpha}^{(1)},\beta^{(1)} \vert \Ga_{1,l+1} \vert \vec{\alpha}^{(1)},\beta^{(1)} \rangle  \cdots \langle \vec{\alpha},\beta \vert  \Ga_{l,l+1} \vert \vec{\alpha},\beta\rangle }  = \prod_{r=1}^l q^{ \beta^{(r)}_{[i_r+1,n+1]} - \beta_{[i_r+1,n+1]}}  = q^{-\mathcal{N}_2(\vec{\alpha},\vec{\gamma})},
$$
as needed. 

A similar argument holds for $\widetilde{\Ga}$. Similarly, if $j_s$ is defined by $\vec{\gamma}_s = \epsilon_{j_s}$ then 
$$
\frac{\langle \vec{\gamma}, \delta \vert \widetilde{\Ga}_{s'1} \vert \vec{\gamma},\delta \rangle }{ \langle \vec{\alpha}^{(s-1)},\beta^{(s-1)} \vert  \widetilde{\Ga}_{s'1} \vert \vec{\alpha}^{(s-1)},\beta^{(s-1)}\rangle} = q^{ \beta^{(s-1)}_{[1,j_s-1]}  - \delta_{[1,j_s-1]} }.
$$
By particle conservation, and using that $\vec{\alpha}^{(s-1)}_r = \vec{\gamma}_r$ for $r \geq s$, and $\vec{\alpha}^{(s-1)}_r = \vec{\alpha}_r$ for $r<s,$
\begin{align*}
 \beta^{(s-1)}_{[1,j_s-1]}  - \delta_{[1,j_s-1]}  &= - | \{ (r,s):1 \leq r<s \leq l \text{ and }  \vec{\alpha}^{(s-1)}_r \in \{\epsilon_1, \ldots, \epsilon_{j_s-1}\} \} | \\
 & \quad \quad \quad + | \{(r,s): 1 \leq r<s \leq l \text{ and }  \vec{\gamma}_r  \in \{\epsilon_1, \ldots, \epsilon_{j_s-1}\} \} | \\
&= | \{ (r,s): 1 \leq r<s \leq l \text{ and } \vec{\gamma}_r < \vec{\gamma}_s \}| - | \{ (r,s): 1 \leq r<s \leq l \text{ and }  \vec{\alpha}_r < \vec{\gamma}_s\} | 
\end{align*}
so that
$$
\frac{\langle {\gamma},\delta \vert  \widetilde{\Ga} \vert {\gamma},\delta \rangle}{ \langle \vec{\gamma}, \delta  \vert \widetilde{\Ga}_{1,l+1} \vert \vec{\gamma}, \delta\rangle  \cdots \langle \vec{\alpha}^{(l-1)},\beta^{(l-1)}  \vert  \widetilde{\Ga}_{l,l+1} \vert \vec{\alpha}^{(l-1)},\beta^{(l-1)} \rangle } = q^{ \mathcal{N}_1(\vec{\alpha},\vec{\gamma})}.
$$
This finishes the proof of the lemma.
\end{proof}
}

Therefore, the lemma implies that
\begin{multline*}
 \langle {\gamma},\delta \vert \text{RHS of } \eqref{TWOSIDES} \vert \alpha,\beta \rangle= \langle {\gamma},\delta \vert \widetilde{\Ga}^{-1} \vert {\gamma},\delta \rangle\langle \alpha ,\beta \vert  \Ga \vert {\alpha},\beta\rangle \\
\times \sum_{\vec{\alpha},\vec{\gamma}}  q^{\mathcal{N}_1(\vec{\alpha},\vec{\gamma}) + \mathcal{N}_2(\vec{\alpha},\vec{\gamma})} \langle\gamma \vert \Lambda^{(n)} \vert \vec{\gamma} \rangle \langle \vec{\gamma},\delta \vert  R_{1,l+1}(z q^{l-1})  \vert \vec{\alpha}^{(1)},\beta^{(1)} \rangle \cdots  \langle \vec{\alpha}^{(l-1)}, \vec{\beta}^{(l-1)} \vert R_{l,l+1}(zq^{1-l}) \vert \vec{\alpha},\beta\rangle \langle \vec{\alpha} \vert \Xi^{(n)} \vert \alpha\rangle.
\end{multline*}
At the same time, by the fusion described in Section \ref{FusionSection}, 
\begin{multline*}
 \langle {\gamma},\delta \vert  \text{LHS of } \eqref{TWOSIDES} \vert \alpha,\beta\rangle = \langle {\gamma},\delta \vert \widetilde{\Ga}^{-1} \vert {\gamma},\delta \rangle\langle \alpha ,\beta \vert  \Ga \vert {\alpha},\beta\rangle \\
\times  \sum_{\vec{\alpha}} \langle \gamma, \delta \vert ( \mathcal{I}_l \otimes \mathrm{id} )^{-1}  R_{1,l+1}(zq^{l-1}) \cdots R_{l,l+1}(zq^{1-l}) \vert \vec{\alpha}, \beta\rangle  \langle \vec{\alpha}, \beta \vert ( \mathcal{I}_l \otimes \mathrm{id}) \vert \alpha, \beta \rangle. 
\end{multline*}
Since
$$
\mathcal{N}_1(\vec{\alpha},\vec{\gamma}) + \mathcal{N}_2(\vec{\alpha},\vec{\gamma}) = | \{ (r,s):1 \leq r<s \leq l \text{ and } \vec{\gamma}_r < \vec{\gamma}_s \}| - | \{ (r,s):1 \leq r<s \leq l \text{ and } \vec{\alpha}_s > \vec{\alpha}_r \}
$$
the theorem will now follow from two identities. The first is
$$
q^{- | \{ (r,s):1 \leq r<s \leq l \text{ and }   \vec{\alpha}_s > \vec{\alpha}_r \}| }  \langle \vec{\alpha}  \vert \Xi^{(n)} \vert \alpha \rangle = \langle \vec{\alpha} \vert  \mathcal{I}_l \vert \alpha \rangle.
$$
After proving this first identity, it will remain to prove the second identity
\begin{equation}\label{SECIND}
\Lambda^{(n)} \mathfrak{d} \Big|_{\text{Im}( \mathcal{I}_l)}= \mathcal{I}_l^{-1},
\end{equation}
where $\mathfrak{d}$ is the diagonal operator on $V_1^{\otimes l}$ with entries
$$
\langle \vec{\gamma} \vert \mathfrak{d} \vert \vec{\gamma} \rangle = q^{  | \{ (r,s):1 \leq r<s \leq l \text{ and } \vec{\gamma}_r < \vec{\gamma}_s | \}}.
$$
Recall that because of \eqref{INT}, it suffices to restrict the domain of $\Lambda^{(n)}\mathfrak{d}$ to the image of $\mathcal{I}_l$.

To prove the first identity, note that 
$$
q^{- | \{ (r,s):1 \leq r<s \leq l \text{ and }   \vec{\alpha}_s > \vec{\alpha}_r \}| }  \langle \vec{\alpha}  \vert \Xi^{(n)} \vert \alpha \rangle = 
\begin{cases}
Z_{\alpha}^{-1}  q^{C_{\alpha}}q^{-E_{\vec{\alpha}}}, & \text{ if }  \alpha = \vec{\alpha}_1 + \ldots + \vec{\alpha}_r, \\
0,  & \text{ if } \alpha \neq \vec{\alpha}_1 + \ldots + \vec{\alpha}_r 
\end{cases}
$$
where $C_{\alpha}$ is the constant, only depending on $\alpha$ (and not on $\vec{\alpha}$), defined by 
\begin{align*}
&- | \{ (r,s): 1\leq r < s \leq l \text{ and } \vec{\alpha}_s > \vec{\alpha}_r \}| - 2 | \{ (r,s): 1\leq r <s \leq l \text{ and }  \vec{\alpha}_r > \vec{\alpha}_s \}| \\
&= - \binom{l}{2} + | \{ (r,s): 1 \leq r<s \leq l \text{ and } \vec{\alpha}_r = \vec{\alpha}_s \}|  -  | \{ (r,s): 1 \leq r<s \leq l \text{ and }  \vec{\alpha}_r > \vec{\alpha}_s \}| \\
&:= C_{\alpha} - E_{\vec{\alpha}}
\end{align*}
{\color{black} By \eqref{ProjGround} and \eqref{Gre}, this shows that the first identity holds up to a constant that depends only on $\alpha$. Since both the left--hand--side and right--hand--side of \eqref{TWOSIDES} satisfy the property that every column sums to $1$, this constant must be $1$.} 

Now proceed to the second identity. It is immediate from the definitions that the identity holds on $\vert \Omega \rangle^{\otimes l}$. To show the identity holds in general, it suffices to show 
$$
\langle \gamma \vert \Lambda^{(n)} \mathfrak{d}  \Delta^{(l)} f_1^{\gamma_1} \cdot \ldots \cdot \Delta^{(l)}  f_n^{ \gamma_{[1,n]}}    \vert \Omega\rangle^{\otimes l} = \langle \gamma \vert \mathcal{I}_l^{-1} \Delta^{(l)}  f_1^{\gamma_1} \cdot \ldots \cdot  \Delta^{(l)}  f_n^{ \gamma_{[1,n]}}    \vert \Omega\rangle^{\otimes l} ,
$$
since the $\textrm{Im}(\mathcal{I}_l)$ is generated by the generators $f_1,\ldots,f_n$ acting on $\vert \Omega \rangle^{\otimes l}$.  Since any exponents in $f_i$ other than $\gamma_{[1,i]}$ will result in zero, it suffices to show
$$
\langle \gamma \vert \Lambda^{(n)} \mathfrak{d}  \Delta^{(l)} u_0  \vert \Omega\rangle^{\otimes l} = \langle \gamma \vert  \mathcal{I}_l^{-1} \Delta^{(l)}u_0  \vert \Omega\rangle^{\otimes l} .
$$
Since $\mathcal{I}_l$ is an intertwiner of representations, the right--hand--side equals $\langle \gamma \vert u_0 \mathcal{I}_l^{-1} \vert \Omega \rangle^{\otimes l}$. By the definition of the representation \eqref{ACT}, the right--hand--side evaluates to 
\begin{align*}
\langle \gamma \vert u_0  \vert \Omega\rangle&= \prod_{i=1}^n \frac{[\gamma_{[1,i+1]}  ]_q[\gamma_{[1,i+1]}-1]_q \cdots [\gamma_{[1,i+1]}-\gamma_{[1,i]}+1]_q}{\{\gamma_{[1,i]}\}_{q^{2}}^!}   \\
&= \prod_{i=1}^n q^{g_i}\frac{\{\gamma_{[1,i+1]}  \}_{q^2}\{\gamma_{[1,i+1]}-1\}_{q^2} \cdots \{\gamma_{[1,i+1]}-\gamma_{[1,i]}+1\}_{q^2}}{  \{\gamma_{[1,i]}\}_{q^2}^!} \\
&= \prod_{i=1}^n q^{g_i}\binomqq{\gamma_{[1,i+1]}}{\gamma_{[1,i]}  }
\end{align*}
where 
$$
g_i = -(\gamma_{[1,i+1]}-1) - (\gamma_{[1,i+1]}-2) - \ldots -\gamma_{i+1}.
$$

Now proceed to the left--hand--side. Given $\vec{\gamma}$, define $\vec{\gamma}^{(i)}$ for $ 1 \leq i \leq n+1$ by
$$
\vec{\gamma}^{(i)}_s 
= 
\begin{cases}
\vec{\gamma}_s, & \text{ if } \vec{\gamma}_s > i, \\
i, & \text{ if } \vec{\gamma}_s \leq i. 
\end{cases}
$$
For example, $\vec{\gamma}^{(n+1)} = \vert \Omega \rangle^{\otimes l}$ and $\vec{\gamma}^{(1)} = \vec{\gamma}$. Then 
\begin{align*}
\langle \gamma \vert \Lambda^{(n)} \mathfrak{d}  \Delta^{(l)} u_0  \vert \Omega\rangle^{\otimes l} &= \sum_{\vec{\gamma}} \langle \gamma \vert \Lambda^{(n)} \vert \vec{\gamma} \rangle \langle \vec{\gamma} \vert \mathfrak{d} \vert \vec{\gamma} \rangle \prod_{i=1}^n \langle \vec{\gamma}^{(i)} \vert \exp_{q^{-2}}(\Delta^{(l)}f_i) \vert \vec{\gamma}^{(i+1)} \rangle
\end{align*}
The sum over $\vec{\gamma}$ can be re-written as a sum over sequences of subsets $\emptyset \subset L_1 \subset \cdots \subset L_n \subset \{1,\ldots,l\}$, where each $L_i$ is the set
$$
L_i = \{s: \vec{\gamma}_s \leq i \}  =  \{s: \vec{\gamma}^{(i)}_s \leq i \} .
$$
This is advantageous because now defining
$$
D(L_i,L_{i+1}) = |\{ (r,s):1\leq r<s \leq l \text{ and } r \in L_i \text{ and } s \in L_{i+1} -L_i\}|,
$$
leads to the two identities 
\begin{align*}
\langle \vec{\gamma}^{(i)} \vert \exp_{q^{-2}}(\Delta^{(l)}f_i) \vert \vec{\gamma}^{(i+1)} \rangle &= q^{g_i} q^{ D(L_i,L_{i+1})}\\
 \langle \vec{\gamma} \vert \mathfrak{d} \vert \vec{\gamma} \rangle  &= \prod_{i=1}^n q^{D(L_i,L_{i+1})}
\end{align*}
The first identity above uses the pseudo--factorization property \eqref{PseudoFac}. These two identities together establish that
$$
\langle \gamma \vert \Lambda^{(n)} \mathfrak{d}  \Delta^{(l)} u_0  \vert \Omega\rangle^{\otimes l} =  \sum_{L_1 \subset \cdots \subset L_n} \prod_{i=1}^n q^{g_i}q^{2D(L_i,L_{i+1})}
$$
Each $L_{i+1}$ is a set with $\gamma_{[1,i+1]}$ elements, and the sum over subsets $L_i$ is a sum over subsets of $\gamma_{[1,i]}$ elements. By the $q$--Binomial theorem, stated as equation \eqref{qBin},
$$
\langle \gamma \vert \Lambda^{(n)} \mathfrak{d}  \Delta^{(l)} u_0  \vert \Omega\rangle^{\otimes l} = \prod_{i=1}^n q^{g_i}\binomqq{\gamma_{[1,i+1]}}{\gamma_{[1,i]}},
$$
which equals the right--hand--side.

This completes the proof of (a).

(b) It follows immediately from the definitions that $\Lambda^{(n)} \circ \Xi^{(n)} $ is the identity on $V_l$. 

By Lemma \ref{ABCDE}, 
\begin{multline*}
\langle \gamma, \delta \vert \text{RHS of } \eqref{RPCP} \vert \vec{\alpha},\beta\rangle 
= \langle {\gamma},\delta \vert \widetilde{\Ga}^{-1} \vert {\gamma},\delta \rangle \langle {\alpha},\beta \vert  {\Ga} \vert\alpha ,\beta \rangle\\
\times \sum_{{\gamma}: \Lambda^{(n)}\vec{\gamma} = \gamma}q^{\mathcal{N}_1(\vec{\alpha},\vec{\gamma}) + \mathcal{N}_2(\vec{\alpha},\vec{\gamma}) } \langle {\gamma} \vert \Lambda^{(n)}  \vert \vec{\gamma}\rangle\langle \vec{\gamma},\delta\vert R_{1,l+1}(zq^{l-1}) \cdots R_{l,l+1}(zq^{1-l}) \vert \vec{\alpha}, \beta \rangle
\end{multline*}
and 
\begin{multline*}
\langle \gamma, \delta \vert \text{LHS of } \eqref{RPCP} \vert \vec{\alpha},\beta\rangle 
= \langle {\gamma},\delta \vert \widetilde{\Ga}^{-1} \vert {\gamma},\delta \rangle \langle {\alpha},\beta \vert  {\Ga} \vert\alpha ,\beta \rangle\\
\times \sum_{{\gamma}: \Lambda^{(n)}\vec{\gamma} = \gamma}q^{\mathcal{N}_1(\vec{\alpha},\vec{\gamma}) + \mathcal{N}_2(\vec{\alpha},\vec{\gamma}) } \langle \gamma \vert \Lambda^{(n)} \vert \vec{\gamma} \rangle \langle \vec{\gamma} \vert R_{1,l+1}(zq^{l-1}) \cdots R_{l,l+1}(zq^{1-l}) (\Xi^{(n)}\Lambda^{(n)}  \otimes \mathrm{id})  \vert \vec{\alpha}, \beta \rangle
\end{multline*}
Because the weight spaces of $V_l$ are one--dimensional, $\Xi^{(n)}\Lambda^{(n)}= \mathrm{const}\cdot P^+$. Again using 
$$
\mathcal{N}_1(\vec{\alpha},\vec{\gamma}) + \mathcal{N}_2(\vec{\alpha},\vec{\gamma}) = | \{ (r,s): 1 \leq r<s\leq l \text{ and } \vec{\gamma}_r < \vec{\gamma}_s \}| - | \{ (r,s): 1\leq r<s \leq l \text{ and } \vec{\alpha}_s > \vec{\alpha}_r \},
$$
it suffices to show that
$$
\Lambda^{(n)} \mathfrak{d} R_{1,l+1}(zq^{l-1}) \cdots R_{l,l+1}(zq^{1-l})  = \mathrm{const} \cdot \Lambda^{(n)} \mathfrak{d} R_{1,l+1}(zq^{l-1}) \cdots R_{l,l+1}(zq^{1-l}) (P^+ \otimes \mathrm{id})
$$
But this follows from \eqref{Stronger} and \eqref{SECIND}.

\end{proof}

\subsection{Stochasticity of $S(z)$}\label{Stoch}
For certain explicit values of the spectral parameter $z$, the $S$--matrix is stochastic. 

\begin{proposition}\label{Stochastic} The operator $S(z)$ acting on $V_l \otimes V_m$ is stochastic in the cases:
\begin{itemize}
\item Both $q>1$ and $0 \leq z \leq q^{2-l-m} $ hold.
\item Both $0<q<1$ and $z \geq q^{2-l-m}$ hold.
\end{itemize}

\end{proposition} 
\begin{proof}
We have already seen that the output of $S(z)$ sums to $1$. Since the entries of the gauge transformation $\Ga$ are non--negative (for $q>0$), it remains to show that the entries of $R(z)$ are non--negative. 

By \eqref{l=1}, the proposition holds for $l=m=1$, when $q^{2-l-m}=1$. By the fusion procedure \eqref{Fusion}, if $ 0 \leq zq^{m+l-2} \leq 1$ in the first case (and $zq^{m+l-2} \geq 1$ in the second case), then  $R(z)$ is a product of matrices with non--negative entries, so is itself a matrix with non--negative entries.

The equation \eqref{Fusion2} can also be used to arrive at the same result. For $l=1$ and general values of $m$, \eqref{l=1} shows that non--negativity holds for $0 \leq z \leq q^{1-m}$ in the first case (and $z \geq q^{1-m}$ in the second case). For general values of both $l$ and $m$, the necessary inequality from \eqref{Fusion2} is $0 \leq z q^{l-1} \leq  q^{1-m}$ or $zq^{l-1} \geq q^{1-m}$.

\end{proof}

Note that this is not an exhaustive list of all values for which $S(z)$ is stochastic, since $z=q^{l-m}$ is not included. Also note that the second case of Proposition \ref{Stochastic} is similar to (3.9) of \cite{BM}.

There is a certain symmetry in the two cases in Proposition \ref{Stochastic}, in that the second case can be derived from the first under simultaneous change of variables $q\rightarrow q^{-1},z\rightarrow z^{-1}$. Indeed, it turns out the two cases are related according to the choice of co--product and the charge reversal $\Pi$. Recall the alternative co--product $\Delta_0$ defined in section \ref{RT}, and that $R(z)$ was uniquely defined by the intertwining property \eqref{Inter} and the unit normalization condition \eqref{Unit}. Equation \eqref{Unit} is encapsulated in the sum--to--one property of stochastic matrices, and \eqref{Inter} is described in the next proposition.

\begin{proposition}\label{Invert}
The $R$--matrix is preserved under simultaneous inversion of the spectral parameter $z$, asymmetry parameter $q$, and charge reversal, in the sense that for all $u \in \mathcal{U}_q(A_n^{(1)})$, the equality
$$
\check{R}(z_2/z_1)\Big|_{q\rightarrow q^{-1}} \circ \Pi^{\otimes 2} (\pi_l^{z_1} \otimes \pi_m^{z_2})(\Delta_{\mathrm{0}}(u))\Pi^{\otimes 2} = \Pi^{\otimes 2} (\pi_m^{z_1} \otimes \pi_l^{z_2})(\Delta_{\mathrm{0}}(u))\Pi^{\otimes 2} \circ \check{R}(z_2/z_1)\Big|_{q\rightarrow q^{-1}}
$$
holds as operators on $V_l ^{z_1}\otimes V_m^{z_2}$.
\end{proposition}
\begin{proof}
First show that as operators on $V_m^z$ for any $m \geq 0$,
$$
\Pi \pi_m^z(e_i) \Pi = \pi_m^{1/z}(f_{n+1-i}), \quad \Pi \pi_m^z(f_i) \Pi = \pi_m^{1/z}(e_{n+1-i}), \quad, \Pi \pi_m^z(k_i) \Pi = \pi_m^{1/z}(k_{n+1-i}^{-1})
$$
This is actually straightforward from $(\alpha + \hat{i})' = \alpha' - \widehat{n{+}1{-}i}$. For instance,
\begin{align*}
\Pi \pi_m^z(e_i) \Pi  \vert \alpha \rangle &= \Pi \pi_m^z(e_i) \vert \alpha' \rangle \\
&= \Pi z^{\delta_{i,0}} [\alpha_{n+2-i}]_q \vert \alpha' - \hat{i} \rangle \\
&= \Pi z^{\delta_{i,0}} [\alpha_{n+2-i}]_q \vert (\alpha +   \widehat{n{+}1{-}i})' \rangle\\
&= \pi_m^{1/z}(f_{n+1-i})\vert \alpha \rangle
\end{align*}
and similarly for $f_i$ and $k_i$. 

And now as operators on $V_l ^{z_1}\otimes V_m^{z_2}$,
\begin{align*}
(\Pi \otimes \Pi) (\pi_l^{z_1} \otimes \pi_m^{z_2})(\Delta_{\mathrm{0}}(e_i) )(\Pi \otimes \Pi) &= (\Pi \otimes \Pi) (\pi_l^{z_1} \otimes \pi_m^{z_2})( e_i \otimes 1 + k_i ^{-1}\otimes e_i)  (\Pi \otimes \Pi)  \\
&= (\pi_l^{1/z_1} \otimes \pi_m^{1/z_2} )\left( f_{n+1-i} \otimes 1 + k_{n+1-i}^{} \otimes f_{n+1-i}\right) \\
&= (\pi_l^{1/z_1} \otimes \pi_m^{1/z_2} )\left( f_{n+1-i} \otimes 1 + k_{n+1-i}^{-1} \otimes f_{n+1-i}\right)\Big|_{q\rightarrow q^{-1}} \\
&=  (\pi_l^{1/z_1} \otimes \pi_m^{1/z_2} ) (\Delta(f_{n+1-i})) \Big|_{q\rightarrow q^{-1}},
\end{align*}
where the third equality reflects the fact that the action of $f_{n+1-i}$ is preserved under $q\rightarrow q^{-1}$, but the action of $k_{n+1-i}$ is inverted. Similar arguments hold for $\Delta_0(f_i),\Delta_0(k_i)$. Therefore, recalling the definition of the involution $\omega$, 
\begin{align*}
\check{R}(z_2/z_1)\Big|_{q\rightarrow q^{-1}} \circ \Pi^{\otimes 2} (\pi_l^{z_1} \otimes \pi_m^{z_2})(\Delta_{\mathrm{0}}(u))\Pi^{\otimes 2} &= \check{R}(z_2/z_1)\Big|_{q\rightarrow q^{-1}} \circ (\pi_l^{{1/z_1}} \otimes \pi_m^{1/z_2})(\Delta_{\mathrm{}}(\omega u))\Big|_{q\rightarrow q^{-1}}  \\
&= (\pi_m^{{1/z_1}} \otimes \pi_l^{1/z_2})(\Delta_{\mathrm{}}(\omega u))\Big|_{q\rightarrow q^{-1}}  \circ \check{R}(z_2/z_1)\Big|_{q\rightarrow q^{-1}} \\
&= \Pi^{\otimes 2} (\pi_m^{z_1} \otimes \pi_l^{z_2})(\Delta_{\mathrm{0}}(u))\Pi^{\otimes 2} \circ \check{R}(z_2/z_1)\Big|_{q\rightarrow q^{-1}},
\end{align*}
as needed. 
\end{proof}
Note that the differing choice of co--product would also result in a different expression for the gauge transformation, since the latter comes from the $U_q(A_n)$--orbit of \eqref{Unit}, and the action of $U_q(A_n)$ is determined by the co--product.

\subsection{Lumpability of $S(z)$}
Recall the definition of lumpability in section \ref{LumpSection}. The next proposition says that projecting the $\mathcal{U}_q(A_n^{(1)})$ model onto consecutive particles is the $\mathcal{U}_q(A_p^{(1)})$ model.   

\begin{proposition}\label{JJJ}
Fix $1 \leq r \leq n$. The stochastic matrix $S(z)$ on $V_l \otimes V_m$ is lumpable with respect to the partition on $\mathcal{B}_l^{(n+1)} \times \mathcal{B}_m^{(n+1)}$ defined by 
\begin{align*}
(\alpha,\beta) \sim (\gamma,\delta) & \text{ if } \alpha_r + \alpha_{r+1} = \gamma_r + \gamma_{r+1} \text{ and } \beta_r + \beta_{r+1} = \delta_r + \delta_{r+1} \\
& \quad \text{ and } \alpha_s=\gamma_s,\beta_s=\delta_s \text{ for } s \neq r,r+1.
\end{align*}
The lumped $S(z)$ matrix is the $S(z)$ matrix acting on $\mathcal{B}_l^{(n)} \times \mathcal{B}_m^{(n)}.$

Furthermore, given any $1 \leq r_1 < r_2 < \ldots < r_p \leq n+1$, $S(z)$ is lumpable with respect to the partition
\begin{align*}
(\alpha,\beta) \sim (\gamma, \delta) & \text{ if } \alpha_1 + \ldots + \alpha_{r_1-1} = \gamma_1 + \ldots + \gamma_{r_1-1} \text{ and } \beta_1 + \ldots + \beta_{r_1-1} = \delta_1 + \ldots + \delta_{r_1-1} \\
& \text{ if } \alpha_{r_1} + \ldots + \alpha_{r_2-1} = \gamma_{r_1} + \ldots + \gamma_{r_2-1} \text{ and } \beta_{r_1} + \ldots + \beta_{r_2-1} = \delta_{r_1} + \ldots + \delta_{r_2-1} \\
& \cdots \\
& \text{ if } \alpha_{r_p} + \ldots + \alpha_{n+1} = \gamma_{r_p} + \ldots + \gamma_{n+1} \text{ and } \beta_{r_p} + \ldots + \beta_{n+1} = \delta_{r_p} + \ldots + \delta_{n+1}.
\end{align*}
The lumped $S(z)$ matrix is the $S(z)$ matrix acting on $\mathcal{B}_l^{(p)} \times \mathcal{B}_m^{(p)}.$
\end{proposition}
\begin{proof}
Since the projections can be composed, it suffices to prove the first statement. 

By Theorem \ref{StochasticFusion}, $S(z)$ is a composition of stochastic operators, so it suffices to show that each of those operators is lumpable with respect to the same partition. It is straightforward that $\Xi^{(n)}$ and $\Lambda^{(n)}$ are lumpable. The matrix $S(z)$ acting on $V_1 \otimes V_m$ is lumpable, which follows from the following calculations from \eqref{Stol=1}:

For $1\leq j\leq r-1$,
\begin{align*}
({q^{m+1}-z} )(S(z)_{\epsilon_j,\beta}^{\epsilon_r,\beta+\epsilon_j-\epsilon_r}  + S(z)_{\epsilon_j,\beta}^{\epsilon_{r+1},\beta+\epsilon_j-\epsilon_{r+1}} )&= {-q^{2\beta_{[1,r-1]} - m + 1} \left( 1 - q^{2\beta_r} \right) -q^{2\beta_{[1,r]} - m + 1} \left( 1 - q^{2\beta_{r+1}} \right)} \\
&={-q^{2\beta_{[1,r-1]} - m + 1} \left( 1 - q^{2(\beta_r+ \beta_{r+1})} \right)  }.
\end{align*}
For $j=r$,
\begin{align*}
(q^{m+1}-z)(S(z)_{\epsilon_j,\beta}^{\epsilon_r,\beta+\epsilon_j-\epsilon_r}  + S(z)_{\epsilon_j,\beta}^{\epsilon_{r+1},\beta+\epsilon_j-\epsilon_{r+1}} ) &=  q^{2\beta_{[1,r]}  -  m + 1} \left( 1 - q^{-2\beta_r+m-1}z\right) -q^{2\beta_{[1,r]} - m + 1} \left( 1 - q^{2\beta_{r+1}} \right) \\
&= q^{2\beta_{[1,r]}  -  m + 1}\left( q^{2\beta_{r+1}} - q^{-2\beta_r+m-1} z\right) \\
&= q^{2\beta_{[1,r+1]}  -  m + 1}\left( 1 - q^{-2(\beta_r+\beta_{r+1})+m-1} z\right).
\end{align*}
For $j=r+1$,
\begin{align*}
(q^{m+1}-z)(S(z)_{\epsilon_j,\beta}^{\epsilon_r,\beta+\epsilon_j-\epsilon_r}  + S(z)_{\epsilon_j,\beta}^{\epsilon_{r+1},\beta+\epsilon_j-\epsilon_{r+1}}  )&= -q^{2\beta_{[1,r-1]}}z(1-q^{2\beta_r})+  q^{2\beta_{[1,r+1]}  -  m + 1} \left( 1 - q^{-2\beta_{r+1}+m-1}z\right)\\
&=  q^{2\beta_{[1,r+1]}  -  m + 1} - q^{2\beta_{[1,r-1] }}  z \\
&=  q^{2\beta_{[1,r+1]}  -  m + 1} \left( 1 - q^{-2(\beta_{r-1}+\beta_r)+m-1}z \right).
\end{align*}
For $j>r+1$, 
\begin{align*}
(q^{m+1}-z)(S(z)_{\epsilon_j,\beta}^{\epsilon_r,\beta+\epsilon_j-\epsilon_r}  + S(z)_{\epsilon_j,\beta}^{\epsilon_{r+1},\beta+\epsilon_j-\epsilon_{r+1}}  )&= -q^{2\beta_{[1,r-1]}}z(1 -q^{2\beta_r}) -q^{2\beta_{[1,r]}}z(1 -q^{2\beta_{r+1}}) \\
&= -q^{2\beta_{[1,r-1]}}z( 1 - q^{2(\beta_r+\beta_{r+1})}).
\end{align*}

\end{proof}

\subsection{Analytic Continutation}
In the case when $l=1$, the formula \eqref{Stol=1} allows for an analytic continuation in the variable $\mu=q^{-m}$. The expressions for the matrix entries can be re-written as
$$
( \mu^{-1}q-z)\mathcal{S}(z)_{\epsilon_j,\beta}^{\epsilon_k,\delta}  = 1_{ \{\epsilon_j  + \beta = \epsilon_k + \delta \}} \times 
\begin{cases}
q^{2\beta_{[1,k]}   } \left( q\mu - q^{-2\beta_k}z\right), & \text{ if } k = j<n+1 , \\
q\mu^{-1} - q^{2 \vert \beta \vert }z, & \text{ if } k = j =n+1 , \\
-q^{2\beta_{[1,k-1]}  }q\mu \left( 1 - q^{2\beta_k} \right), & \text{ if } n+1>k > j , \\
q\mu (-q^{2 \vert \beta \vert  }   + \mu^{-2} ) , & \text{ if }  n+1 = k> j , \\
-q^{2 \beta_{[1,k-1]}  } z \left( 1 - q^{2\beta_k} \right), & \text{ if } k < j .
\end{cases}
$$
This is a matrix acting on $V_1 \otimes \bar{V}_{\infty}$. Note that the third line does not occur in the $n=1$ case. However, this scenario actually significantly restricts the cases in which $S(z)$ is stochastic.

\begin{proposition}
Assume that $q, z,\mu$ all take real values. The matrix $\mathcal{S}(z)$ is stochastic if and only if one of the following cases holds:
\begin{itemize}
\item 
$\mu=0$
\item
$\vert q\vert =1$ \text{ and either } $\mu^{-1} \leq \mu \leq z, \text{  or  } \mu^{-1} \geq \mu \geq z.$
\end{itemize}
\end{proposition}
\begin{proof}
If $\mu=0$, then the second and fourth lines are equal to $1$ and all other lines are $0$ (corresponding to the fourth item of Remark \ref{SimpleCase}). In this case, $\mathcal{S}(z)$ is stochastic, so assume hereafter that $\mu \neq 0$. 
 
First consider the case when $\vert q \vert>1$. If $\mathcal{S}(z)$ is stochastic, then the fifth, third, and second lines respectively show that $z,q\mu$ and $q\mu^{-1} - q^{2 \vert \beta\vert}z$ all have the same sign. Since $\mu^2 \geq 0$, then $q\mu$ and $q\mu^{-1}$ have the same sign as well. But for sufficiently large values of $\vert \beta\vert$,  the expression $q\mu^{-1} - q^{2 \vert \beta\vert}z$ will have a different sign from $q\mu$. This is a contradiction.

Consider the case when $ \vert q\vert<1$ and $\mu^{-1}q-z > 0$. If $\mathcal{S}(z)$ is stochastic, then the third and fifth lines show that $q\mu,z<0$. If $q\mu <0$, then the fourth line shows that $\mu^{-2} - q^{2\vert \beta\vert}<0$ for all $\beta$. But for sufficiently large values of $\vert \beta\vert$, this implies that $\mu^{-2}<0$, which is also a contradiction.

Consider the case when $ \vert q\vert<1$ and $\mu^{-1}q-z < 0$. If $\mathcal{S}(z)$ is stochastic, then the third and fifth lines show that $q\mu,z>0$. If $q\mu >0$, then the fourth line shows that $\mu^{-2} - q^{2\vert \beta\vert}<0$ for all $\beta$. This is again a contradiction.

Now suppose $ \vert q \vert=1$. Then the third and fifth lines are always $0$, and the second line is always $1$. So it suffices to see that stochasticity holds if and only if 
$$
\frac{\mu^{-1}-\mu}{\mu^{-1}-z}\geq 0, \quad \frac{\mu-z}{\mu^{-1}-z} \geq 0.
$$
This happens precisely in the two cases listed. 

\end{proof}

Note that while this proposition does not imply that analytic continuation in the parameter $\mu$ is fruitless, it does seem to imply that analytic continuation needs to be done in both $\lambda=q^{-l}$ and $\mu=q^{-m}$ simultaneously, as was done in \cite{KMMO}.

\subsubsection{The single--species $n=1$ case}

Now consider when $n=1$. In this case, the $\mathcal{U}_q({A}_n^{(1)})$ stochastic vertex model matches that of \cite{BP},\cite{CP}, described in Section \ref{Osvm}. 

\begin{proposition}\label{Match}

When $n=1$, then 
$$
\mathcal{S}(- \alpha q^{l}\mu^{-1}) \Big|_{V_l \otimes \bar{V}_{\infty}} = \mathring{S}_{\alpha} \Big|_{V_l \otimes \bar{V}_{\infty}}.
$$
\end{proposition}
\begin{proof}
First consider the case when $l=1$.  In the expression for $\mathcal{S}(z)$, substitute $
-\alpha q \mu^{-1}=z$ to get
$$
(1+\alpha)\mathcal{S}(z)_{\epsilon_j,\beta}^{\epsilon_k,\delta}  = 1_{ \{\epsilon_j  + \beta = \epsilon_k + \delta \}} \times 
\begin{cases}
q^{2\beta_{[1,k]}   }  \left( \mu^2 + q^{-2\beta_k} \alpha\right), & \text{ if } k = j<n+1 , \\
1 + q^{2 \vert \beta \vert }\alpha, & \text{ if } k = j =n+1 , \\
-q^{2\beta_{[1,k-1]}  }\mu^2 \left( 1 - q^{2\beta_k} \right), & \text{ if } n+1 \neq k > j , \\
1 - \mu^2 q^{2 \vert \beta \vert  }   , & \text{ if } k = n+1 > j , \\
q^{2 \beta_{[1,k-1]}  }  \alpha \left( 1 - q^{2\beta_k} \right), & \text{ if } k < j .
\end{cases}
$$
When $n=1$, there are only two possible outputs given any input. Since the columns sum to $1$, it suffices to check that the formulas match when $k=j$. Setting $g = \vert \beta \vert = \beta_1$, the expressions from the first and second lines are respectively
$$
\frac{\alpha+\mu^2q^{2g}}{1+\alpha}, \quad \frac{1+q^{2g}\alpha}{1+\alpha}
$$
This matches \eqref{n=1}.

Now consider when $l>1$. By Theorem \ref{StochasticFusion} and the uniqueness of analytic continuation,
\begin{align*}
\mathcal{S}(-\alpha q^l \mu^{-1}) &= (\Lambda^{(1)} \otimes \mathrm{id} ) \circ \mathcal{S}_{1,l+1}( - \alpha q^{2l-1}\mu^{-1}) \cdots \mathcal{S}_{l,l+1}( - \alpha \mu^{-1} q) \circ (\Xi^{(1)} \otimes \mathrm{id}) \\
&=(\Lambda^{(1)} \otimes \mathrm{id} ) \circ [\mathring{S}_{\alpha q^{2l-2}}]_{1,l+1} \cdots [\mathring{S}_{\alpha }]_{l,l+1}  \circ (\Xi^{(1)} \otimes \mathrm{id}) 
\end{align*}
This is exactly \eqref{FusionCP}.
\end{proof}

\section{Intertwining of transfer matrices}\label{ALGDUA}

\subsection{Equivalent expression for duality function}
Let us relate the gauge in section \ref{KMMOBACK} and the ground state transformation in section \ref{KUANBACK}, which allows us to rewrite the duality function. Here, the duality function will also act on the auxiliary space $V_l$. Given $\mu = (q^{-m_1},\ldots, q^{-m_L})$, define
$$
\mu^{+x} = (q^{-m_1},\ldots,q^{-m_x},q^{-l},q^{-m_{x+1}},\ldots,q^{-m_L})
$$ 
and set $D^{+x}_{\mu}(u) = D_{\mu^{+x}}(u).$ In other words, $D^{+x}_{\mu}(u)$ acts as 
$$
\Pi^{\otimes L+1} \circ \Gr^{-1} \circ \Delta^{(L+1)}(u) \circ \Gr^{-1} \circ \left(B^{-1}\right)^{\otimes L+1}
$$
on the representation 
$$
V_{m_1} \otimes \cdots \otimes V_{m_x} \otimes V_l \otimes V_{m_{x+1}} \otimes \cdots \otimes V_{m_L}.
$$
The subscript $\mu$ and the superscript $x$ will be dropped if it is clear from the context, but the $^+$ will always be retained.  

\begin{proposition}\label{GaGr}

(a) The ground state transformation $\Gr$ is related to the gauge transform $\Ga$ by
\begin{equation}\label{Grub}
\mathrm{const} \cdot \Gr \circ B^{\otimes L} = \Ga^{-1},
\end{equation}
where $\mathrm{const}$ is a constant under particle conservation.

(b) For any $x$, 
$$
D^{+x}(u)=\mathrm{const} \cdot  \Pi^{\otimes L+1} \circ B^{\otimes L+1} \circ \Ga^{(L+1)} \circ \Delta^{(L+1)}(u) \circ \mathrm{Ga}^{(L+1)},
$$
where $\mathrm{const}$ is a constant under particle conservation.

\end{proposition}

\begin{proof}
(a)
In order see the relationship between $\Gr$ and $\Ga$, use the identity
$$
[k]_q^! = [1]_q \cdots [k]_q = \frac{(q-q^{-1}) \cdots (q^k - q^{-k})}{ (q-q^{-1})^k } = \frac{(-1)^k q^{-k(k+1)/2} (1-q^2) \cdots (1-q^{2k}) }{ (q-q^{-1})^k } = \frac{(-1)^k q^{-k(k+1)/2}  }{ (q-q^{-1})^k } (q^2)_k.
$$
Therefore, recalling \eqref{Gr},
\begin{align*}
\Gr_{\eta}^{\eta} &= \mathrm{const} \cdot \prod_{i=1}^{n+1} \prod_{x=1}^L  \frac{1}{ (q^2)_{\eta_i^x}  }q^{(\eta_i^x )^2/2}  \prod_{1 \leq y <x \leq L} q^{  -\sum_{1 \leq i < j \leq n+1} \eta^y_j\eta^x_i} \\
&= \mathrm{const} \cdot \prod_{i=1}^{n+1} \prod_{x=1}^L  \frac{1}{ (q^2)_{\eta_i^x}  }q^{-\sum_{1\leq i<j\leq n+1} \eta_i^x \eta_j^x} \prod_{1 \leq y <x \leq L} q^{  -\sum_{1 \leq i < j \leq n+1} \eta^y_j\eta^x_i} \\ 
&= \mathrm{const} \cdot \prod_{i=1}^{n+1} \prod_{x=1}^L  \frac{1}{ (q^2)_{\eta_i^x}  } \prod_{1 \leq y <x \leq L}  q^{  \sum_{1\leq i<j \leq n+1} \eta^y_i\eta^x_j  },
\end{align*}
which implies that $\Gr \circ B^{\otimes L}=\Ga^{-1}$. Above, the second and third equalities use the respective identities
\begin{align*}
\mathrm{const} = (m_x)^2 = \left( \sum_{i=1}^{n+1} \eta^x_i \right)^2 &= \sum_{i=1}^{n+1}( \eta^x_i)^2  + 2 \sum_{1 \leq i<j \leq n+1} \eta^x_i\eta^x_j, \\
\mathrm{const}=\left(\eta_i^1 + \cdots + \eta_i^L\right)\left(\eta_j^1 + \cdots + \eta_j^L\right) &= \sum_{x=1}^L \eta_i^x \eta_j^x+ \sum_{1 \leq  y<x \leq L} \left(  \eta_i^y \eta_j^x + \eta_j^y\eta_i^x \right).
\end{align*}

(b)
By part (a),
\begin{align*}
D^{+x}(u) &= \Pi^{\otimes L+1} \circ \Gr^{-1} \circ \Delta^{(L+1)}(u) \circ \Gr^{-1} \circ \left(B^{-1}\right)^{\otimes L+1}\\
&=  \Pi^{\otimes L+1} \circ B^{\otimes L+1} \circ \Ga^{(L+1)} \circ \Delta^{(L+1)}(u) \circ \mathrm{Ga}^{(L+1)} .
\end{align*}

\end{proof}

The next lemma concerning the function $D^+(u_0)$ will be beneficial: from a probabilistic standpoint, a proper duality function should not act on the auxiliary space, and in certain cases the function $D^+$ can be reduced to $D$. 
\begin{lemma}\label{Reduction}
$$
\Q{l}^! \langle \xi \vert D(u_0) \vert \eta \rangle = \langle \xi, \mathrm{A} \vert D^{+L}(u_0) \vert \eta,\Omega\rangle =  \langle \Omega, \xi \vert D^{+0}(u_0) \vert \mathrm{A}, \eta \rangle.
$$
\end{lemma}
\begin{proof}

To see that $\Q{l}^!\langle \xi \vert D^{}(u_0) \vert \eta \rangle = \langle \xi, \mathrm{A} \vert D^{+L}(u_0) \vert \eta,\Omega\rangle$, note that by the explicit expression for $D(u_0)$ in \eqref{ExpExp}, adding lattice sites to $\eta$ which contain no particles multiplies $D(u_0)$ by $\Q{l}^!$, due to the $\Q{\eta_n^x}^!$ term, with all others equal to $1$.

To see that $\Q{l}^! \langle \xi \vert D^{}(u_0) \vert \eta \rangle  =  \langle \Omega, \xi \vert D^{+0}(u_0) \vert \mathrm{A}, \eta \rangle$, examine the $q$--factorial terms from \eqref{ExpExp0}.
When $\eta = \vert \mathrm{A} \rangle$, so that $\eta_1^x=m_x$ and all other $\eta_i^x=0$, then $\eta_{[1,n-i]}^x = \eta_{[i,n-i-1]}^x$ for $i \neq n-1$. Thus the only change is from $i=1$, where $D(u_0)$ is multiplied by $\Q{l}^!$.
\end{proof}

\subsection{The one--site case}
Start with the case $L=1$. 

\begin{theorem}\label{2Site} As maps from $V_l^z \otimes V_m^w$ to $V_m^w \otimes V_l^z$, and for any $u \in \mathcal{U}_q( A_n^{(1)})$,
$$
\left(S(z/w)P\right)^* D^+(u) = D^+(u)P{S}(z/w).
$$
\end{theorem}

\begin{proof}
The statement is equivalent to 
$$
PS^*(z/w) D^+(u) = D^+(u)P{S}(z/w).
$$
and then to
$$
S^*(z/w) D^+(u)P = PD^+(u)\widetilde{S}(z/w),
$$
as maps from $V_m \otimes V_l$ to $V_l \otimes V_m$.

We have that (abbreviating $B^{\otimes 2}$ to $B$ and $\Pi^{\otimes 2}$ to $\Pi$)
\begin{align*}
S^*(z/w) &= \Ga \circ R_{}^*(z/w) \circ \widetilde{\Ga}^{-1}\\
& \stackrel{\eqref{Trans}}{=} \Ga \circ \Pi \circ B \circ R_{}(z/w) B^{-1} \Pi^{-1} \widetilde{\Ga}^{-1}. 
\end{align*}
This implies that
\begin{align*}
S^*(z/w)D^+(u)P &= \Ga \circ \Pi \circ B \circ R_{}(z/w) B^{-1} \Pi^{-1} \widetilde{\Ga}^{-1} \circ \Pi \circ B \circ \Ga \circ \Delta_{}(u) \circ \Ga \circ P \\
&\stackrel{\eqref{VVV}}{=} \Pi \circ \widetilde{\Ga} \circ B \circ R_{}(z/w) B^{-1} \circ {\Ga}^{-1}  \circ B \circ \Ga \circ \Delta_{}(u) \circ \Ga \circ P \\
& = \Pi \circ P \circ {\Ga}^{} \circ P \circ B \circ R_{}(z/w) \circ \Delta_{}(u) \circ \Ga \circ P \\
& \stackrel{\eqref{Diag}}{=} \Pi \circ P \circ {\Ga}^{} \circ B \circ P \circ R_{}(z/w) \circ \Delta_{}(u) \circ \Ga \circ P
\end{align*}
Meanwhile,
\begin{align*}
PD^+(u) \widetilde{S}_{}(z/w) &= P \circ \Pi \circ B \circ \Ga \circ \Delta_{}(u) \circ \Ga \circ \Ga^{-1} \widetilde{R}_{}(z/w) \widetilde{\Ga}_{} \\
&= P \circ \Pi \circ B \circ \Ga \circ \Delta_{}(u) \circ \widetilde{R}_{}(z/w) \circ \widetilde{\Ga}_{}.
\end{align*}
Recalling that $\widetilde{M}=PMP,$ we have
$$
PD^+(u) \widetilde{S}_{}(z/w) = P_{} \circ \Pi \circ B \circ \Ga  \circ \Delta_{}(u) \circ P_{} \circ R_{}(z/w) \circ  \Ga_{} \circ P_{}
$$
and
$$
S_{}^*(z/w)D^+(u)P = \Pi \circ P_{} \circ B \circ \Ga  \circ P_{} \circ R_{}(z/w) \circ \Delta_{}(u) \circ \Ga_{} \ P_{}
$$
Finally, recall that $\Pi P=P\Pi$ (by \eqref{Diag}) and that by \eqref{Inter}
$$
P_{} R_{}(z/w) \Delta_{}(u) = \Delta_{}(u) P_{} R_{}(z/w).
$$
This finishes the proof.
\end{proof}

\begin{figure}

\caption{The below example shows the sequential update of $\mathcal{T}_{\text{rev}}$ from left to right, defined as the sequence of maps
$
V_l \otimes V_{m_1} \otimes V_{m_2} \otimes V_{m_3} \rightarrow V_{m_1} \otimes V_{l} \otimes V_{m_2} \otimes V_{m_3} \rightarrow V_{m_1} \otimes V_{m_2} \otimes V_{l} \otimes V_{m_3} \rightarrow  V_{m_1} \otimes V_{m_2} \otimes V_{m_3}  \otimes V_l.
$
Dotted lines indicate updates that have yet to occur, and solid lines with no arrows indicate no particles.}

\begin{center}

\includegraphics{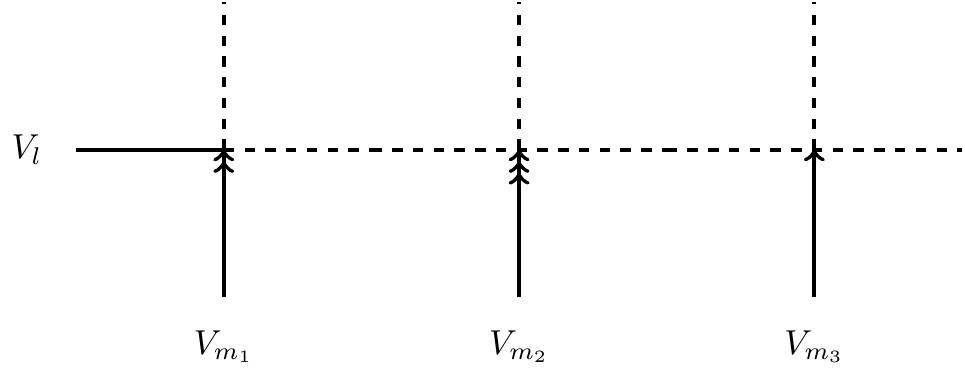}

\vspace{0.15in}

\includegraphics{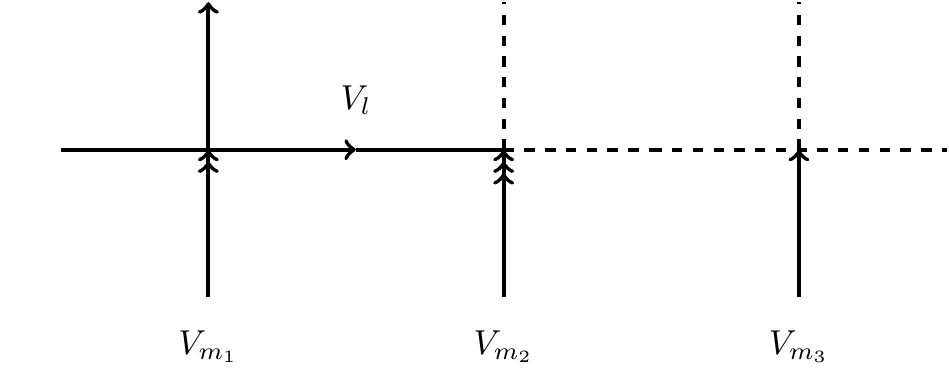}

\vspace{0.15in}

\includegraphics{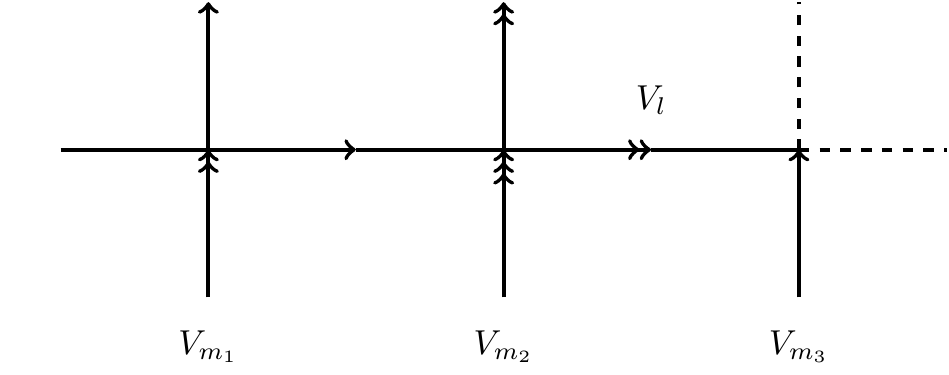}

\vspace{0.15in}

\hspace{0.25in} \includegraphics{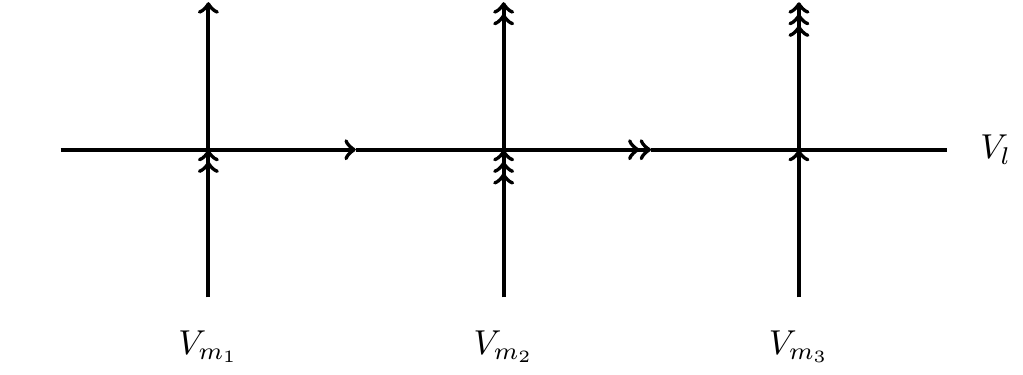}

\end{center}

\label{Tikz}
\end{figure}

\subsection{Extension to $L$ sites}

This section will extend the results of the previous section to $L$ lattice sites. 

The definition of the transfer matrix will be slightly modified in order to state the theorem. The transfer matrix (with open boundary conditions and left jumps) can be written as a map $\mathcal{T}: V^{(L)} \otimes V_l \rightarrow V_l \otimes V^{(L)}$ defined by
$$
\mathcal{T} = (S(z/w_1)P)_{12} \circ \cdots \circ (S(z/w_L)P)_{L,L+1},
$$
where each $(S(z/w_j)P)_{j,j+1}$ is the map acting on the $j,j+1$ tensor powers:
$$
(S(z/w_j)P)_{j,j+1} = \id{V^{[1,j-1]} } \otimes (S(z/w_j)P) \otimes \id{V^{[j+2,L+1]}}.
$$
At the same time, the space reversed transfer matrix (with right jumps) can be written as a map $V_l \otimes V^{(L)} \rightarrow V^{(L)} \otimes V_l$ defined by
$$
\mathcal{T}_{\text{rev}} =(PS(z/w_L))_{L,L+1} \circ \ldots \circ (PS(z/w_1))_{12}.
$$
If the dependence of $\mathcal{T}$ on the spectral and spin parameters needs to be emphasized, it will be written
$$
\mathcal{T}(l,z \vert \substack{m_1,\ldots,m_L \\ w_1,\ldots,w_L}) 
$$
and similarly for $\mathcal{T}_{\text{rev}}$. See Figure \ref{Tikz} for a pictorial understanding of the transfer matrix. 

We need a few more multi--site identities.

\begin{lemma}\label{LSite}
For any $u \in \mathcal{U}_q( A_n^{(1)}),$

\begin{multline}\label{Inter*}
\left(S(z/w_j)P\right)^*_{j,j+1} \Pi^{\otimes 2}_{j,j+1} B^{\otimes 2}_{j,j+1} \Ga_{j,j+1} \Delta^{(L+1)}(u) \Ga_{j,j+1} \\
= \Pi^{\otimes 2}_{j,j+1} B^{\otimes 2}_{j,j+1} \Ga_{j,j+1} \Delta^{(L+1)}(u) \Ga_{j,j+1} \left( PS(w/z_j)\right)_{j,j+1}.
\end{multline}

Furthermore, if $M_{12}$ satisfies particle conservation, i.e.
$$
\langle \xi \vert M \vert \eta \rangle = 0, \text{ if } \eta^1_i + \eta^2_i \neq \xi^1_i + \xi^2_i \text{ for some } 1\leq i \leq n+1, 
$$
then
\begin{equation}\label{GGG}
M_{12} \Ga_{13}\Ga_{23} = \Ga_{13}\Ga_{23} M_{12}, \quad M_{23}\Ga_{12}\Ga_{13} = \Ga_{12}\Ga_{13}M_{23}. 
\end{equation}
and
\begin{equation}\label{GGG2}
M_{12}\Pi^{\otimes 2}_{12} \Ga_{12}\Ga_{13}  = \Ga_{12}\Ga_{13} M_{12}\Pi^{\otimes 2}_{12} , \quad  M_{23}\Pi^{\otimes 2}_{23} \Ga_{12}\Ga_{13}  = \Ga_{12}\Ga_{13} M_{23}\Pi^{\otimes 2}_{23}. 
\end{equation}

\end{lemma}
\begin{proof}
Equation \eqref{Inter*} is true by co--associativity: if 
$$
\Delta^{(2)}(u) = \sum_{(u)} u_{(1)} \otimes u_{(2)} \otimes u_{(3)}
$$
then
$$
\Delta^{(L+1)}(u) = \sum_{(u)} \Delta^{(j-2)}(u_{(1)}) \otimes \Delta(u_{(2)}) \otimes \Delta^{(L-j-1)}(u_{(3)}),
$$
so Theorem \ref{2Site} applies with $u_{(2)}$ at the $j,j+1$ lattice sites.

For the first identity in \eqref{GGG}, note that
\begin{align*}
\langle \xi \vert M_{12} \Ga_{13}\Ga_{23}  \vert \eta \rangle &= \langle \xi \vert M_{12} \vert \eta \rangle q^{-\sum_{i=1}^n \left(\eta^1_{[1,i]} + \eta^2_{[1,i]} \right)\eta^3_{i+1}} ,\\
\langle \xi \vert \Ga_{13}\Ga_{23} M_{12}  \vert \eta \rangle &= q^{-\sum_{i=1}^n \left(\xi^1_{[1,i]} + \xi^2_{[1,i]} \right)\xi^3_{i+1}}  \langle \xi \vert M_{12} \vert \eta \rangle .
\end{align*}
Since $M_{12}$ does not act on the third tensor power and $\Ga$ is diagonal, then $\xi^3=\eta^3$. By particle conservation, if $ \langle \xi \vert M_{12} \vert \eta \rangle$ is nonzero, then for $1\leq i \leq n$, $\left(\xi^1_{[1,i]} + \xi^2_{[1,i]} \right) = \left(\eta^1_{[1,i]} + \eta^2_{[1,i]} \right)$. This shows the identity. A similar argument applies for the other identities.

\end{proof}

\begin{theorem}\label{Major} As maps from $V_l\otimes V^{(L)}$ to $V^{(L)}\otimes V_l$, and for any $u\in \mathcal{U}_q(A_n^{(1)})$,
$$
\mathcal{T}^* D^+_{\mu}(u) = D^+_{\mu}(u) \mathcal{T}_{\text{rev}}
$$
\end{theorem}
\begin{proof}
Recall that the operator $D^+(u)$ is defined by
$$
\Pi^{\otimes L+1} \circ B^{\otimes L+1} \circ \Ga^{(L+1)} \circ \Delta^{(L+1)}(u) \circ \mathrm{Ga}^{(L+1)} .
$$
It suffices to show that $\left( S(z/w_j)P \right)^*_{j,j+1} \circ D^{+(j-1)}(u) = D^{+j}(u) \circ \left( PS(w/z_j)\right)_{j,j+1}$, for that would imply that
\begin{align*}
\mathcal{T}^*D^+(u) &= \left(S(z/w_L)P\right)^*_{L,L+1} \circ \cdots \circ \left(S(z/w_1)P\right)^*_{12} \circ D^{+0}(u)\\
&= D^{+L}(u) \circ (PS(z/w_L))_{L,L+1} \circ \cdots \circ \left(PS(z/w_1)\right)_{12} \\
&= D^{+L}(u)\mathcal{T}_{\text{rev}}.
\end{align*}
Expand the gauge as 
$$
\Ga^{(L+1)} = \Ga^{[1,j-1]} \Ga^{[j+2,L+1]}  \Ga_{j,j+1} \prod_{\substack{x<j \\  j+1 < y }} \Ga_{xy} \prod_{x=1}^{j-1} \Ga_{xj} \Ga_{x,j+1}\prod_{y=j+2}^{L+1} \Ga_{jy} \Ga_{j+1,y},
$$
and note that in this expansion of $\Ga^{(L+1)}$, the first three terms commute with any operator acting on the $j,j+1$ lattice sites. This means that the desired equality $\left( S(z/w_j)P \right)^*_{j,j+1} \circ D^+(u) = D^+(u) \circ \left( PS(w/z_j)\right)_{j,j+1}$ is equivalent to the equality
\begin{align*}
&\left(S(z/w_j)P\right)^*_{j,j+1} \Pi^{\otimes 2}_{j,j+1} B^{\otimes 2}_{j,j+1} \Ga_{j,j+1}\prod_{x=1}^{j-1} \Ga_{xj}\Ga_{x,j+1} \prod_{y=j+2}^{L+1} \Ga_{jy} \Ga_{j+1,y} \Delta^{(L+1)}(u) \\
& \quad \quad\quad\quad \times \Ga_{j,j+1}\prod_{x=1}^{j-1} \Ga_{xj}\Ga_{x,j+1} \prod_{y=j+2}^{L+1} \Ga_{jy} \Ga_{j+1,y} \\
&= \Pi^{\otimes 2}_{j,j+1} B^{\otimes 2}_{j,j+1} \Ga_{j,j+1}\prod_{x=1}^{j-1} \Ga_{xj}\Ga_{x,j+1} \prod_{y=j+2}^{L+1} \Ga_{jy} \Ga_{j+1,y} \Delta^{(L+1)}(u)\\
& \quad \quad\quad\quad \times  \Ga_{j,j+1}\prod_{x=1}^{j-1} \Ga_{xj}\Ga_{x,j+1} \prod_{y=j+2}^{L+1} \Ga_{jy} \Ga_{j+1,y} \left( PS(w/z_j)\right)_{j,j+1}
\end{align*}

Since $S$ satisfies particle conservation \eqref{PCP}, as does the permutation operator $P$, by \eqref{GGG} and \eqref{GGG2}, it then suffices to show that 
 \begin{multline*}
\left(S(z/w_j)P\right)^*_{j,j+1} \Pi^{\otimes 2}_{j,j+1} B^{\otimes 2}_{j,j+1} \Ga_{j,j+1} \Delta^{(L+1)}(u) \Ga_{j,j+1} \\
= \Pi^{\otimes 2}_{j,j+1} B^{\otimes 2}_{j,j+1} \Ga_{j,j+1} \Delta^{(L+1)}(u) \Ga_{j,j+1} \left( PS(w/z_j)\right)_{j,j+1}.
\end{multline*}
But this is exactly \eqref{Inter*}.

\end{proof}

\begin{remark}\label{Explain}
Note that despite the notational similarities between  Theorem \ref{Major} and the definition of duality, it is not technically accurate to describe this as a duality result. This is because the maps send $V_l \otimes V_{m_1} \otimes \cdots \otimes V_{m_L}$ to  $V_{m_1} \otimes \cdots \otimes V_{m_L} \otimes V_l$, rather than mapping a single vector space to itself. The theorem could have been stated as a map from $V_l \otimes V_{m_1} \otimes \cdots \otimes V_{m_L}$ to itself using an equivalent form of \eqref{Inter}
$$
 {R}(z) \Delta(u) = \widetilde{\Delta(u)}  {R}(z),
$$
and without modifying the transfer matrix, but then the duality function $D(u)$ would not be the same on both sides of the equation, as it is in the definition of stochastic duality.
\end{remark}

\subsection{Interpretation as a particle system satisfying duality}\label{Iaap}
Define the operators $\mathcal{Z}$ and $\mathcal{Z}_{\text{rev}}$ on $V_{m_1} \otimes \cdots \otimes V_{m_L}$ by 
\begin{align*}
\langle \eta \vert \mathcal{Z} \vert \xi \rangle &=  \langle \Omega, \eta \vert \mathcal{T} \vert \xi, \mathrm{A} \rangle, \\
\langle \eta \vert \mathcal{Z}_{\text{rev}} \vert \xi \rangle &=  \langle \eta,\Omega \vert \mathcal{T}_{\text{rev}} \vert \mathrm{A},\xi \rangle,
\end{align*}
where $\vert \mathrm{A} \rangle, \vert \Omega \rangle \in V_l$ and the notation $\vert \xi, \mathrm{A}\rangle = \vert \xi \rangle \otimes \vert \mathrm{A} \rangle$ is used.

\begin{theorem}\label{Corollary} (a) For any $u \in \mathcal{U}_q (A_n^{(1)})$ and any $\epsilon,\zeta \in V_l,$
$$
\sum_{\substack{ \sigma \in \mathcal{B}_{m_1} \times \cdots \times \mathcal{B}_{m_L} \\ \iota \in \mathcal{B}_l }} \langle \iota,\sigma \vert \mathcal{T} \vert \eta,\zeta\rangle \langle \iota,\sigma \vert D^+(u) \vert \epsilon,\xi \rangle = \sum_{\substack{ \sigma \in \mathcal{B}_{m_1} \times \cdots \times \mathcal{B}_{m_L} \\ \iota  \in \mathcal{B}_l  }} \langle \eta,\zeta \vert D^+(u) \vert \sigma,\iota\rangle \langle \sigma,\iota \vert \mathcal{T}_{\text{rev}}\vert \epsilon,\xi\rangle.
$$
(b) For  the summation taken over $\sigma \in \mathcal{B}_{m_1} \times \cdots \times \mathcal{B}_{m_L} $, and fixing $u=u_0$,
\begin{multline*}
\sum_{\sigma} \langle \sigma \vert \mathcal{Z} \vert \eta \rangle \langle \sigma \vert D(u_0) \vert \xi \rangle  + \frac{1}{ [l]^!_q} \sum_{\substack{ \sigma \\\Omega \neq \iota \in \mathcal{B}_l}} \langle \iota,\sigma \vert \mathcal{T} \vert \eta,\mathrm{A}\rangle \langle \iota,\sigma \vert D^+(u_0) \vert \mathrm{A},\xi \rangle \\
= \sum_{\sigma} \langle \eta \vert D(u_0) \vert \sigma \rangle \langle \sigma \vert \mathcal{Z}_{\text{rev}} \vert \xi\rangle  + \frac{1}{ [l]^!_q}  \sum_{\substack{\sigma \\\Omega \neq \iota \in \mathcal{B}_l}} \langle \eta,\mathrm{A} \vert D^+(u_0) \vert \sigma,\iota\rangle \langle \sigma,\iota \vert \mathcal{T}_{\text{rev}}\vert \mathrm{A},\xi\rangle  .
\end{multline*}
(c)
Suppose that $z,m_x,w_x,{ \color{black} L}$ depend on a parameter $\mathfrak{p}$ such that for $\iota \neq \Omega$, 
$$
\sum_{\sigma}\langle \iota,\sigma \vert \mathcal{T}(l,z \vert \substack{m_1,\ldots,m_L \\ w_1,\ldots,w_L}) \vert \eta,\mathrm{A}\rangle  \rightarrow 0, \quad \quad \sum_{\sigma}\langle \sigma,\iota \vert \mathcal{T}_{\text{rev}}(l,z \vert \substack{m_1,\ldots,m_L \\ w_1,\ldots,w_L})  \vert \mathrm{A},\xi\rangle  \rightarrow 0
$$
as $\mathfrak{p}\rightarrow \infty$. 
Then in the limit $\mathfrak{p}\rightarrow \infty$,
$$\mathcal{Z}^*D(u_0)= D(u_0)\mathcal{Z}_{\text{rev}}$$ and $\mathcal{Z}$ is a stochastic operator if every $S(z/w_L)$ is stochastic.
\end{theorem}
\begin{proof}
(a) By Theorem \ref{Major}, for any $\epsilon,\zeta,$
$$
\langle \eta,\zeta \vert \mathcal{T}^*D^+(u) \vert \epsilon,\xi\rangle = \langle \eta,\zeta \vert D^+(u) \mathcal{T}_{\text{rev}}\vert \epsilon,\xi\rangle.
$$
This can be written as
$$
\sum_{\sigma,\iota} \langle \eta,\zeta \vert \mathcal{T}^* \vert \iota,\sigma\rangle \langle \iota,\sigma \vert D^+(u) \vert \epsilon,\xi \rangle = \sum_{\sigma,\iota} \langle \eta,\zeta \vert D^+(u) \vert \sigma,\iota\rangle \langle \sigma,\iota \vert \mathcal{T}_{\text{rev}}\vert \epsilon,\xi\rangle,
$$
which is equivalent to the needed statement.

(b) Take $\epsilon= \zeta = \mathrm{A}$ in part (a). Divide both sides by $[l]_q^!$ and apply Lemma \ref{Reduction} to get the result.

(c) From the explicit expression \eqref{ExpExp}, the terms $\langle \iota,\sigma \vert D^+(u_0) \vert \mathrm{A},\xi\rangle$ and $ \langle \eta,\mathrm{A} \vert D^+(u_0) \vert \sigma,\iota\rangle  $ are uniformly bounded in $\sigma$. Therefore, with the assumptions here, the $\sum_{\iota \neq \Omega}$ summation in (b) converges to $0$. This implies that $\mathcal{Z}^*D(u_0)= D(u_0)\mathcal{Z}_{\text{rev}}$.

If every $S(z/w_L)$ is stochastic, then the transfer matrix $\mathcal{T}$ is stochastic. Therefore the entries of $\mathcal{Z}$ are non--negative. The columns sum to 
$$
1 - \sum_{\iota \neq \Omega} \sum_{\eta}  \langle \iota, \eta \vert \mathcal{T} \vert \xi, \mathrm{A} \rangle,
$$
which converges to $1$. An identical argument holds for $\mathcal{T}_{\text{rev}}$.
\end{proof}

Statement (c) can be interpreted as saying that if almost surely no particles exit the lattice, then $\mathcal{Z}$ and $\mathcal{Z}_{\text{rev}}$ define Markov chains which satisfy space--reversed duality with respect to $D(u_0)$. This will hold when the lattice is the infinite line $\mathbb{Z}$, or when the lattice is finite with closed boundary conditions. One can think of $\mathcal{Z}$ and $\mathcal{Z}_{\text{rev}}$ as defining particle systems with evolution to the left and right, respectively. Furthermore, for certain parameters of the transfer matrix, the ``$\mathrm{A}$'' particles entering the lattice from the auxiliary space will not affect the evolution of the system. The next sections will elaborate on this. 

\begin{remark}
It is natural to ask if these duality results hold for open or periodic boundary conditions. For open boundary conditions, the conditions of Theorem {\color{black} \ref{Corollary}}(c) do not hold and Lemma \ref{Reduction} does not apply. For periodic boundary conditions, the operators $\mathcal{Z},\mathcal{Z}_{\mathrm{rev}}$ need to be redefined, but doing so results in Lemma \ref{Reduction} being inapplicable. There are duality results for open or periodic boundary conditions (see \cite{O} and \cite{S2}), but it is not clear if it is possible to obtain similar results from the framework here. 
\end{remark}

\subsubsection{On the infinite line}\label{OTIL}
To define the transition matrix for the particle system, we restrict the state space to states $\eta$ with finitely many particles, in the sense if $\eta$ is a particle configuration written as $\eta=(\eta_i^x)$ for $ x\in \mathbb{Z}, 1\leq i \leq n+1$, then the set
$$
\{ x \in \mathbb{Z}: \eta_1^x + \ldots + \eta_n^x \neq 0\}
$$ 
is finite. Let $W$ denote this state space. The particles of the configurations $\xi, \eta \in W$ are contained in a finite lattice, and if $M$ empty lattice sites are added on both sides of this finite lattice, and an auxiliary space is also added, then the transfer matrices $\mathcal{T}$ and $\mathcal{T}_{\text{rev}}$ can act on $\eta$ and $\xi$.

In other words, for $\xi,\eta \in W$, define $\mathcal{Z}$ (with particle jumps to the left) by 
$$
\langle \eta \vert \mathcal{Z} \vert \xi \rangle = \lim_{M\rightarrow\infty} \langle \Omega, \eta \vert \mathcal{T}^{(M)} \vert \xi, \Omega \rangle,
$$
where the superscript $(M)$ indicates that $L$ (the total number of lattice sites on which $\mathcal{T}$ acts) depends on $M$.
Similarly, define the space reversed version by
$$
\langle \eta \vert \mathcal{Z}_{\text{rev}} \vert \xi \rangle = \lim_{M\rightarrow\infty} \langle \eta,\Omega \vert \mathcal{T}_{\text{rev}}^{(M)} \vert \Omega, \xi \rangle.
$$
See Figure \ref{Adding} for an example.

Here, a lemma of the stochastic matrices will be needed. 

\begin{lemma}\label{Univ}
(a) Suppose $S(z)$ is stochastic. Then there exists a fixed $\kappa \in [0,1)$ such that for all $\alpha \neq \Omega$,
$$
\langle {\alpha,\Omega} \vert S(z) \vert {\alpha,\Omega} \rangle \leq \kappa.
$$
In words, this means that for an input $\alpha$ in the auxiliary space, the probability of no particles settling in at the lattice site is at most $\kappa$.

(b) For any $\epsilon,\zeta$, the following formulas for $\mathcal{Z}$ and $\mathcal{Z}_{\text{rev}}$ also hold:
\begin{align*}
\langle \eta \vert \mathcal{Z} \vert \xi \rangle &= \lim_{M\rightarrow\infty} \langle \Omega, \eta \vert \mathcal{T}^{(M)} \vert \xi, \zeta \rangle \\
\langle \eta \vert \mathcal{Z}_{\text{rev}} \vert \xi \rangle &= \lim_{M\rightarrow\infty} \langle \eta,\Omega \vert \mathcal{T}_{\text{rev}}^{(M)} \vert \epsilon, \xi \rangle.
\end{align*}

(c) For $\iota \neq \Omega$, 
\begin{align*}
\sum_{\sigma} \langle \iota, \sigma \vert \mathcal{T}^{(M)} \vert \eta, \zeta \rangle & =  O\left(\kappa^M\right), \\
\sum_{\sigma} \langle \sigma,\iota \vert \mathcal{T}_{\text{rev}}^{(M)}\vert \epsilon,\xi\rangle & = O\left(\kappa^M\right).
\end{align*} 

\end{lemma}
\begin{proof}
(a) Suppose this were not true. Then there must be a $\alpha \neq \Omega$ such that
$$
\langle {\alpha,\Omega} \vert S(z) \vert {\alpha,\Omega} \rangle = 1.
$$
However, by Theorem \ref{StochasticFusion} and \eqref{Stol=1}, this cannot hold.

(b) By part (a), at each lattice site there is a probability of at most $\kappa<1$ that no particles will settle at that lattice site. So if $\zeta$ particles enter at the right boundary, the probability that no particle interacts with $\xi$ is asymptotically $O(\kappa^M)$.
Therefore, 
$$
\left| \langle \eta \vert \mathcal{Z} \vert \xi \rangle - \langle \Omega, \eta \vert \mathcal{T}^{(M)} \vert \xi, \zeta \rangle \right| = O (\kappa^M),
$$
which converges to $0$. The same argument applies for $\mathcal{Z}_{\text{rev}}$.

(c) By a similar argument as in (b), the probability that a particle in $\eta$ makes its way to the left boundary is asymptotically $O(\kappa^M)$.

\end{proof}

\begin{figure}
\caption{Here, the particles in the configuration defined by $\xi$ are contained in two lattice sites, and the particles in the configuration defined by $\eta$ are contained in three lattice sites. The value of $M$ is $1$, so one empty lattice site is added on both sides. Again, solid lines with no arrows indicate no particles.}
\begin{center}

\includegraphics{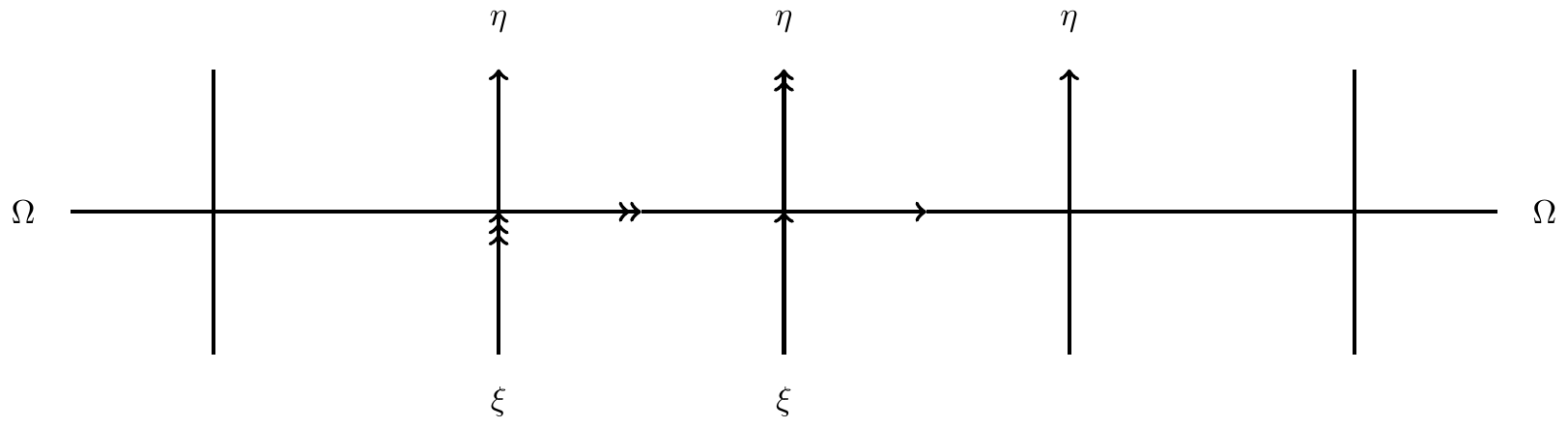}


\end{center}
\label{Adding}
\end{figure}

Note that if Lemma \ref{Univ}(a) were not true, then there would be a positive probability that particles could exit the lattice at infinity, which would violate particle conservation.

\begin{theorem}\label{DualThm}
The process $\mathcal{Z}$ satisfies {\color{black} space--reversed self--duality} with respect to $D(u_0)$, given explicitly by {\color{black}\eqref{ExpExp}}:
\begin{equation*}
\langle \xi \vert D(u_0) \vert \eta \rangle=
\prod_{x\in \mathbb{Z}} \Q{\eta_1^x}^! \cdots\Q{\eta_{n+1}^x}^!  \prod_{i=1}^{n}     \binomQ{  m_x - \eta_{[1,n-i]}^x - \xi_{[1,i]}^x   }{ \eta_{n+1-i}^x }   q^{ -\xi_i^x (\sum_{z>x} 2\eta_{[1,n+1-i]}^z + \eta_{[1,n+1-i]}^x)}.
\end{equation*}
In other words,
$$
\mathcal{Z}^*D(u_0) = D(u_0) \mathcal{Z}_{\text{rev}}.
$$

\end{theorem}
\begin{proof}
By Theorem \ref{Univ}(c), the conditions of Theorem \ref{Corollary} hold. This immediately implies the result. 
\end{proof}

\begin{remark}
Notice the duality function does not depend on the horizontal spin parameter $l$. This is a similar phenomenon to that of \cite{CGRS}. In that framework, if one takes a $l$--th degree polynomial in the Casimir element, the resulting process can have up to $l$ particles jumping at a time. However, the duality function does not depend on the choice of the central element, and therefore does not depend on $l$.
\end{remark}

\subsubsection{Closed boundary conditions}\label{CBC}
The transfer matrices can be used to define a particle system on a lattice with closed boundary conditions satisfying space--reversed duality with respect to $D(u_0)$.

Assume here that $q>1$. Consider the limit of 
$$
\mathcal{T}(l,z \vert \substack{m_1,\ldots,m_L \\ w_1,\ldots,w_L})  \text{ and } \mathcal{T}_{\text{rev}}(l,z \vert \substack{m_1,\ldots,m_L \\ w_1,\ldots,w_L}) 
$$
as
$$
m_1=m_2=m_{L-1}=m_L \rightarrow\infty, \quad w_1=w_2 = w_{L-1}=w_L \rightarrow \infty,
$$
where the limits in the spectral parameters $w$ are taken first. Then define
\begin{align*}
\langle \eta \vert \mathcal{Z} \vert \xi \rangle &= \lim \langle \Omega, \eta \vert \mathcal{T} \vert \xi, \mathrm{A} \rangle, \\
\langle \eta \vert \mathcal{Z}_{\text{rev}} \vert \xi \rangle &= \lim \langle \eta,\Omega \vert \mathcal{T}_{\text{rev}} \vert \epsilon, \mathrm{A} \rangle,
\end{align*}
where the limit as taken as in the one above. This results in $\mathcal{Z}$ and $\mathcal{Z}_{\text{rev}}$ as operators on the space 
$$
\bar{V}_{\infty} \otimes \bar{V}_{\infty}\otimes V_{m_3} \otimes \cdots \otimes V_{m-2} \otimes \bar{V}_{\infty} \otimes \bar{V}_{\infty}.
$$
\begin{proposition}\label{412}
The operators $\mathcal{Z}$ and $\mathcal{Z}_{\text{rev}}$ are stochastic and satisfy
$$
\mathcal{Z}^*D(u_0)= D(u_0)\mathcal{Z}_{\text{rev}}.
$$
\end{proposition}
\begin{proof}
By Theorem \ref{Stuff}(c), the conditions of Theorem \ref{Corollary} hold. This implies the proposition.
\end{proof}

The stochastic operators $\mathcal{Z}$ and $\mathcal{Z}_{\text{rev}}$ can be interpreted as the transition matrices of an interacting particle system with closed boundary conditions. To see this, notice by Theorem \ref{Stuff}(c), with probability $1$, the particles entering the lattice along the auxiliary space all settle in at the endpoints, and do not interact with the other particles. See Figure \ref{ClosedFig} for an example. Therefore, $\mathcal{Z}$ and $\mathcal{Z}_{\text{rev}}$ can be viewed as stochastic operators on 
$$
\bar{V}_{\infty}\otimes V_{m_3} \otimes \cdots \otimes V_{m-2} \otimes \bar{V}_{\infty} .
$$
One way of thinking about this is that they only act on particle configurations on the lattice $\{2,\ldots,L-1\}$. Since particles do not exit this lattice, $\mathcal{Z}$ and $\mathcal{Z}_{\text{rev}}$ can be seen as transition matrices of an interacting particle system on the lattice $\{2,\ldots,L-1\}$ with closed boundary conditions.

By applying \eqref{mLimit}  to \eqref{ExpExp}, the duality functional can be written as
\begin{multline*}
\langle \xi \vert D(u_0) \vert \eta \rangle =\mathrm{const}\prod_{x \in \{1,2,L-1,L\}}\prod_{i=1}^n q^{-\xi_i^x \sum_{y>x} 2\eta^y_{[1,n+1-i]}} \\
\times  \prod_{x=3}^{L-2} \Q{\eta_1^x}^! \cdots\Q{\eta_{n+1}^x}^!  \prod_{i=1}^{n}     \binomQ{  m_x - \eta_{[1,n-i]}^x - \xi_{[1,i]}^x   }{ \eta_{n+1-i}^x }   q^{ -\xi_i^x (\sum_{y>x} 2\eta_{[1,n+1-i]}^y + \eta_{[1,n+1-i]}^x)}  .
\end{multline*}
Note that because $\xi^1$ and $\eta^L$ are the particle configurations with no particles, the duality simplifies to
\begin{multline*}
\langle \xi \vert D(u_0) \vert \eta \rangle =\mathrm{const} \prod_{i=1}^n q^{-\xi_i^2 \sum_{y=3}^{L-1} 2\eta^y_{[1,n+1-i]}} \\
\times  \prod_{x=3}^{L-2} \Q{\eta_1^x}^! \cdots\Q{\eta_{n+1}^x}^!  \prod_{i=1}^{n}     \binomQ{  m_x - \eta_{[1,n-i]}^x - \xi_{[1,i]}^x   }{ \eta_{n+1-i}^x }   q^{ -\xi_i^x (\sum_{y=x+1}^{L-1} 2\eta_{[1,n+1-i]}^y + \eta_{[1,n+1-i]}^x)} .
\end{multline*}
Again, $D(u_0)$ does not depend on the particle configurations at the lattice sites $1$ and $L$, so can also be viewed as a duality functional on the particle configurations on the lattice $\{2,\ldots,L-1\}$.

\begin{remark}\label{Closed}
For a \textit{totally} asymmetric process, it is necessary to take $m\rightarrow\infty$ at the endpoints in order for this sort of duality to hold. To see this, suppose $\xi$ evolves to the left, with initial condition consisting of one particle at the lattice site $L-1$. Suppose $\eta$ evolves to the right, with initial condition consisting of particles contained in $\{3,\ldots,L-2\}$. If $\xi$ evolves with $\eta$ fixed, then eventually $\xi$ has one particle at lattice site $2$, so all particles in $\eta$ will be counted. If $\eta$ evolves with $\xi$ fixed, then in order for all particles to be counted, they must all occupy site $L-1$. This is only possible if an arbitrary number of particles can occupy lattice site $L-1$.
\end{remark}

\begin{figure}
\caption{In this example, $l=4$. The top diagram shows $\mathcal{T}$ (with particle configuration denoted by $\xi$) and the bottom shows $\mathcal{T}_{\text{rev}}$ (with particle configuration denoted by $\eta$). The spectral parameter $w$ in the leftmost vertex is set to infinity so that particles entering the lattice in $\mathcal{T}_{\text{rev}}$ do not interfere with the rest of the lattice. On the second vertex from the left, $w=\infty$ so that particles in $\mathcal{T}$ do not enter the leftmost vertex. Similar statements hold for the two vertices on the right.}
\begin{center}

\includegraphics{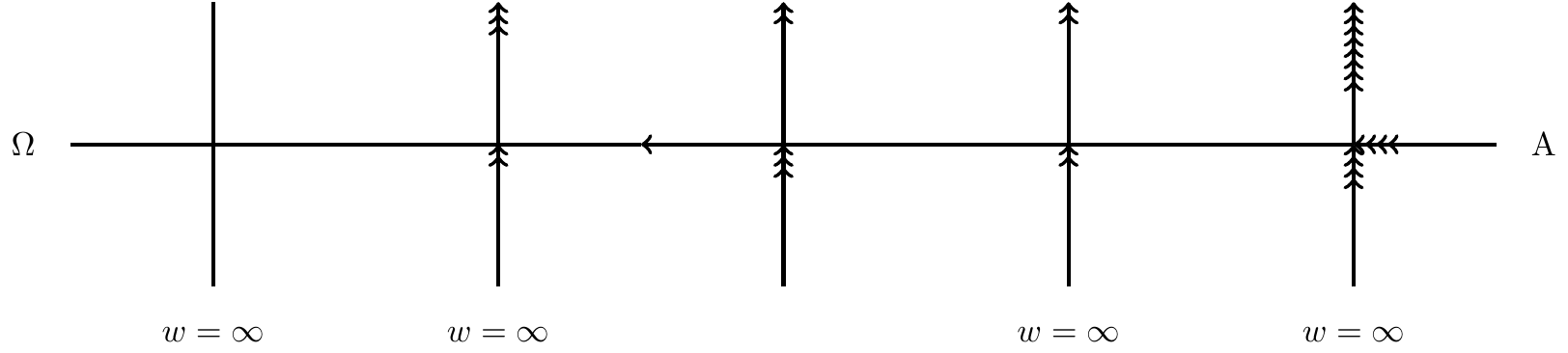}

\includegraphics{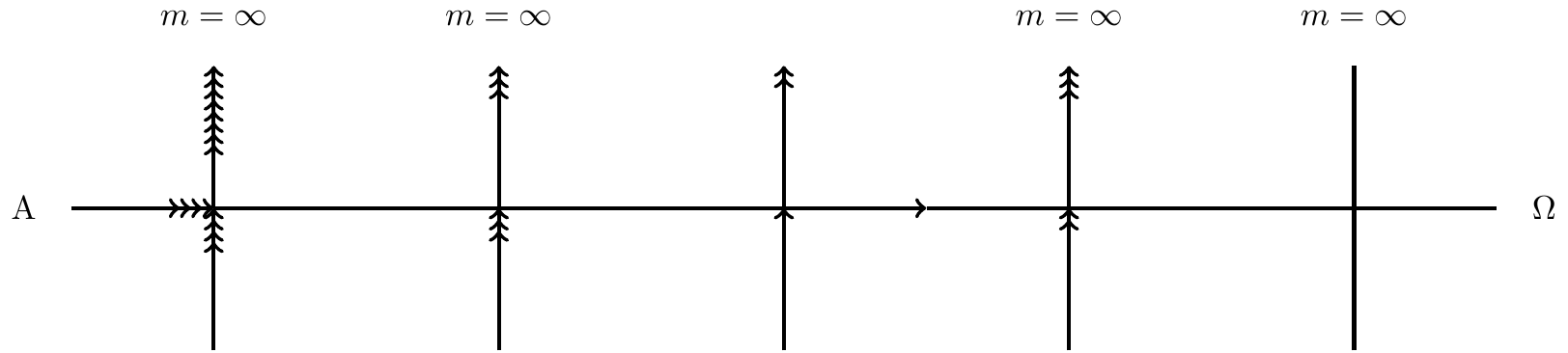}

\end{center}
\label{ClosedFig}
\end{figure}

\subsubsection{Continuous--time zero range process}\label{CTP}
In the $m\rightarrow\infty$ limit, the stochastic matrix $S(z)$ can be used to define a continuous--time zero range process, either on the infinite lattice or on a finite lattice with closed boundary conditions. This definition is different than the $\lambda\rightarrow\mu$ degeneration in the $q$--Hahn Boson process. In that case, it would not have been a priori obvious that after the degeneration, the off--diagonal entries would be non--negative. On the other hand, with the construction here, non--negativity will always hold.

In the definition of $\mathcal{T}$, suppose that some $w_j$ is fixed and all other values $w_i$ are taken to infinity. Furthermore, take the limit of all $m_1,\ldots,m_L\rightarrow \infty$. That is, define
\begin{align*}
\mathcal{L}^x &= \lim_{m\rightarrow\infty} (S(0)P)_{12} \circ \cdots \circ (S(0)P)_{x-1,x} \circ (S(z/w_x)P)_{x,x+1} \circ (S(0)P)_{x+1,x+2} \circ \cdots \circ (S(0)P)_{L,L+1} \\
&=\mathcal{T}(l,z \vert \substack{\infty,\ldots,\infty,\infty,\infty,\ldots \infty \\ \infty,\ldots,\infty,w_x,\infty,\ldots,\infty}) 
\end{align*}
and set
$$
\mathcal{L} = \frac{1}{L}\sum_{x=1}^L \mathcal{L}^x.
$$
Similarly, define
\begin{align*}
\mathcal{L}^x_{\text{rev}} &= \lim_{m\rightarrow \infty} (PS(0))_{L,L+1} \circ \cdots \circ (PS(0))_{x+1,x+2} \circ (PS(z/w_x))_{x,x+1} \circ (PS(0))_{x-1,x} \circ \cdots \circ  (PS(0))_{12} \\
&=\mathcal{T}_{\text{rev}}(l,z \vert \substack{\infty,\ldots,\infty,\infty,\infty,\ldots \infty \\ \infty,\ldots,\infty,w_x,\infty,\ldots,\infty}) 
\end{align*}
and set
$$
\mathcal{L}_{\text{rev}} = \frac{1}{L}\sum_{x=1}^L \mathcal{L}^x_{\text{rev}}.
$$
By Theorem \ref{Stuff}(c), each $\mathcal{L}^x$ and $\mathcal{L}^x_{\text{rev}}$ can be viewed as a local stochastic operator, in which only particles from lattice site $x$ can jump, and all other particles cannot jump.

The operators $\mathcal{L}$ and $\mathcal{L}_{\text{rev}}$ are maps
\begin{align*}
\mathcal{L}&: \bar{V}_{\infty}^{\otimes L} \otimes V_l \rightarrow V_l \otimes \bar{V}_{\infty}^{\otimes L},\\
\mathcal{L}_{\text{rev}} &:  V_l \otimes \bar{V}_{\infty}^{\otimes L} \rightarrow \bar{V}_{\infty}^{\otimes L} \otimes V_l.
\end{align*}
Then define
\begin{align*}
\langle \eta \vert \mathcal{Y} \vert \xi \rangle &= \langle \Omega, \eta \vert \mathcal{L} \vert \xi, \mathrm{A} \rangle, \\
\langle \eta \vert \mathcal{Y}_{\text{rev}} \vert \xi \rangle &=  \langle \eta,\Omega \vert \mathcal{L}_{\text{rev}} \vert \epsilon, \mathrm{A} \rangle,
\end{align*}
\begin{proposition}\label{AAA2}
The operators $\mathcal{Y}$ and $\mathcal{Y}_{\mathrm{rev}}$ are stochastic on $\bar{V}_{\infty}^{\otimes L}$ and satisfy
$$
\mathcal{Y}^*D(u_0) = D(u_0)\mathcal{Y}_{\mathrm{rev}}.
$$
If $\mathrm{Id}$ denotes the identity matrix, then 
$$
(\mathcal{Y}^*- \mathrm{Id})D(u_0) = D(u_0)(\mathcal{Y}_{\mathrm{rev}} - \mathrm{Id}).
$$
\end{proposition}
\begin{proof}
{\color{black} This follows immediately from Proposition \ref{412}. }
\end{proof}
Because  $\mathcal{Y}$ and $\mathcal{Y}_{\mathrm{rev}}$ are stochastic, the operators $(\mathcal{Y}^*- \mathrm{Id})$ and $(\mathcal{Y}_{\mathrm{rev}} - \mathrm{Id})$ satisfy the property that the off--diagonal entries are non--negative and the columns sum to $0$. Therefore, they are generators for a continuous--time process on particle configurations on the lattice $\{2,\ldots,L-1\}$. Because they can be written as a sum of local generators, the process is a zero--range process. Theorem \ref{Stuff}(b) establishes that if $w_x=0$ in the definition of $\mathcal{L}^x$, the process is nontrivial. Indeed, it will be seen below that for $l=1$, the process is the multi-species $q$--Boson process of \cite{T}.

\begin{remark}
The use of subtracting the identity matrix to obtain a continuous--time process with duality from a discrete--time process with duality is not new: see e.g. section 5.2 of \cite{S}.
\end{remark}

{\color{black}
\begin{remark}\label{InfiniteParticles}
A continuous--time zero--range process must allow an arbitrary number of particles to occupy a lattice site. Indeed, the jump rates for a particle jump from lattice site $x$ to $x+1$ cannot depend on the occupancy at the site $x+1$, which implies that there can not be a constraint on the number of particles allowed to occupy site $x+1$. 
\end{remark}
}

\section{Descriptions of processes}\label{Dop}
This section describes the processes that can be defined from the stochastic $\mathcal{U}_q(A_n^{(1)})$ vertex model. See Figure \ref{Deg2}.

\subsection{At $z=q^{l-m}$}\label{ZZZ}
At $z=q^{l-m}$, the operator $S(z)$ acting on $V_l \otimes V_m$ is meromorphic in $\lambda = q^{-l}$ and $\mu=q^{-m}$. Then Theorem \ref{DualThm} can be viewed as an identity holding on ${\color{black} W }\subset \left(\mathbb{Z}_{\geq 0}^{n}\right)^{\otimes \infty}$ which depends on the parameters $l,\{m_x\}_{x \in \mathbb{Z}}$. In particular, both sides of the equation are also meromorphic functions in the complex variables $\lambda^{-1}=q^l,\mu_x^{-1}=q^{m_x}$. Since the equality holds for all $l,m_x\geq 0$, it holds on the set $\{q,q^2,q^3,\ldots\}$, which as $0$ as a limit point. Therefore, the equality holds for all values of $\lambda$ and $\mu_x$.

\begin{theorem}\label{qHb}
The $n$--species discrete--time $q$--Hahn Boson process satisfies space--reversed self--duality with respect to the function $D_{\mu}$ in \eqref{ExpExp}, given explicitly by 
$$
\langle \xi \vert D_{\mu}(u_0) \vert \eta \rangle=
\prod_{x\in \mathbb{Z}} \Q{\eta_1^x}^! \cdots \Q{\eta_n^x}^!\Q{m_x - \eta_{[1,n]}^x}^!  \prod_{i=1}^{n}     \binomQ{  m_x - \eta_{[1,n-i]}^x - \xi_{[1,i]}^x   }{ \eta_{n+1-i}^x }   q^{ -\xi_i^x (\sum_{z>x} 2\eta_{[1,n+1-i]}^z + \eta_{[1,n+1-i]}^x)},
$$
where $m_x$ is defined by $q^{-m_x}=\mu_x$.
Here, $\eta$ evolves to the left and $\xi$ evolves to the right.
\end{theorem}

This theorem implies dualities for two degenerations:

\begin{corollary}\label{qHb1} The $n$--species continuous--time $q$--Hahn Boson process satisfies space--reversed self--duality with respect to $D_{\mu}$. 
\end{corollary} 
\begin{proof}
Observe that the duality $D$ does not depend on $\lambda$ and only on $\mu$. Therefore if we take Theorem \ref{qHb} and differentiate with respect to $\lambda$, then
$$
\sum_{x} \mathcal{L}_x^* D = \sum_{x} D \mathcal{L}^x_{\text{rev}}
$$
\end{proof}

The next corollary was previously shown in Theorem 2.5(b) of \cite{K}.

\begin{corollary}\label{ReProve} The $n$--species $q$--TAZRP of \cite{T} is dual to its space--reversed version with respect to the duality function $D_0$ defined by
$$
\langle \xi \vert D_0 \vert \eta \rangle =\prod_{x\in \mathbb{Z}} \prod_{i=1}^{n}      q^{ \xi_i^x \left(\sum_{y \leq x} 2\eta_{[1,n+1-i]}^y \right) } .
$$
Here, the $\eta$ process evolves to the left and the $\xi$ process evolves to the right.
\end{corollary}
\begin{proof}
One proof would involve taking $m_x\rightarrow -\infty$, but this requires knowledge of the asymptotics of the $q$--Gamma function at $-\infty$. Instead, assume that $q>1$ and now suppose that all $m_x\rightarrow\infty$, which means that $\mu\rightarrow 0.$ The asymptotic analysis was essentially already done in the proof of Proposition 5.2(b), with the only exception being

$$
\lim_{m\rightarrow\infty} (q-q^{-1})^{m-k} q^{-(m-k)(m-k+1)/2} [m-k]_q^! = \lim_{m\rightarrow \infty} \prod_{l=1}^{m-k} q^{-l}(q^{l}-q^{-l}) = \prod_{l=1}^{\infty} (1-q^{-2l}),
$$
reflecting that $q>1$. This results in duality with respect to the function
$$
\prod_{x\in \mathbb{Z}} \prod_{i=1}^{n}      q^{ \xi_i^x \left(\sum_{y < x} 2\eta_{[1,n+1-i]}^y\right) }.
$$
Since the process is translation invariant in the limit when all $m_x\rightarrow\infty$, this proves the corollary for $q>1$. By analytic continuation, it also holds for $0<q<1$.

\end{proof}

\subsection{$l=1$}\label{PPP} Recall from section \ref{atl=1} that when $l=1$,
$$
(q^{m+1}-z)S(z)_{\epsilon_j,\beta}^{\epsilon_k,\delta}  = 
\begin{cases}
q^{2\beta_{[1,k]}  -  m + 1} \left( 1 - q^{-2\beta_k+m-1}z\right), & \text{ if } k = j , \\
-q^{2\beta_{[1,k-1]} - m + 1} \left( 1 - q^{2\beta_k} \right), & \text{ if } k > j , \\
-q^{2 \beta_{[1,k-1]}  } z \left( 1 - q^{2\beta_k} \right), & \text{ if } k < j .
\end{cases}
$$
This section will consider the various processes that can be defined from $S(z)$ when $l=1$.

\subsubsection{Discrete--time process with blocking}
The stochastic operator $\mathcal{Z}$ from Section \ref{Iaap} defines a discrete--time interacting particle system in which at most $m$ particles of $n$ different species can occupy a site. The update is defined sequentially, and at most one particle can jump from a site. To understand the dynamics, first consider the degenerate cases when $z\rightarrow 0$ or $z\rightarrow \infty$.

When $z\rightarrow \infty$, corresponding to $w\rightarrow 0$, the limit is
$$
\lim_{z\rightarrow \infty} S(z)_{\epsilon_j,\beta}^{\epsilon_k,\delta} = 
\begin{cases}
q^{2\beta_{[1,k-1]}  }, & \text{ if } k = j , \\
0, & \text{ if } k > j , \\
q^{2 \beta_{[1,k-1]}  }  \left( 1 - q^{2\beta_k} \right), & \text{ if } k < j .
\end{cases}
$$
This is stochastic for $0<q<1$. In this case, the ordering of the species of particles is more apparent, and there is a verbal description of the model, which is similar to the verbal description of the multi--species $\mathrm{ASEP}(q,j)$ in \cite{K}. Particles have a desire to jump in the direction of movement (left for $\mathcal{Z}$ and right for $\mathcal{Z}_{\text{rev}}$). Particles with smaller indices are considered to have higher mass (or higher class, or higher priority) than particles with larger indices. The species $n+1$ particles are considered to be holes.  For example, a particle configuration at a site indexed by $(3,1,1)$ has three particles of the heaviest mass and one hole.  If a particle of species $j$ enters a lattice site, then no particle of species $j+1,\ldots,n+1$ can exit, because those have smaller mass. However, the particles of species $1,\ldots,j-1$ have higher mass, so their inclination to jump is higher than the species $j$ particle. The species $j$ particle asks each higher mass particle if it would like to jump, starting from species $1$. Each particle says ``no'' with probability $q^2$ and says ``yes'' with probability $1-q^2$. If the particle says ``no'' then the species $j$ particle proceeds to ask the next particle. If the particle says ``yes'', then the particle blocks the jump of the species $j$ particle and jumps itself instead. If the species $j$ particle receives an answer of ``no'' from all particles of species $1,\ldots,j-1$, then it is finally allowed to jump. See Figure \ref{Permit} for an example.

Notice that a species $1$ particle entering a lattice site also exits that lattice with probability one. If the lattice is infinite, then infinitely many of the lattice sites must satisfy $w\neq 0$, or else Lemma \ref{Univ}(a) will not hold,  since there is a positive probability that a particle could jump forever in one direction.

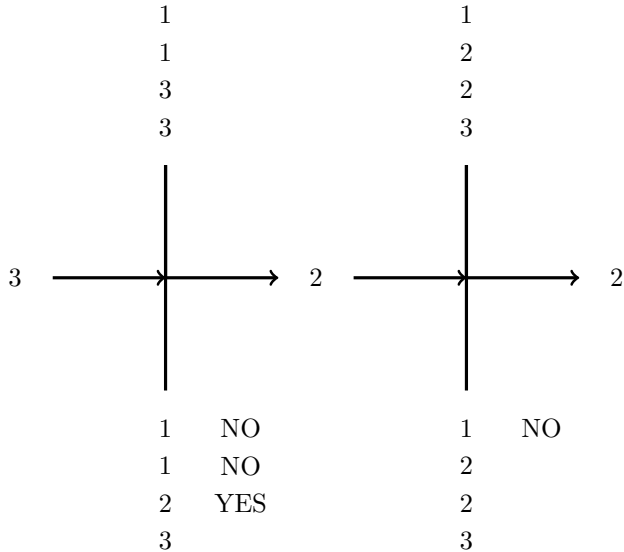
\begin{figure}
\caption{Each number represents a particle with species indicated by that number. In this example, the species $3$ particle enters the left vertex and is blocked by a species $2$ particle, which happens with probability $q^4(1-q^2)$. The species $2$ particle then enters the right vertex and is allowed to jump, which happens with probability $q^2$. }
\begin{center}

\begin{tikzpicture}
\draw [->,very thick](-1.5,0) -- (0,0) ;
\draw [->,very thick] (0,0) -- (1.5,0) ;
\draw [very thick](0,-1.5) -- (0,1.5) ;
\draw [very thick](0,0) -- (0,1.5) ;
\draw (-2,0) node {$\mathrm{3}$};
\draw (0,-2) node {$\mathrm{1}$};
\draw (0,-2.5) node {$\mathrm{1}$};
\draw (0,-3) node {$\mathrm{2}$};
\draw (0,-3.5) node {$\mathrm{3}$};
\draw (1,-2) node {NO};
\draw (1,-2.5) node {NO};
\draw (1,-3) node {YES};
\draw (0,3.5) node {${1}$};
\draw (0,3) node {${1}$};
\draw (0,2.5) node {${3}$};
\draw (0,2) node {${3}$};
\draw (2,0) node{2};
\draw [->,very thick](2.5,0) -- (4,0) ;
\draw [->,very thick] (4,0) -- (5.5,0) ;
\draw [very thick](4,-1.5) -- (4,1.5) ;
\draw [very thick](4,0) -- (4,1.5) ;
\draw (4,-2) node {$\mathrm{1}$};
\draw (4,-2.5) node {$\mathrm{2}$};
\draw (4,-3) node {$\mathrm{2}$};
\draw (4,-3.5) node[circle] {$\mathrm{3}$};
\draw (5,-2) node {NO};
\draw (4,3.5) node {${1}$};
\draw (4,3) node {${2}$};
\draw (4,2.5) node {${2}$};
\draw (4,2) node {${3}$};
\draw (6,0) node{2};
\end{tikzpicture}
\end{center}
\label{Permit}
\end{figure}

Also notice that $S(\infty)$ does not depend on $m$, as shown in Theorem \ref{Stuff}(b), so it is sensible to take the $m\rightarrow\infty$ limit at every lattice site. The description of the dynamics is the same, but arbitrarily many particles may occupy each lattice site. 

In the $z\rightarrow 0$ limit, then 
$$
S(0)_{\epsilon_j,\beta}^{\epsilon_k,\delta} = 
\begin{cases}
q^{-2\beta_{[k+1,n+1]}   } , & \text{ if } k = j , \\
q^{-2\beta_{[k+1,n+1]}   } \left(  1 - q^{-2\beta_k} \right), & \text{ if } k > j , \\
0, & \text{ if } k < j .
\end{cases}
$$
This is stochastic for $q>1$, and is in fact the same dynamics as the $z\rightarrow\infty$ situation, but with the ordering of the species reversed. The two cases $z=0$ and $z\rightarrow\infty$ can be thought of as cases of \textit{strong blocking}, due to lower mass particles being completely forbidden to jump, as shown in Theorem \ref{Stuff}(a). The values of $z$ in $(0,\infty)$ are then an interpolation between the two extreme cases. The most ``intermediate'' value in $(0,\infty)$ is at $z=q^{1-m}$, when the jump rates out of the lattice site are independent of the input. This is the case in which parallel jumps are possible, so can be viewed as \textit{no blocking}. 

\subsubsection{Continuous--time zero--range process}
By the general framework in section \ref{CTP}, $S(z)$ defines a continuous--time zero--range process. As explained in Remark \ref{InfiniteParticles}, in order to define a continuous--time zero--range process, all values of $m_x$ need to be taken to infinity, and for $z \in (0,\infty)$ and $q>1$,
$$
\lim_{m\rightarrow\infty} S(z)_{\epsilon_j,\beta}^{\epsilon_k,\delta} = 
\begin{cases}
1 , & \text{ if } k =n+1  , \\
0, & \text{ if } k < n+1.
\end{cases}
$$
And for $q<1$, the limit is no longer stochastic. The nontrivial case only occurs when $z\rightarrow\infty$ before the limit $m\rightarrow\infty$, and in this case the jump rates are precisely the jump rates for the multi--species $q$--Boson process of \cite{T}; see the first item of Remark \ref{SimpleCase}. Therefore, this provides another proof of Theorem 2.5(b) of \cite{K}, which was already given another proof in Corollary \ref{ReProve}.

\subsection{$m=1$}
By section \ref{SSm=1}, the entries of $S(z)$ when $m=1$ are
\begin{equation*}
S(z)_{\alpha,\epsilon_j}^{\gamma,\epsilon_k} 
= 1_{ \{\alpha + \epsilon_j   = \delta +  \epsilon_k  \}} \times 
\begin{cases}
q^{l-2\alpha_{[1,k-1]} +1}\dfrac{1-q^{-2\alpha_k+l-1}z}{q^{l+1}-z}, & \text{ if } k=j \\
-q^{2l-2\alpha_{[1,k]}   }   \dfrac{z(1-q^{2\alpha_k})}{q^{l+1}-z}, & \text{ if } k >j,\\
-q^{l-2\alpha_{[1,k]} +1  } \dfrac{1-q^{2\alpha_k}}{q^{l+1}-z}, & \text{ if } k<j.
\end{cases}
\end{equation*}
Consider the same degenerations as in the $l=1$ case of the previous section.
In the limit $z\rightarrow \infty$, 
\begin{equation*}
\lim_{z \rightarrow \infty}S(z)_{\alpha,\epsilon_j}^{\gamma,\epsilon_k} 
= 1_{ \{\epsilon_j  + \beta = \epsilon_k + \delta \}} \times 
\begin{cases}
q^{2l-2\alpha_{[1,k]} }, & \text{ if } k=j \\
q^{2l-2\alpha_{[1,k]}   }   (1-q^{2\alpha_k}), & \text{ if } k >j,\\
0, & \text{ if } k<j.
\end{cases}
\end{equation*}
When  $z=0$, then 
$$
S(0)_{\alpha,\epsilon_j}^{\gamma,\epsilon_k} 
= 1_{ \{\epsilon_j  + \beta = \epsilon_k + \delta \}} \times 
\begin{cases}
q^{-2\alpha_{[1,k-1]} } , & \text{ if } k=j \\
0, & \text{ if } k >j,\\
q^{-2\alpha_{[1,k-1]}   } ( 1- q^{-2\alpha_k} ), & \text{ if } k<j.
\end{cases}
$$

\subsection{Conjecture for general $l,m,z$}
For generic values of $l,m$ and $z$, we know that a well--defined process exists in both discrete and continuous--time, and that duality holds on both the infinite line and for closed boundary conditions. Each lattice site can hold up to $m$ particles, and up to $l$ particles may jump at a time. Furthermore, for $z=0$ and $z\rightarrow\infty$, the strong blocking phenomenon occurs again, due to Theorem \ref{Stuff}(a). At $z=q^{l-m}$, parallel updates occur, as shown in \cite{KMMO}. 

Due to the degeneration of multi--species $\textrm{ASEP}(q,j)$ to the multi--species $q$--Boson shown in \cite{K}, it is not unreasonable to conjecture that a generalization must hold for all values of $l$. Namely, for each $l\geq 1$, there should exist a central element $C_l$ of $\mathcal{U}_q(\mathfrak{gl}_{n+1})$ such that the framework of \cite{CGRS} produces a continuous--time asymmetric exclusion process in which up to $2j$ particles may occupy a site and up to $l$ particles may jump simultaneously. In the limit $j\rightarrow\infty$, the process should degenerate to the same totally asymmetric continuous--time zero range process produced by Section \ref{CTP}. Duality results should hold for both the asymmetric exclusion process and the totally asymmetric zero range process.

\section{``Direct'' results for multi--species $q$--Hahn Boson}\label{DIRECTDUA}
Theorem \ref{qHb} shows that the multi--species $q$--Hahn Boson process satisfies space--reversed self--duality with respect to $D_{\mu}$. Taking the limit $\mu\rightarrow 0$ degenerates the multi--species $q$--Hahn Boson process to the multi--species $q$--Boson process, and shows that the latter process satisfies space--reversed self--duality with respect to $D_0$. 

It turns out that the multi--species $q$--Hahn Boson process is also dual with respect to $D_0$, even before taking the degenerations in the process. This statement will be proved in this section through direct means, as it is unclear how to prove it using algebraic machinery.  First, start with some identities.

\subsection{Identities}

Given $\beta,\gamma \in \mathbb{Z}^n$, recall the notation that $\vert \beta\vert = \beta_1 + \ldots + \beta_n$, write $\gamma \leq \beta$ to mean $\gamma_i\leq \beta_i$ for $1\leq i \leq n$, and that
$$
\chi_{\beta,\gamma} = \sum_{1\leq i<j \leq n} (\beta_i-\gamma_i)\gamma_j,
$$

\begin{lemma}\label{qBinomialLemma}
The following identities hold:
\begin{align*}
\binomq{n}{k-1} + q^k \binomq{n}{k} &= \binomq{n+1}{k} ,\\
\sum_{\gamma_1+\gamma_2=l} \binomq{\eta_1}{\gamma_1} \binomq{\eta_2}{\gamma_2} q^{(\eta_1-\gamma_1)\gamma_2} q^{\xi_1l + \xi_2\gamma_1}&= \binomq{\eta_1+\eta_2}{l}q^{l(\xi_1+\xi_2)} ,\\
\sum_{\substack{\vert \gamma\vert=m}} q^{\chi_{\eta,\gamma}} q^{\sum_{i=1}^{n} \xi_i\gamma_{[1,n+1-i]}}  \prod_{i=1}^{n} \binomq{\eta_i}{\gamma_i}   &= \binomq{\vert \eta\vert}{m} q^{m\vert \xi \vert}.
\end{align*}
\end{lemma}
\begin{proof}
The first identity is not new (see e.g. 10.0.3  of \cite{AAR}), but can be easily seen to follow from
\begin{align*}
\frac{ \q{n} }{ \q{k-1} \q{n-k+1} } + \frac{ q^k \q{n} }{ \q{k} \q{n-k}  } &= \frac{ \q{n} }{ \q{k} \q{n-k+1} } \left(  (1-q^k) + q^k (1 - q^{n-k+1})\right) \\
&= \frac{ \q{n+1} }{ \q{k} \q{n-k+1} } .
\end{align*}

The second identity follows by an induction argument on $\eta_1+\eta_2$. By the induction hypothesis and the first identity, 
\begin{align*}
&\binomq{\eta_1+\eta_2}{l} q^{l(\xi_1+\xi_2)}= \binomq{\eta_1+\eta_2-1}{l-1}q^{l(\xi_1+\xi_2)} + q^{l(\xi_1+\xi_2+1)} \binomq{\eta_1+\eta_2-1}{l} \\
&=  \sum_{\gamma_1+\gamma_2=l-1} \binomq{\eta_1-1}{\gamma_1} \binomq{\eta_2}{\gamma_2} q^{(\eta_1-1-\gamma_1)\gamma_2} q^{\xi_1 l+\xi_2(\gamma_1+1)}+ \sum_{\gamma_1+\gamma_2=l} \binomq{\eta_1-1}{\gamma_1} \binomq{\eta_2}{\gamma_2} q^{(\eta_1-1-\gamma_1)\gamma_2} q^{(\xi_1 +1)l + \xi_2 \gamma_1}.
\end{align*}
Replacing $\gamma_1$ with $\gamma_1-1$ in the first summation, the two sums combine into
\begin{align*}
\sum_{\gamma_1+\gamma_2=l} &\left(  \binomq{\eta_1-1}{\gamma_1-1} \binomq{\eta_2}{\gamma_2} q^{(\eta_1-\gamma_1)\gamma_2}q^{\xi_1 l + \xi_2 \gamma_1} + \binomq{\eta_1-1}{\gamma_1} \binomq{\eta_2}{\gamma_2} q^{(\eta_1-1-\gamma_1)\gamma_2} q^{(\xi_1 + 1) l + \xi_2 \gamma_1}    \right)\\
&=\sum_{\gamma_1+\gamma_2=l} \binomq{\eta_2}{\gamma_2} q^{(\eta_1-1-\gamma_1)\gamma_2} q^{\xi_1 l + \xi_2 \gamma_1}\left( q^{\gamma_2} \binomq{\eta_1-1}{\gamma_1-1} +  q^l\binomq{\eta_1-1}{\gamma_1}  \right) \\
&= \sum_{\gamma_1+\gamma_2=l} q^{(\eta_1-\gamma_1)\gamma_2} q^{\xi_1 l + \xi_2 \gamma_1}\binomq{\eta_1}{\gamma_1} \binomq{\eta_2}{\gamma_2},
\end{align*}
where the last equality used the first identity again.

For the third identity, proceed by induction on $n$. The base case $n=2$ is the second identity. For general $n$,
$$
\sum_{ \vert \gamma \vert =m} q^{\chi_{\eta,\gamma}} q^{\sum_{i=1}^n \xi_i \gamma_{[1,n+1-i]}} \prod_{i=1}^n \binomq{\eta_i}{\gamma_i} = \sum_{l=0}^m \sum_{\substack{\gamma_{[1,n-2]}=m-l \\ \gamma_{n-1}+\gamma_n=l }} q^{\chi_{\eta,\gamma}}  q^{\sum_{i=1}^n \xi_i \gamma_{[1,n+1-i]}} \prod_{i=1}^n \binomq{\eta_i}{\gamma_i} .
$$
For each $\gamma$ in the summand,
$$
\chi_{\eta,\gamma} = \sum_{1 \leq i < j \leq n-2} (\eta_i - \gamma_i)\gamma_j +  \sum_{i=1}^{n-2} (\eta_i - \gamma_i)(\gamma_{n-1}+\gamma_{n}) + (\eta_{n-1}-\gamma_{n-1})\gamma_n .
$$
Therefore, the summand becomes
\begin{multline*}
q^{\sum_{1 \leq i<j\leq n-2}(\eta_i-\gamma_i)\gamma_j}  q^{(\xi_1 +\xi_2)\gamma_{[1,n-2]} } q^{\sum_{i=3}^n \xi_i \gamma_{[1,n+1-i]} } \prod_{i=1}^{n-2} \binomq{\eta_i}{\gamma_i} \\
\times q^{ (\eta_{[1,n-2]}-\gamma_{[1,n-2]})l} q^{(\eta_{n-1}-\gamma_{n-1})\gamma_n}q^{\xi_1 l+ \xi_2 \gamma_{n-1}} \binomq{\eta_{n-1}}{\gamma_{n-1}} \binomq{\eta_n}{\gamma_n}.
\end{multline*}
Evaluating the sum over $\gamma_{n-1}+\gamma_n=l$ yields
\begin{multline*}
 \sum_{l=0}^m \sum_{\substack{\gamma_{[1,n-2]}=m-l }}  q^{\sum_{1 \leq i<j\leq n-2}(\eta_i-\gamma_i)\gamma_j}  q^{(\xi_1 +\xi_2)\gamma_{[1,n-2]} } q^{\sum_{i=3}^n \xi_i \gamma_{[1,n+1-i]} }  \prod_{i=1}^{n-2} \binomq{\eta_i}{\gamma_i}  \\ \times q^{ (\eta_{[1,n-2]}-\gamma_{[1,n-2]})l} \binomq{\eta_{n-1}+\eta_n}{l} q^{l(\xi_1+\xi_2)} .
\end{multline*}
Then setting $\widetilde{\eta}=(\eta_1,\ldots,\eta_{n-2},\eta_{n-1}+\eta_n)$ and $\widetilde{\gamma} = (\gamma_1 , \ldots, \gamma_{n-2},l)$ and $\widetilde{\xi}=(\xi_1+\xi_2,\xi_3,\ldots,\xi_n)$, substitute the equalities
\begin{align*}
\prod_{i=1}^{n-2} \binomq{\eta_i}{\gamma_i} \cdot \binomq{\eta_{n-1}+\eta_n}{l} &= \prod_{i=1}^{n} \binomq{\widetilde{\eta}_i}{\widetilde{\gamma}_i},\\
q^{\sum_{1 \leq i<j\leq n-2}(\eta_i-\gamma_i)\gamma_j} q^{ (\eta_{[1,n-2]}-\gamma_{[1,n-2]})l} &= q^{\chi_{\widetilde{\eta},\widetilde{\gamma}}},\\
 q^{(\xi_1 +\xi_2)\gamma_{[1,n-2]} } q^{\sum_{i=3}^n \xi_i \gamma_{[1,n+1-i]} } q^{l(\xi_1+\xi_2)}  &= q^{ \sum_{i=1}^{n-1}   \widetilde{\xi}_i \widetilde{\gamma}_{[1,n-i]}   } 
\end{align*}
to obtain
$$
 \sum_{l=0}^m \sum_{\substack{\gamma_{[1,n-2]}=m-l }} q^{\chi_{\widetilde{\eta},\widetilde{\gamma}}}q^{ \sum_{i=1}^{n-1}   \widetilde{\xi}_i \widetilde{\gamma}_{[1,n-i]}   } \prod_{i=1}^{n} \binomq{\widetilde{\eta}_i}{\widetilde{\gamma}_i} .
$$
The summation can be written as being over $\widetilde{\gamma} \in \mathbb{Z}^{n-1}$ such that $\vert \widetilde{\gamma}\vert =m$, showing that
$$
\sum_{ \vert \gamma \vert =m} q^{\chi_{\eta,\gamma}} q^{\sum_{i=1}^n \xi_i \gamma_{[1,n+1-i]}} \prod_{i=1}^n \binomq{\eta_i}{\gamma_i} =  \sum_{\vert \widetilde{\gamma} \vert = m} q^{\chi_{\widetilde{\eta},\widetilde{\gamma}}}q^{ \sum_{i=1}^{n-1}   \widetilde{\xi}_i \widetilde{\gamma}_{[1,n-i]}   } \prod_{i=1}^{n} \binomq{\widetilde{\eta}_i}{\widetilde{\gamma}_i} .
$$
Since $\vert \eta\vert = \vert \widetilde{\eta} \vert$ and $\vert \widetilde{\xi} \vert = \vert \xi \vert$, applying the induction hypothesis completes the proof.

\end{proof}

This next identity was previously shown in \cite{C} and again in \cite{B}, which pertains to the $n=1$ case of $\Phi_q$.

\begin{lemma}\label{BC}
Fix $\vert q\vert<1$ and $0 \leq \mu <1$, $0\leq \lambda \leq 1$. Then for all $m,y \in \mathbb{Z}_{\geq 0},$
$$
\sum_{j=0}^m  \Phi_q\left( j \vert m; \lambda , \mu \right) q^{jy}=   \sum_{s=0}^y  \Phi_q\left(s \vert y;\lambda,\mu\right) q^{sm} .
$$
In particular, setting $y=0$ shows that
$$
\sum_{j=0}^m \Phi_q(j\vert m;\lambda,\mu)=1.
$$
\end{lemma}
As an immediate corollary,
\begin{corollary}\label{BC2} 
Fix $\vert q\vert<1$ and $0 \leq \mu <1$. Then for all $m,y \in \mathbb{Z}_{\geq 0},$
$$
\sum_{j=0}^m \mu^j \frac{(q)_{j-1}}{(\mu q^{m-j};q)_j} \binom{m}{j} q^{jy} = \sum_{s=0}^y \mu^s \frac{(q)_{s-1}}{(\mu q^{y-s};q)_s} \binom{y}{s} q^{sm}.
$$
Setting $y=0$ shows that
$$
\sum_{j=0}^m \mu^j \frac{(q)_{j-1}}{(\mu q^{m-j};q)_j} \binom{m}{j} =1.
$$
\end{corollary}

Here is a multi--species generalization:

\begin{proposition}\label{PropIdentity} (1) Fix $\vert q \vert<1$ and $0 \leq \mu < 1, 0 \leq \lambda \leq 1$. Then
$$
\sum_{\gamma \leq \eta} \Phi_q(\gamma \vert \eta;\lambda,\mu) q^{ \sum_{i=1}^{n} \xi_i \gamma_{[1,n+1-i]}} = \sum_{\gamma \leq \xi} \Phi_q(\gamma \vert \xi;\lambda,\mu) q^{\sum_{i=1}^{n}\eta_i \gamma_{[1,n+1-i]}}.
$$

(2) Fix $\vert q \vert<1$ and $0 \leq \mu <1.$ Then
\begin{multline*}
 \sum_{\substack{ \gamma \leq \eta \\ \gamma \neq 0}} q^{\chi_{\eta,\gamma}}\mu^{\vert \gamma\vert -1} \frac{  (q)_{\vert \gamma \vert - 1}  }{ (\mu q^{\vert \eta \vert - \vert \gamma \vert};q)_{\vert \gamma \vert}} \prod_{i=1}^{n} \binomq{\eta_i}{\gamma_i} \left( q^{\sum_{i=1}^{n} \xi_i \gamma_{[1,n+1-i]}}  - 1 \right) \\
 =  \sum_{\substack{\gamma \leq \xi \\ \gamma \neq 0}} q^{\chi_{\xi,\gamma}}\mu^{\vert \gamma\vert -1} \frac{  (q)_{\vert \gamma \vert - 1}  }{ (\mu q^{\vert \xi \vert - \vert \gamma \vert};q)_{\vert \gamma \vert}} \prod_{i=1}^{n} \binomq{\xi_i}{\gamma_i} \left( q^{\sum_{i=1}^{n}\eta_i \gamma_{[1,n+1-i]} } - 1 \right).
\end{multline*}
\end{proposition}
\begin{proof}
(1) Plugging in the expression for $\Phi_q$, the necessary identity is
\begin{multline*}
\sum_{\gamma \leq \eta} q^{\chi_{\eta,\gamma}} \left( \frac{\mu}{\lambda} \right)^{\vert \gamma \vert} \frac{(\lambda;q)_{\vert \gamma \vert} (\tfrac{\mu}{\lambda};q )_{\vert \eta \vert - \vert \gamma \vert} }{ (\mu;q)_{\vert \eta \vert}}  \prod_{i=1}^n \binomq{\eta_i}{\gamma_i} q^{ \sum_{i=1}^{n} \xi_i \gamma_{[1,n+1-i]}} \\
= \sum_{\gamma \leq \xi} q^{\chi_{\xi,\gamma}} \left( \frac{\mu}{\lambda} \right)^{\vert \gamma \vert} \frac{(\lambda;q)_{\vert \gamma \vert} (\tfrac{\mu}{\lambda};q )_{\vert \xi \vert - \vert \gamma \vert} }{ (\mu;q)_{\vert \xi \vert}}  \prod_{i=1}^n \binomq{\xi_i}{\gamma_i} q^{\sum_{i=1}^{n}\eta_i \gamma_{[1,n+1-i]}}.
\end{multline*}
Now split the sum into $\vert \gamma \vert = j$. The left--hand--side is
$$
\sum_{j=0}^{\vert \eta \vert} \left( \frac{\mu}{\lambda} \right)^{\vert \gamma \vert} \frac{ (\lambda;q)_j \left(\tfrac{\mu}{\lambda};q\right)_{\vert \eta\vert-j}  }{(\mu;q)_{\vert \eta\vert}} \sum_{\substack{\gamma \leq \eta \\ \vert \gamma\vert =j}} q^{\chi_{\eta,\gamma}} \prod_{i=1}^{n} \binomq{\eta_i}{\gamma_i}  q^{\sum_{i=1}^{n} \xi_i\gamma_{[1,n+1-i]}}.
$$
By the third identity in Lemma \ref{qBinomialLemma},
$$
\sum_{\substack{\gamma \leq \eta \\ \vert \gamma\vert=j}} q^{\chi_{\eta,\gamma}} \prod_{i=1}^{n} \binomq{\eta_i}{\gamma_i} q^{\sum_{i=1}^{n} \xi_i\gamma_{[1,n+1-i]}} = \binomq{\vert \eta\vert}{j}  q^{m\vert \xi \vert}.
$$
Therefore, after applying an identical argument to the right--hand--side, it remains to show
$$
\sum_{j=0}^{\vert \eta \vert} \left( \frac{\mu}{\lambda} \right)^{j} \frac{ (\lambda;q)_j \left(\tfrac{\mu}{\lambda};q\right)_{\vert \eta\vert-j}  }{(\mu;q)_{\vert \eta\vert}} \binomq{ \vert \eta\vert}{j} q^{j\vert \xi\vert} = \sum_{s=0}^{\vert \xi\vert}  \left( \frac{\mu}{\lambda} \right)^{s} \frac{ (\lambda;q)_s \left(\tfrac{\mu}{\lambda};q\right)_{\vert \xi\vert-s}  }{(\mu;q)_{\vert \xi\vert}} \binomq{\vert \xi\vert}{s} q^{s \vert\eta\vert}.
$$
This follows immediately from Proposition \ref{BC}, finishing the proof.

(2) Because the term in parentheses is equal to zero when $\gamma=0$, the condition that $\gamma \neq 0$ can be removed. Now split the sum into $\vert \gamma \vert = j$. The left--hand--side is
$$
\sum_{j=0}^{ \vert \eta \vert  } \mu^{j-1} \frac{  (q)_{j - 1}  }{ (\mu q^{\vert \eta \vert - j};q)_{ j }} \sum_{\substack{\gamma \leq \eta \\ \vert \gamma\vert=j}} q^{\chi_{\eta,\gamma}} \prod_{i=1}^{n} \binomq{\eta_i}{\gamma_i} \left(q^{\sum_{i=1}^{n} \xi_i\gamma_{[1,n+1-i]}}-1 \right),
$$
By the third identity in Lemma \ref{qBinomialLemma},
$$
\sum_{\substack{\gamma \leq \eta \\ \vert \gamma\vert=j}} q^{\chi_{\eta,\gamma}} \prod_{i=1}^{n} \binomq{\eta_i}{\gamma_i} \left(q^{\sum_{i=1}^{n-1} \xi_i\gamma_{[1,n-i]}}-1 \right) = \binomq{\vert \eta\vert}{j} (q^{m\vert \xi \vert} - 1).
$$
Therefore, after applying an identical argument to the right--hand--side, it remains to show
$$
\sum_{j=0}^{ \vert \eta \vert  } \mu^{j-1} \frac{  (q)_{j - 1}  }{ (\mu q^{\vert \eta \vert - j};q)_{ j }} \binomq{\vert \eta\vert}{j} \left( q^{j \vert \xi\vert} - 1 \right) = \sum_{s=0}^{ \vert \xi \vert} \mu^{s-1} \frac{ (q)_{s-1} }{ (\mu q^{ \vert \xi \vert -s};q)_{s}} \binomq{ \vert \xi \vert}{s} \left( q^{s \vert \eta\vert} -1 \right).
$$
This follows immediately from  Corollary \ref{BC2}, finishing the proof.
\end{proof}

\subsection{The Duality Result}

\begin{proposition}\label{DualityResult}
The $n$--species $q$--Hahn Boson process (in both discrete and continuous time) satisfies space--reversed self--duality with respect to the function
$$
D(\eta,\xi) = \prod_{x = 1 }^L \prod_{i=1}^{n} q^{ \xi_i^x\sum_{y \leq x} \eta_{[1,n+1-i]}^y},
$$
where $\eta$ evolves with total asymmetry to the left and $\xi$ evolves with total asymmetry to the right.

\end{proposition}
\begin{proof}
The duality function can be written equivalently as
$$
\prod_{y=1}^L \prod_{i=1}^{n} q^{  \eta_{i}^y \sum_{x \geq y} \xi_{[1,n+1-i]}^x }.
$$
Write the duality function as 
\begin{align*}
D(\eta,\xi) = \prod_{x=1}^L D_x(\eta,\xi), \quad \text{ where } D_x(\eta,\xi) = \prod_{i=1}^{n} q^{ \xi_i^x\sum_{y \leq x} \eta_{[1,n+1-i]}^y}, \\
D(\eta,\xi) = \prod_{y=1}^L \widetilde{D}_y(\eta,\xi), \quad \text{ where } \widetilde{D}_y(\eta,\xi) = \prod_{i=1}^{n} q^{\eta_i^y \sum_{x \geq y} \xi_{[1,n+1-i]}^x}
\end{align*}
Since the process is a zero range process, the generator of the continuous--time dynamics with evolution to the left can be written as a sum of local generators:
$$
\mathcal{L} = \sum_{x=1}^{L-1} \mathcal{L}_{x+1},
$$
where $\mathcal{L}_y$ is the contribution when particles jump out of lattice site $y$. Similarly the generator of the dynamics with evolution to the right can be written as
$$
\widetilde{\mathcal{L}} = \sum_{x=1}^{L-1} \widetilde{\mathcal{L}}_x.
$$
Since $D_x(\eta,\xi)$ involves counting the number of particles in $\eta$ at sites to the left of $x$ (inclusive), 
\begin{align*}
\text{If } \mathcal{L}_{x+1}(\eta,\sigma) \neq 0, &\text{ then } D_y(\sigma,\xi) = D_y(\eta,\xi) \text{ for } y\neq x,\\
\text{If } \widetilde{\mathcal{L}}_x(\xi,\sigma) \neq 0, &\text{ then } \widetilde{D}_y(\eta,\sigma) = \widetilde{D}_y(\eta,\xi) \text{ for } y\neq x+1.
\end{align*}
Furthermore, if $\mathcal{L}_{x+1}(\eta,\sigma)\neq 0$, then $\sigma = \eta + \gamma^{(x)} - \gamma^{(x+1)}$ for some $\gamma \in \mathbb{Z}^n$. Similarly, if $\widetilde{\mathcal{L}}_x(\xi,\sigma) \neq 0$, then $\sigma = \xi + \gamma^{(x+1)} - \gamma^{(x)}$. In these cases, 
\begin{align*}
D_x\left( \eta + \gamma^{(x)} - \gamma^{(x+1)}, \xi \right)  &= D_x(\eta,\xi) \cdot q^{\sum_{i=1}^{n} \xi_i^x \gamma_{[1,n+1-i]}},  \\
\widetilde{D}_{x+1}\left(\eta, \xi + \gamma^{(x+1)} - \gamma^{(x)}  \right)  &= \widetilde{D}_{x+1}(\eta,\xi) \cdot q^{\sum_{i=1}^{n} \gamma_i \eta^{x+1}_{[1,n+1-i]}} 
\end{align*}

These two statements imply the two equalities
\begin{align*}
\mathcal{L}D (\eta,\xi) &= \sum_{x=1}^{L-1} \sum_{ \sigma } \mathcal{L}_{x+1}(\eta,\sigma)D(\sigma,\xi) \\
&= \sum_{x=1}^{L-1} \sum_{ \sigma } \mathcal{L}_{x+1}(\eta,\sigma)D_x(\sigma,\xi) \prod_{y \neq x} D_y(\eta,\xi) \\
&= D(\eta,\xi) \sum_{x=1}^{L-1} \sum_{ \gamma } \mathcal{L}_{x+1}(\eta,\eta + \gamma^{(x)} - \gamma^{(x+1)}) q^{ \sum_{i=1}^{n} \xi_i^x \gamma_{[1,n+1-i]}   }
\end{align*}
and
\begin{align*}
D \widetilde{\mathcal{L}}^*(\eta,\xi) &= \sum_{x=1}^{L-1} \sum_{\sigma} D(\eta,\sigma) \widetilde{\mathcal{L}}_x(\xi,\sigma) \\
&= \sum_{x=1}^{L-1} \sum_{\sigma}  D_{x+1}(\eta,\sigma) \widetilde{\mathcal{L}}_x(\xi,\sigma) \prod_{y \neq x+1} D_y(\eta,\xi) \\
&= D(\eta,\xi) \sum_{x=1}^{L-1} \sum_{\gamma} \widetilde{\mathcal{L}}_x(\xi, \xi + \gamma^{(x+1)} - \gamma^{(x)}) q^{ \sum_{i=1}^{n} \gamma_i \eta_{[1,n+1-i]}^{x+1}}
\end{align*}
Therefore, since $\sum_{i=1}^{n} \eta_i^{x+1} \gamma_{[1,n+1-i]}  =  \sum_{i=1}^{n} \gamma_i \eta^{x+1}_{[1,n+1-i]}$,  it suffices to show
$$
\sum_{\gamma} \Phi'(\gamma \vert \eta) q^{ \sum_{i=1}^{n} \xi_i \gamma_{[1,n+1-i]}} = \sum_{\gamma} \Phi'(\gamma \vert \xi ) q^{ \sum_{i=1}^{n} \eta_i \gamma_{[1,n+1-i]}}.
$$
Since $ \Phi'(0 \vert \eta) = - \sum_{\gamma \neq 0} \Phi'(0 \vert \gamma)$, the equality that needs to be shown is 
\begin{multline*}
 \sum_{\substack{ \gamma \leq \eta \\ \gamma \neq 0}} q^{\chi_{\eta,\gamma}}\mu^{\vert \gamma\vert -1} \frac{  (q)_{\vert \gamma \vert - 1}  }{ (\mu q^{\vert \eta \vert - \vert \gamma \vert};q)_{\vert \gamma \vert}} \prod_{i=1}^{n} \binomq{\eta_i}{\gamma_i} \left( q^{\sum_{i=1}^{n} \xi_i \gamma_{[1,n+1-i]}}  - 1 \right) \\
 =  \sum_{\substack{\gamma \leq \xi \\ \gamma \neq 0}} q^{\chi_{\xi,\gamma}}\mu^{\vert \gamma\vert -1} \frac{  (q)_{\vert \gamma \vert - 1}  }{ (\mu q^{\vert \xi \vert - \vert \gamma \vert};q)_{\vert \gamma \vert}} \prod_{i=1}^{n} \binomq{\xi_i}{\gamma_i} \left( q^{\sum_{i=1}^{n}\eta_i \gamma_{[1,n+1-i]} } - 1 \right).
 \end{multline*}
But this is just Proposition \ref{PropIdentity}, finishing the proof.

Now turn to the discrete--time $q$--Hahn Boson process. If the evolution is to the right, then the transition probabilities are
$$
P(\eta,\xi) = \prod_{x=1}^{L-1} \Phi( \gamma^x \vert \eta^x),
$$
where
$$
\xi^x = \eta^x - \gamma^x + \gamma^{x-1}, \quad 1 \leq x \leq L
$$
where by convention $\gamma^L = \gamma^0=0$. (If $\gamma^L=\gamma^0 \neq 0$ then the boundary conditions are periodic instead of closed). If the evolution is to the left, then the transition probabilities are
$$
P(\eta,\zeta) = \prod_{x=2}^{L} \Phi( \gamma^x \vert \eta^x),
$$
where
$$
\zeta^x = \eta^x - \gamma^x + \gamma^{x+1}, \quad 1 \leq x \leq L 
$$
and by convention $\gamma^{L+1}= \gamma^1 =0$. 

Letting $\mathfrak{X}$ be the set of particle configurations at one lattice site and $\mathfrak{X}^L$ be the $L$--fold Cartesian product, we have
\begin{align*}
\sum_{\zeta \in \mathfrak{X}^L} P(\eta,\zeta)D(\zeta,\xi) &= \sum_{\gamma \in \{0\} \times \mathfrak{X}^{L-1}} \left( \prod_{x=2}^L \Phi(\gamma^x \vert \eta^x) \right) D(\zeta,\xi) \\
&= \sum_{\gamma \in \{0\} \times \mathfrak{X}^{L-1}} D_1(\zeta,\xi)  \left( \prod_{x=2}^L  \Phi(\gamma^x \vert \eta^x) D_x(\zeta,\xi) \right).
\end{align*}
Now, for each $x$,
$$
D_x( \zeta,\xi) = D_x(\eta,\xi) \prod_{i=1}^{n} q^{\xi_i^x \gamma^{x+1}_{[1,n+1-i]}} ,
$$
so therefore
$$
\sum_{\zeta \in \mathfrak{X}^L} P(\eta,\zeta)D(\zeta,\xi)  = D(\eta,\xi)  \sum_{\gamma \in \{0\} \times \mathfrak{X}^{L-1}} \left( \prod_{i=1}^{n} q^{\xi_i^1 \gamma^{2}_{[1,n+1-i]}} \right) \prod_{x=2}^L  \Phi(\gamma^x \vert \eta^x) \prod_{i=1}^{n} q^{\xi_i^x \gamma^{x+1}_{[1,n+1-i]}}.
$$
Since $\gamma^{L+1}=0$, the product can be re--indexed to show that
$$
\sum_{\zeta \in \mathfrak{X}^L} P(\eta,\zeta)D(\zeta,\xi)  = D(\eta,\xi)  \sum_{\gamma \in \{0\} \times \mathfrak{X}^{L-1}}\prod_{x=2}^L  \Phi(\gamma^x \vert \eta^x) \prod_{i=1}^{n} q^{\xi_i^{x-1} \gamma^{x}_{[1,n+1-i]}}.
$$

By similar reasoning, for the evolution to the right,
\begin{align*}
\sum_{\zeta \in \mathfrak{X}^L} D(\eta,\zeta)P(\xi,\zeta) &= \sum_{\gamma \in  \mathfrak{X}^{L-1} \times \{0\} }  D(\eta,\zeta)    \left( \prod_{y=1}^{L-1}\Phi(\gamma^y \vert \xi^y) \right) \\
&= \sum_{\gamma \in  \mathfrak{X}^{L-1} \times \{0\} } \widetilde{D}_L(\eta,\zeta)  \left( \prod_{y=1}^{L-1} \widetilde{D}_y(\eta,\zeta)  \Phi(\gamma^y \vert \xi^y) \right).
\end{align*}
Again, for each $x$,
$$
\widetilde{D}_y(\eta,\zeta) = \widetilde{D}_y(\eta,\xi) \prod_{i=1}^{n} q^{ \eta^y_{[1,n+1-i]} \gamma^{y-1}_{i} }
$$
so therefore
$$
\sum_{\zeta \in \mathfrak{X}^L} D(\eta,\zeta)P(\xi,\zeta) = D(\eta,\xi) \sum_{\gamma \in  \mathfrak{X}^{L-1} \times \{0\} } \left( \prod_{i=1}^{n} q^{\eta^L_{[1,n+1-i]} \gamma_i^{L-1}}  \right) \prod_{y=1}^{L-1} \Phi(\gamma^y \vert \xi^y)  \prod_{i=1}^{n} q^{\eta^y_{[1,n+1-i]} \gamma_i^{y-1}}.
$$
Since here $\gamma^0=0$, the product can be re--indexed to show that
\begin{align*}
\sum_{\zeta \in \mathfrak{X}^L} D(\eta,\zeta)P(\xi,\zeta) &= D(\eta,\xi) \sum_{\gamma \in  \mathfrak{X}^{L-1} \times \{0\} } \prod_{y=1}^{L-1} \Phi( \gamma^y \vert \xi^y) \prod_{i=1}^{n}q^{\eta^{y+1}_{[1,n-i]} \gamma_i^y} \\
&= D(\eta,\xi) \sum_{\gamma \in  \mathfrak{X}^{L-1} \times \{0\} } \prod_{x=2}^{L} \Phi( \gamma^{x-1} \vert \xi^{x-1}) \prod_{i=1}^{n}q^{\eta^{x}_i \gamma_{[1,n-i]}^{x-1}}
\end{align*}
where the second equality follows from substituting $y=x-1$ and the identity
$$
\sum_{i=1}^{n} \eta_i^{x+1} \gamma_{[1,n+1-i]}  =  \sum_{i=1}^{n} \gamma_i \eta^{x+1}_{[1,n+1-i]}.
$$
Therefore, it suffices to show that
\begin{equation}\label{Sufficient}
\sum_{\gamma \in \mathfrak{X}} \Phi\left(\gamma \vert \eta \right) \prod_{i=1}^{n} q^{\xi_i \gamma_{[1,n+1-i]}} = \sum_{\gamma \in \mathfrak{X}} \Phi\left(\gamma \vert \xi \right) \prod_{i=1}^{n} q^{\eta_i \xi_{[1,n+1-i]}}.
\end{equation}
But this is just Proposition \ref{PropIdentity}.
\end{proof}

\subsection{Lumpability}\label{Lump}

In \cite{K}, it is shown that the $n$--species $\textrm{ASEP}(q,j)$ has the property that the projection onto the first $k$ species is again a $k$--species $\textrm{ASEP}(q,j)$ process. Here we briefly show the same for the $n$--species $q$--Hahn Boson process.

Given $\alpha=(\alpha_1,\ldots,\alpha_n)$, define the projection $\Pi_k^n$ for $1\leq k\leq n$ by
$$
\Pi^n_k\alpha=(\alpha_1,\ldots,\alpha_{k-1},\alpha_k + \cdots + \alpha_n).
$$
Notice that 
\begin{equation}\label{Inductive}
\Pi^n_k = \Pi^{k+1}_k \circ \Pi^{k+2}_{k+1} \circ \cdots \circ \Pi^n_{n-1}.
\end{equation}

\begin{proposition} For any $\alpha=(\alpha_1,\ldots,\alpha_n)$ and any $1\leq k \leq n$,
$$
\sum_{\gamma \in (\Pi_k^n)^{-1}\left(\widetilde{\gamma}\right)} \mathcal{L}(\gamma \vert \alpha;\mu) = \mathcal{L}(\widetilde{\gamma} \vert \Pi^n_k \alpha;\mu).
$$
In particular, the projection of the $n$--species $q$--Hahn Boson process to the first $k$--species is again a $k$--species $q$--Hahn Boson process.
\end{proposition}
\begin{proof}
Since $\left\vert \Pi^n_k \beta \right\vert = \vert \beta\vert$, the proposition does not depend on the expression for $f(\gamma \vert \beta)$, as long as it only depends on $\vert\gamma\vert$ and $\vert \beta\vert$. Therefore, it suffices to show that (setting $\widetilde{\alpha} = \Pi^n_k\alpha$)
$$
\sum_{\gamma \in (\Pi_k^n)^{-1}\left(\widetilde{\gamma}\right)} q^{\chi_{\alpha,\gamma}} \prod_{i=1}^n \binomq{\alpha_i}{\gamma_i} = q^{\chi_{\widetilde{\alpha},\widetilde{\gamma}}} \prod_{i=1}^k \binomq{\widetilde{\alpha}_i}{\widetilde{\gamma}_i} . 
$$
By \eqref{Inductive}, it suffices to consider $n-k=1$. In this case, 
\begin{align*}
\chi_{\alpha,\gamma} &= \sum_{1\leq i<j \leq n-1} (\alpha_i-\gamma_i)\gamma_j  + \sum_{i=1}^{n-1} (\alpha_i-\gamma_i)\gamma_n, \\
\chi_{\widetilde{\alpha},\widetilde{\gamma}} &=  \sum_{1\leq i<j \leq n-2} (\alpha_i-\gamma_i)\gamma_j + \sum_{i=1}^{n-2} (\alpha_i-\gamma_i)(\gamma_{n-1} + \gamma_n)\\
&=  \sum_{1\leq i<j \leq n-1} (\alpha_i-\gamma_i)\gamma_j + \sum_{i=1}^{n-2} (\alpha_i-\gamma_i)\gamma_n,\\
\end{align*}
so that
$$
\chi_{\alpha,\gamma} = \chi_{\widetilde{\alpha},\widetilde{\gamma}}  + (\alpha_{n-1}-\gamma_{n-1})\gamma_n.
$$
Therefore, it suffices to show
$$
\sum_{\gamma_{n-1} + \gamma_n = \widetilde{\gamma}_{n-1}} q^{(\alpha_{n-1}-\gamma_{n-1})\gamma_n} \binomq{\alpha_{n-1}}{\gamma_{n-1}}\binomq{\alpha_n}{\gamma_n} = \binomq{\alpha_{n-1}+\alpha_n}{\widetilde{\gamma}_{n-1}},
$$
which is true by Lemma \ref{qBinomialLemma}.
\end{proof}

\appendix
\section{Explicit examples}
\subsection{Fusion for $l=1,m=2,n=1$}\label{ReshConfirm}
For $l=m=1$ and $n=1$, the $R$--matrix $R(z)$ is given by \eqref{m=1}: 
$$
R(z)=
\left( 
\begin{array}{cccc}
1 & 0 & 0 & 0 \\
0 & \dfrac{q(z-1)}{z-q^2} & \dfrac{z(1-q^2)}{z-q^2} & 0 \\
0 & \dfrac{(1-q^2)}{z-q^2} & \dfrac{q(z-1)}{z-q^2} & 0 \\
0 & 0 & 0 & 1
\end{array}
\right)
$$
with respect to the basis $\vert 10\rangle \otimes \vert 10 \rangle, \vert 10 \rangle \otimes \vert 01 \rangle, \vert 01\rangle \otimes \vert 10 \rangle, \vert 01 \rangle \otimes \vert 01 \rangle$.
From this, 
$$
\check{R}(q^{-2})=\left(
\begin{array}{cccc}
 1 & 0 & 0 & 0 \\
 0 & \frac{q^2}{q^2+1} & \frac{q}{q^2+1} & 0 \\
 0 & \frac{q}{q^2+1} & \frac{1}{q^2+1} & 0 \\
 0 & 0 & 0 & 1 \\
\end{array}.
\right)
$$
It is straightforward to check that $\check{R}(q^{-2})^2=\check{R}(q^{-2})$, so is a projection. Other interesting cases are that $R(1)$ is the usual permutation matrix, and for $q=1$,  $R(z)$ is the identity matrix. 

Take $l=1$ and $m=2$ in the expression \eqref{Fusion} for fusion. The symmetric projection is
$
P^+ = \mathrm{Id}_2 \otimes \check{R}(q^{-2}),
$
where $\mathrm{Id}_2$ is the $2\times 2$ identity matrix and $\check{R}(q^{-2})$ is the matrix from above. The other terms are
$$
R_{12}(zq^{-1}) = R(zq^{-1}) \otimes \mathrm{Id}_2
$$ 
and 
$$
R_{1'2}(zq) = \left(
\begin{array}{cccccccc}
 1 & 0 & 0 & 0 & 0 & 0 & 0 & 0 \\
 0 & \frac{1-q z}{q-z} & 0 & 0 & \frac{(q^2-1)z}{q^2-q z} & 0 & 0 & 0 \\
 0 & 0 & 1 & 0 & 0 & 0 & 0 & 0 \\
 0 & 0 & 0 & \frac{1-q z}{q-z} & 0 & 0 & \frac{(q^2-1)z}{q^2-q z} & 0 \\
 0 & \frac{q^2-1}{q-z} & 0 & 0 & \frac{1-q z}{q-z} & 0 & 0 & 0 \\
 0 & 0 & 0 & 0 & 0 & 1 & 0 & 0 \\
 0 & 0 & 0 & \frac{q^2-1 }{q-z} & 0 & 0 & \frac{1-q z}{q-z} & 0 \\
 0 & 0 & 0 & 0 & 0 & 0 & 0 & 1 \\
\end{array}
\right).
$$
Multiplying the $8\times 8$ matrices $P^+R_{13}(zq)R_{12}(zq^{-1}) P^+$  yields
$$
\left(
\begin{array}{cccccccc}
 1 & 0 & 0 & 0 & 0 & 0 & 0 & 0 \\
 0 & \frac{q^3 (q-z)}{\left(q^2+1\right) \left(q^3-z\right)} & \frac{q^2
   (q-z)}{\left(q^2+1\right) \left(q^3-z\right)} & 0 & \frac{q(q^2-1)z}{q^3-z} & 0 & 0 & 0
   \\
 0 & \frac{q^2 (q-z)}{\left(q^2+1\right) \left(q^3-z\right)} & \frac{q
   (q-z)}{\left(q^2+1\right) \left(q^3-z\right)} & 0 & \frac{z
   \left(q^2-1\right)}{q^3-z} & 0 & 0 & 0 \\
 0 & 0 & 0 & \frac{q (1-q z)}{q^3-z} & 0 & \frac{zq(q^2-1)}{q^3-z} & \frac{z
   \left(q^2-1\right)}{q^3-z} & 0 \\
 0 & \frac{\left(q^2-1\right) q}{q^3-z} & \frac{ q^2-1}{q^3-z} & 0 &
   \frac{q (1-q z)}{q^3-z} & 0 & 0 & 0 \\
 0 & 0 & 0 & \frac{\left(q^2-1\right) q}{q^3-z} & 0 & \frac{q^3 (q-z)}{\left(q^2+1\right)
   \left(q^3-z\right)} & \frac{q^2 (q-z)}{\left(q^2+1\right) \left(q^3-z\right)} & 0 \\
 0 & 0 & 0 & \frac{ q^2-1}{q^3-z} & 0 & \frac{q^2
   (q-z)}{\left(q^2+1\right) \left(q^3-z\right)} & \frac{q (q-z)}{\left(q^2+1\right)
   \left(q^3-z\right)} & 0 \\
 0 & 0 & 0 & 0 & 0 & 0 & 0 & 1 \\
\end{array}
\right).
$$
Note that this is singular at $z=q^3$ but not at $z=q$. One can check that 
\begin{align*}
P^+R_{13}(zq^{-1})R_{12}(zq) P^+&= P^+R_{13}(zq^{-1})R_{12}(zq) \\
=P^+R_{13}(zq)R_{12}(zq^{-1}) P^+&=  R_{13}(zq)R_{12}(zq^{-1}) P^+,
\end{align*}
as predicted by \eqref{Stronger}.

The above $8\times 8$ matrix is consistent with \eqref{l=1} at $m=2$. For instance, 
$$
R(z)\left( \vert 10\rangle \otimes  \vert 11\rangle\right) =   \frac{q(q-z)}{q^3-z}\vert 10\rangle \otimes  \vert 11\rangle +  \frac{q^2-1}{q^3-z}\vert 01\rangle \otimes  \vert 20\rangle.
$$
Since
\begin{align*}
\vert 10\rangle \otimes  \vert 11\rangle &= \vert 10 \rangle \otimes \check{R}(q^{-2})\left( \vert 01 \rangle \otimes \vert 10 \rangle\right) = (q^2+1)^{-1} \left( q \vert 10 \rangle \otimes \vert 10 \rangle  \otimes \vert 01 \rangle +  \vert 10 \rangle \otimes \vert 01 \rangle  \otimes \vert 10 \rangle \right), \\
\vert 01\rangle \otimes  \vert 20\rangle &= \vert 01\rangle \otimes  \vert 10\rangle \otimes \vert 10\rangle,
\end{align*}
then the $8 \times 8$ matrix applied to $\vert 10 \rangle \otimes \vert 01 \rangle \otimes \vert 10 \rangle$ equals
\begin{multline*}
\frac{q^2 (q-z)}{\left(q^2+1\right) \left(q^3-z\right)} \vert 10 \rangle \otimes \vert 10 \rangle \otimes \vert 01 \rangle + \frac{q (q-z)}{\left(q^2+1\right) \left(q^3-z\right)} \vert 10 \rangle \otimes \vert 01 \rangle \otimes \vert 10 \rangle + \frac{q^2-1}{q^3-z} \vert 01 \rangle \otimes \vert 10 \rangle \otimes \vert 10\rangle \\
=   \frac{q(q-z)}{q^3-z}\vert 10\rangle \otimes  \vert 11\rangle +  \frac{q^2-1}{q^3-z}\vert 01\rangle \otimes  \vert 20\rangle.
\end{multline*}

Note that one can also check that
$$
(z-q^2)( \check{R}(q^{-2}) \otimes \mathrm{Id}_2) ( \mathrm{Id}_2 \otimes R(z) ) ( \check{R}(q^{-2}) \otimes \mathrm{Id}_2) 
$$
has rank $2$ at $z\rightarrow q^2$, and that the resulting matrix satisfies
$$
Q^2 = \frac{1-q^6}{1+q^2}Q.
$$

\subsection{$S(z)$ for $l=m=2,n=1$}
For $l=m=2$ and $n=1$ the action of $R(z)$ is given by (replacing $z=-\alpha q^4$)
\begin{align*}
&R(z)(|02\rangle \ot |02\rangle)=|02\rangle \ot |02\rangle, 
\quad R(z)(|20\rangle \ot |20\rangle)=|20\rangle \ot |20\rangle,\\
&R(z)(|02\rangle \ot |11\rangle)=\frac{q^{-2}(1+\alpha q^4)}{1 + \alpha}|02\rangle \ot |11\rangle+\frac{(1-q^4)\alpha}{1 + \alpha}|11\rangle \ot |02\rangle, \\
&R(z)(|02\rangle \ot |20\rangle)=q^{-4}\frac{(1+\alpha q^4)(1+\alpha q^6)}{(1 + \alpha q^2)(1 + \alpha)}|02\rangle \ot |20\rangle+\frac{q^{-1}(1+q^2)(1-q^4)(1+\alpha q^4)\alpha}{(1 + \alpha q^2)(1 + \alpha)}|11\rangle \ot |11\rangle \\
& \ \ \ \ \ \ \ \ \ \ \ \ \ \ \ \ \ \ \ \ \ \ \ \ \   +\frac{(1-q^2)(1-q^4){\color{black}q^2}\alpha^2}{(1 + \alpha q^2)(1 + \alpha)}|20\rangle \ot |02\rangle, \\
&R(z)(|11\rangle \ot |20\rangle)=\frac{\alpha(1-q^4)}{1 + \alpha}|20\rangle \ot |11\rangle+\frac{q^{-2}(1+\alpha q^4)}{1 + \alpha}|11\rangle \ot |20\rangle, \\
&R(z)(|11\rangle \ot |11\rangle)= -\frac{q^{-5}(1-q^2)(1+\alpha q^4)}{(1 + \alpha q^2)(1 + \alpha)}|02\rangle \ot |20\rangle+\frac{q^6z-2q^4z+q^4+q^2z^2-2q^2z+z}{(q^2(1 + \alpha q^2))(q^4(1 + \alpha))}|11\rangle \ot |11\rangle \\
& \ \ \ \ \ \ \ \ \ \ \ \ \ \ \ \ \ \ \ \ \ \ \ \ \   +\frac{q^{-1}(1-q^2)(1+\alpha q^4)\alpha}{(1 + \alpha q^2)(1 + \alpha)}|20\rangle \ot |02\rangle, \\
&R(z)(|11\rangle \ot |02\rangle)=\frac{q^2(1+\alpha q^4)}{q^4(1 + \alpha)}|11\rangle \ot |02\rangle-\frac{1-q^4}{q^4(1 + \alpha)}|02\rangle \ot |11\rangle, \\
&R(z)(|20\rangle \ot |02\rangle)=\frac{q^{-4}(1+\alpha q^4)(1+ \alpha q^6)}{(1 + \alpha q^2)(1 + \alpha)}|20\rangle \ot |02\rangle-\frac{q^{-5}(1+q^2)(1-q^4)(1+\alpha q^4)}{(1 + \alpha q^2)(1 + \alpha)}|11\rangle \ot |11\rangle \\
& \ \ \ \ \ \ \ \ \ \ \ \ \ \ \ \ \ \ \ \ \ \ \ \ \  + q^{-6}\frac{(1-q^2)(1-q^4)}{(1 + \alpha q^2)(1 + \alpha)}|02\rangle \ot |20\rangle, \\
&R(z)(|20\rangle \ot |11\rangle)=
\frac{q^{-2}(1+\alpha q^4)}{1 + \alpha}|20\rangle \ot |11\rangle-\frac{1-q^4}{q^4(1 + \alpha)}|11\rangle \ot |20\rangle,
\end{align*}
where $|\alpha\rangle$ with $\alpha=(\alpha_1,\alpha_2)$ is 
denoted by $|\alpha_1\alpha_2\rangle$.
Now since
$$
\Ga(|\alpha_1 \alpha_2\rangle \ot |\beta_1 \beta_2\rangle) = q^{-\alpha_1\beta_2}|\alpha_1 \alpha_2\rangle \ot |\beta_1 \beta_2\rangle, \quad \widetilde{\Ga}^{-1}(|\alpha_1 \alpha_2\rangle \ot |\beta_1 \beta_2\rangle) = q^{\alpha_2\beta_1}|\alpha_1 \alpha_2\rangle \ot |\beta_1 \beta_2\rangle,
$$
then the action of $S(z)$ is 
\begin{align*}
&S(z)(|02\rangle \ot |02\rangle)=|02\rangle \ot |02\rangle, 
\quad S(z)(|20\rangle \ot |20\rangle)=|20\rangle \ot |20\rangle,\\
&S(z)(|02\rangle \ot |11\rangle)=\frac{1+\alpha q^4}{1 + \alpha}|02\rangle \ot |11\rangle+\frac{(1-q^4)\alpha}{1 + \alpha}|11\rangle \ot |02\rangle, \\
&S(z)(|02\rangle \ot |20\rangle)=\frac{(1+\alpha q^4)(1+\alpha q^6)}{(1 + \alpha q^2)(1 + \alpha)}|02\rangle \ot |20\rangle+\frac{(1+q^2)(1-q^4)(1+\alpha q^4)\alpha}{(1 + \alpha q^2)(1 + \alpha)}|11\rangle \ot |11\rangle \\
& \ \ \ \ \ \ \ \ \ \ \ \ \ \ \ \ \ \ \ \ \ \ \ \ \   +\frac{(q^2-q^4)(1-q^4)\alpha^2}{(1 + \alpha q^2)(1 + \alpha)}|20\rangle \ot |02\rangle, \\
&S(z)(|11\rangle \ot |20\rangle)=\frac{\alpha(1-q^4)}{1 + \alpha}|20\rangle \ot |11\rangle+\frac{1+\alpha q^4}{1 + \alpha}|11\rangle \ot |20\rangle, \\
&S(z)(|11\rangle \ot |11\rangle)= \frac{(1-q^{-2})(1+\alpha q^4)}{(1 + \alpha q^2)(1 + \alpha)}|02\rangle \ot |20\rangle+\frac{q^6z-2q^4z+q^4+q^2z^2-2q^2z+z}{q^5(1 + \alpha q^2)(1 + \alpha)}|11\rangle \ot |11\rangle \\
& \ \ \ \ \ \ \ \ \ \ \ \ \ \ \ \ \ \ \ \ \ \ \ \ \   +\frac{(1-q^2)(q^{-2}+\alpha q^2)\alpha}{(1 + \alpha q^2)(1 + \alpha)}|20\rangle \ot |02\rangle, \\
&S(z)(|11\rangle \ot |02\rangle)=\frac{1+\alpha q^4}{q^4(1 + \alpha)}|11\rangle \ot |02\rangle +\frac{1-q^{-4}}{1 + \alpha}|02\rangle \ot |11\rangle, \\
&S(z)(|20\rangle \ot |02\rangle)=\frac{(q^{-4}+\alpha )(q^{-4}+ \alpha q^2)}{(1 + \alpha q^2)(1 + \alpha)}|20\rangle \ot |02\rangle+\frac{(1+q^2)(1-q^{-4})(q^{-4}+\alpha )}{(1 + \alpha q^2)(1 + \alpha)}|11\rangle \ot |11\rangle \\
& \ \ \ \ \ \ \ \ \ \ \ \ \ \ \ \ \ \ \ \ \ \ \ \ \  + \frac{ (1-q^{-2})(1-q^{-4})}{(1 + \alpha q^2)(1 + \alpha)}|02\rangle \ot |20\rangle, \\
&S(z)(|20\rangle \ot |11\rangle)=
\frac{q^{-4}+\alpha }{1 + \alpha}|20\rangle \ot |11\rangle+\frac{1-q^{-4}}{1 + \alpha}|11\rangle \ot |20\rangle,
\end{align*}
Up to $q\rightarrow q^{1/2}$, with $\nu=q^{-2}$, these are the same weights as in Appendix B.2 of \cite{CP}.

\bibliographystyle{plain}

\end{document}